\documentclass{article}

\usepackage{amsthm}
\usepackage{mathrsfs}
\usepackage{amssymb}
\usepackage{bbm}
\usepackage{bm}
\usepackage{mathalfa}
\usepackage{mathtools}
\usepackage{euscript}
\usepackage{amsmath,amsfonts,graphicx}
\usepackage{cancel}
\usepackage{stmaryrd}
\usepackage{geometry}
\usepackage{hyperref}
\usepackage{enumerate}

\newcommand{\TN}{\mathbb{T}^2_N}
\newcommand{\TT}{\mathbb{T}^2}
\newcommand{\SSS}{\mathbb{S}}
\newcommand{\LL}{\mathcal{L}}
\newcommand{\LLWA}{\mathcal{L}^{\text{WA}}}
\newcommand{\LLR}{\mathcal{L}^{\text{R}}}
\newcommand{\LLD}{\mathcal{L}^{\text{D}}}
\newcommand{\LLS}{\mathcal{L}}

\newcommand{\del}{\bm{\delta}}
\newcommand{\dd}{\mathrm{d}}
\newcommand{\etah}{\hat{\eta}}
\newcommand{\etat}{\widetilde{\eta}}
\newcommand{\zerot}{\widetilde{0}}

\newcommand{\etapo}{\eta^{p,\omega}}
\newcommand{\etaao}{\eta^{a,\omega}}
\newcommand{\etapob}{\eta^{p, \hat{\omega}^p}}
\newcommand{\etaaob}{\eta^{a, \hat{\omega}^a}}

\newcommand{\jop}{j^{p,\omega}}
\newcommand{\joa}{j^{a,\omega}}
\newcommand{\jpoh}{j^{p, \hat{\omega}^p}}
\newcommand{\jaoh}{j^{a, \hat{\omega}^a}}
\newcommand{\ro}{r^{\omega}}

\newcommand{\SK}[1][1]{\Sigma_{{#1}}^{\hat{K}}}

\newcommand{\Eu}[1][1]{
    \EuScript{{#1}}
}
\newcommand{\mf}[1][1]{
    \mathfrak{{#1}}
}

\newcommand{\E}{\mathbb{E}}

\newcommand{\PP}{\mathbb{P}}
\newcommand{\PN}{\mathbb{P}_{\mu^N}}
\newcommand{\EN}{\mathbb{E}_{\mu^N}}

\newcommand{\iip}[1][1]{\ll\!\!  {#1}  \!\!\gg}

\newcommand{\oh}{\hat{\omega}}
\newcommand{\ohw}{\widehat{\omega}}
\newcommand{\oha}{\hat{\omega}^a}
\newcommand{\ohp}{\hat{\omega}^p}
\newcommand{\ob}{\overline{\omega}}
\newcommand{\oba}{\overline{\omega}^a}
\newcommand{\obp}{\overline{\omega}^p}

\newcommand{\ah}{\hat{\alpha}}
\newcommand{\ba}{\bm{\alpha}}
\newcommand{\bah}{\bm{\hat{\alpha}}}
\newcommand{\Kh}{\widehat{K}}
\newcommand{\Kt}{\widetilde{\mathbb{K}}}
\newcommand{\To}{\mathcal{T}_0^\omega}

\DeclareMathOperator{\var}{Var}
\DeclareMathOperator{\Span}{Span}

\DeclareMathOperator{\grad}{grad}

\DeclareMathOperator{\supp}{supp}
\DeclareMathOperator{\Ker}{Ker}

\newtheorem{theorem}{Theorem}[section]
\newtheorem{lemma}[theorem]{Lemma}
\newtheorem{proposition}[theorem]{Proposition}

\newtheorem{corollary}[theorem]{Corollary}
\newtheorem{definition}[theorem]{Definition}
\newtheorem{remark}[theorem]{Remark}
\title{Hydrodynamic limit for an active-passive exclusion process}
\author{Deyue Li\thanks{Shanghai Center for Mathematical Sciences; Fudan University, Shanghai, China 200433 ({lidy21@m.fudan.edu.cn})} 
}
\begin{document}
\allowdisplaybreaks
\maketitle

\begin{abstract}
    The collective non-equilibrium dynamics of multi-component
 mixtures of interacting active (self-propelled) and passive (diffusive) particles 
 have garnered  great interest in physics community.  
 However, the mathematical understanding of these systems remains partial. 
 In this work,  we consider a lattice gas model of active-passive particle mixtures with
 exclusion, where the self-propulsion orientations of active particles 
 undergo Brownian motion on a torus. We derive the hydrodynamic equations
 governing the particle densities. 
 Due to the presence of two types of particles with continuous-valued orientations, 
 further generalizations of non-gradient decomposition and spectral gap estimation developed for 
 pure active case are necessitated,  
 which entails novel challenges and new proofs.


\end{abstract}

\section{Introduction}
Active matter systems, composed of self-propelled agents that extract energy from 
their environment, are inherently out of equilibrium. 
These systems exhibit a rich diversity of collective behaviors arising 
from agent interactions, such as  clustering of 
chemically interacting particles~\cite{agudo2019active}, 
flocking of birds~\cite{VICSEK201271}, 
self-organization of quorum sensing bacteria~\cite{ridgway2023motility}, 
and self-assembly of active materials from
catalytic colloids~\cite{doi:10.1126/science.1230020}. 
Despite their diversity, these systems display 
spontaneous condensate even without attractive or aligning 
interactions~\cite{annurev-conmatphys-031214-014710}.  
This phenomenon, known as 
\textit{Motility Induced Phase Separation} (MIPS) in the physics community, 
results from 
interplay between particle crowding and velocity reduction: 
as particles aggregate, their motion slows, further promoting aggregation. 
Recently, there has been great practical interest in 
the phase-separation behavior of active systems where the micro-swimmers are mixed 
with passive particles. 
Examples include active-passive segregation between rodlike  particles~\cite{mccandlish2012spontaneous}, 
crystallization of glassy systems comprising  mixtures of active and passive hard spheres~\cite{ni2014crystallizing}, 
and  collective dynamics in phase-separated mixtures of active and passive Brownian particles~\cite{wysocki2016propagating}. 

Active matter systems also have drawn great attention from the mathematical community. 
Many mathematical progress has been achieved under  
mean-field assumption~\cite{frouvelle2012dynamics,frouvelle2012continuum,bruna2022phase}. 
While the mean-field assumption simplifies several complexities, 
and allows  explicit 
derivation of macroscopic equations,   
these derivations depend on the BBGKY hierarchy and molecular chaos assumption,  
which generally fail to fully capture 
the microscopic mechanism underlying the emergence of MIPS~\cite{bruna2022phase}. 
A more rigorous and systematic approach is utilizing 
hydrodynamic limit theory to derive hydrodynamic equations from microscopic dynamics. 
One of the most fundamental and widely studied class of models 
amenable to hydrodynamic limit theory is lattice gas models~\cite{kipnis1999scaling}, 
where particles hop randomly on a discrete
lattice according to some local rules. 
By taking the lattice spacing to zero and applying appropriate temporal and spatial rescaling, 
one can  rigorously derive   
continuum equations that describe the large-scale behaviour of active systems 
under relatively general assumptions. 

The application of hydrodynamic limit theory to various classes 
of active lattice gas (ALG) models has achieved remarkable success~\cite{Erignoux21,mason2023exact,bruna2022phase}, 
providing valuable insights into the physics of active matter~\cite{kourbane2018exact}. 
However, while most of the existing studies  focus on the hydrodynamic limit of 
pure active lattice models, 
mathematical progress for active-passive mixtures remains  partial. 
To the best of our knowledge, the first result on hydrodynamic limit of 
active-passive lattice gas (APLG) model, 
restricted to only two self-propelled directions, 
was studied in~\cite{mason2024dynamical}. 

In this work, we remove such restriction on directions and 
consider a more general APLG model  
where active and passive particles interact via exclusion, 
and the continuous-valued  orientations of active particles undergo Brownian motion on a torus. 
We extend the proof of the hydrodynamic limit for pure active exclusion process on lattice~\cite{Erignoux21,mason2023exact}  
to the active-passive mixed exclusion process.  
This extension introduces additional degrees of freedom
into the set of measure-valued parameters characterizing local equilibrium 
and necessitates a broader working function class for the non-gradient method. 
Consequently,  significant adaptations to the proof of the hydrodynamic 
limit in pure active case are required to address these difficulties.  

\subsection{Previous works  and our contributions}
Starting from the seminal work of Vicsek~\cite{vicsek1995novel}, 
active matter has become a focal point of interest for both the physics and 
mathematics communities.    
Active lattice gas (ALG) models were first introduced through the 
active Ising model in~\cite{solon2013revisiting}, 
where active self-propulsion was modeled by biased diffusion 
and there are no exclusion interactions between particles. 
Hard-core interactions, one of the key microscopic mechanisms 
resulting the emergence of MIPS, were first incorporated into the 
ALG framework by Erignoux in his PhD thesis~\cite{erignoux2016limite} 
and later refined in~\cite{Erignoux21}, 
where the  exclusion rules were used to model the hard-core interactions.  
By incorporating the exclusion rule into active Ising model, valuable insights 
into the physics of active matter were successfully provided in~\cite{kourbane2018exact}.

The incorporation of the exclusion rule turns the system from gradient type to non-gradient type,   
significantly increasing the mathematical complexity. 
The first proof of a non-gradient hydrodynamic limit was established by Varadhan in~\cite{key1354152m}, 
and Quastel~\cite{quastel1992diffusion}, where they extended the entropy method proposed in~\cite{guo1988nonlinear} to non-gradient systems.   
In his seminal work~\cite{quastel1992diffusion}, 
Quastel considered two-type mixed exclusion process 
and constructed  a suitable working function class on which the suitable spectral gap shrink rate can be reached. 
This result was later  partially extended to the weakly asymmetric  exclusion process  
to derive 
Navier-Stokes equations via lattice gas~\cite{esposito1994diffusive}. 
The weak asymmetry property was later broadly used to model self-propulsion in active matter systems.   
A further mathematical breakthrough was achieved in~\cite{Erignoux21}, 
where the author extended the multi-typed symmetric exclusion process studied by Quastel~\cite{quastel1992diffusion} 
to an active (weakly asymmetric) exclusion process with continuous angular orientations.   
If each orientation of a particle is treated as a 
particle type in the sense of~\cite{quastel1992diffusion}, 
the model in~\cite{Erignoux21} involves uncountably many types due to the continuous-valued orientations.  
Consequently, significant adaptations were made in~\cite{Erignoux21}, 
including the introduction of measure-valued parameters for local equilibrium and 
the construction of a new working function class for non-gradient method.  
Based on this mathematical breakthrough, 
the authors of~\cite{mason2023exact}  studied an active exclusion process where self-propulsion orientations undergo diffusive dynamics, 
instead of the Glauber dynamics considered in~\cite{Erignoux21}. 
However, all of the above mentioned works focus on  purely active system.  
Regarding active-passive lattice gas (APLG) models,   
the only work to date, to the best of our knowledge, is~\cite{mason2024dynamical}, 
where the authors considered  an APLG model restricted to only two orientations (left and right).  
This restriction limits the model's applicability to many realistic physical systems, 
making its generalization both physically and mathematically meaningful. 

In this work, we consider a more general APLG model  
where active and passive particles interact via exclusion, 
and the continuous valued  orientations of active particles undergo Brownian motion on a torus. 
Due to the presence of mixed active and passive particles, 
which introduces additional degrees of freedom, the  measure-valued parameters and 
the working function class for non-gradient method introduced in~\cite{Erignoux21} 
are no longer sufficient to establish the hydrodynamic limit in our case.  
Drawing  insights from~\cite{quastel1992diffusion,mason2024dynamical,mason2023exact}, 
we further generalize the results of Erignoux~\cite{Erignoux21} by introducing  
new parameters of local equilibrium and new function class for the non-gradient method.  
This generalization  requires novel proofs for non-gradient decomposition and spectral gap estimation. 

\subsection{Structure of the article}
In Section~\ref{sec:notations}, we introduce our model and present the main result 
concerning the hydrodynamic limit. Section~\ref{sec:proof} is dedicated to the proof 
of this main result. In Section~\ref{sec:replaceV}, we establish the non-gradient 
decomposition and the corresponding replacement lemmas, which are essential for the 
proof of the main result. Finally, in Section~\ref{sec:5}, we prove the spectral gap 
estimate and the closed-form decomposition, both of which are crucial for the non-gradient 
method.

\section{Notations and results}\label{sec:notations}

Throughout this article, let 
$\mathbb{T}^2=[0,1)^2$ denote the 2-dimensional torus, which models 
the macroscopic spatial domain, and let
$\SSS=[0, 2\pi)$ denote the 1-dimensional torus, which models  the angular domain. 
For  $N\in \mathbb{N}$, let  $\TN = \{ 0, 1, \ldots, N-1 \}^2$ denote the 2-dimensional lattice, 
equipped with Manhattan norm $|\cdot|$, 
representing the microscopic spatial domain.  
The lattice $\TN$ is naturally embedded into  $\mathbb{T}^2$ 
by corresponding each $x \in \TN$ to $x/N\in \mathbb{T}^2$.
Specially,  $\TT_{\infty}=\mathbb{Z}^2$ represents an infinite lattice. 

Let denote $C^{1,2,2}([0,T]\times \mathbb{T}^2\times \SSS)$  
the set of continuous functions that are 
$C^1$ in time and $C^2$ in both spacial and angular domain. 
For a measurable space $E$, 
let $\mathcal{M}(E)$ denote 
the set of non-negative measures, and 
$\mathcal{P}(E)$ denote the set of 
probability measures, both equipped with the weak topology. 
We also let $\mathcal{M}_1(E)$ denote 
the set of non-negative measures with total mass less than or equal to $1$. 
The space $L^p(E,\mu)$ consists of $p$-th order integrable functions on $E$ 
with respect to the measure $\mu$. 
For $f\in L^1(E,\mu)$,  
we use $<\!\!f,\,\mu\!\!>$ or $<\!\!f\!\!>_\mu$ to denote the integral $\int_{E}f\dd \mu$.  
Similarly, for $f,g \in L^2(E,\mu)$, the $L^2$ inner product is denoted by $<\!\!f,\,g\!\!>_{\mu}=\int_{E}fg\dd \mu$. 
For a polish space $E$, let $D([0,T],E)$ represent the space of 
right-continuous, left-limit (RCLL) $E$-valued processes  
equipped with the Skorohod metric.  

The translation operator $\tau_x$, which shifts a function by $x$, is defined by  
$\tau_x f(y)=f(y+x)$.  
The canonical basis of $\mathbb{R}^2$ is denoted by $\{e_1,e_2\}$. 
For any site $x$ on the lattice $\TN$, 
and any integer $l$, we denote by $B_l(x)=\{y\in \TN:|y-x|\leq l\}$  the closed box of side length $2l + 1$ around $x$. 
In particular, let $B_l=B_l(0)$, the box centered at the origin. 

\subsection{Particle configurations}
The model involves two types of particles: active and passive. 
The \textit{Exclusion rule} ensures that at most one particle can occupy each site, 
regardless of its type. 
The \textit{occupation configuration} at site $x\in \TN$ is denoted by 
$\etat_x =(\eta_x^a,\eta^p_x)$ where $\eta^\sigma_x \in \{0,1\}$ for $\sigma\in \{ \text{a}, \text{p} \}$,
indicates the number of particles of type $\sigma$ at site $x$.   
Let $\eta_x = \eta_x^a+\eta_x^p$ denote the total number of particles at site $x$ regardless of type. 
To simplify notation, we write: $1^a\coloneqq (1,0)$, $1^p\coloneqq (0,1)$, and $\zerot\coloneqq (0,0)$, 
which represent the three possible occupation configuration of a site. 
For any occupied site $x\in \TN$, i.e. $\eta_x=1$, 
we associate an angle $\theta_x \in \SSS$ to indicate the direction of the particle 
 occupying that site. If the site $x$ is unoccupied, i.e. $\eta_x=0$, 
the angle is set to $\theta_x=0$ by default. 
The \textit{complete configuration} (or simply \textit{configuration}) is denoted by 
$\etah=(\etah_x)_{x\in \TN}$, where the pair $\etah_x=(\etat_x,\theta_x)$ can 
fully determines the scenario at site $x$. 
The \textit{configurations space} is given by 
\begin{equation}\label{eq:config_space}
    \Sigma_N = \Bigg\{ \etah=(\etat_x, \theta_x)_{x\in \TN}\in 
    \big( \{1^a,1^p,\zerot  \} \times \SSS \big)^{\TN}
    \Big| \theta_x=0 \text{ if } \etat_x=\zerot  
    \Bigg\}
\end{equation}
which will serve as the state space of our stochastic process. 
In particular, $\Sigma_\infty$ denotes 
the set of infinite configurations, where $\TN$ in~\eqref{eq:config_space} 
is replaced by $\mathbb{Z}^2$. 

We will  refer to  any function $f:\Sigma_\infty \to \mathbb{R}$ that 
depends on the configuration 
through only a finite set of vertices as \textit{local function}. 
For such local function $f$, by slightly abusing notation, we denote this finite set of vertices by $\supp(f)\subset \mathbb{Z}^2$.  
Note that  any local function  admits a natural image as a function on $\Sigma_N$, for  any sufficiently large $N$.  
A local function $f$ that is also $C^2$ with respect to each $\theta_x$  for every $x\in \supp(f)$ is called a \textit{cylinder function}.  
The set of all cylinder functions will be denoted by $\mathcal{C}$.

\subsection{Active-passive exclusion process}
\label{subsec:def_of_MP}

The dynamics of active-passive exclusion process consists of two parts: 
random displacement and rotational diffusion. 
The random displacement part governs the random walk of particles 
and may differ between active and passive particles, 
while the rotational diffusion part governs the angle update process 
and applies equally to both particle types. 

We begin by describing the random displacement dynamics. 
All particles perform 
a continuous time nearest-neighbor random walk with exclusion simultaneously.  
A particle at site $x$ waits an independent, exponentially distributed random time 
before selecting a nearest-neighbor site $x + z$ with $|z|=1$ according to a specified jump law 
that may depend on the particle type. 
The jump is executed 
if the destination site is unoccupied;  
otherwise, the jump is canceled, and  the particle
waits for another new exponential random time before attempting a new move. 

The jump laws and the parameter of the exponential distribution
are chosen  such that  the active and passive jump rates from site  $x$ to a nearest-neighbor 
site $x+z$ with $|z|=1$ are given by 
\begin{equation}\label{eq:jump_rate}
    R^a(x,x+z)=D_T +\frac{v_0}{2N}z^\top \mathbf{e}(\theta_x), \quad \text{ and } \quad R^p(x,x+z)=D_T,
\end{equation}
respectively. 
In~\eqref{eq:jump_rate}, $N$ is sufficiently large to ensure that all quantities remain non-negative.   
The vector $\mathbf{e}(\theta)=[\cos(\theta), \sin(\theta)]^\top$ represents the self-proposed direction,  
 $D_T>0$ is the \textit{translational diffusion constant}, 
and $v_0 >0$ is the \textit{self-propulsion speed}. 
Note the mean velocity of an active particle at site $x$ is 
\[ 
\sum_{|z|=1}zR^a(x,x+z)= \frac{v_0}{N} \mathbf{e}(\theta_x).
\]
The generator $\LLD_N$ of the random displacement dynamics is given by 
\[ 
\LLD_N = D_T \LL_N + \frac{1}{N}\LLWA_N, 
\]
where $\LL_N$ is the generator of 
the nearest-neighbor simple symmetric exclusion process (SSEP) defined in~\eqref{eq:LLS_N}, 
and $\LLWA_N$ is the generator contribute to the weakly asymmetric motion defined in~\eqref{eq:LLWA_N}. 

For a cylinder function $f\in \mathcal{C}$ on  $\Sigma_N$, the SSEP generator $\LL_N$ acts as  
\begin{equation}
    \label{eq:LLS_N}
    \LL_N f(\etah) = \sum_{x\in \mathbb{T}^2_N} \sum_{ |z|=1 } \nabla_{x,x+z} f(\etah), 
\end{equation}
where 
\[ 
    \nabla_{x,x+z} f(\etah) = \eta_x(1-\eta_{x+z})\widetilde{\nabla}_{x,x+z} f(\etah)\quad \text{and} \quad  \widetilde{\nabla}_{x,x+z} f(\etah)=f(\etah^{x,x+z})-f(\etah)
\]
are referred to as \textit{directed gradient} and \textit{undirected gradient} of $f$ respectively. 
Here $\etah^{x,x+z}$ denotes the configuration obtained by swapping $\etah_x$ and  $\etah_{x+z}$:     
\[ 
    \etah^{x,x+z}_y = \begin{cases}
        \etah_{x+z} & \text{if } y=x, \\ 
        \etah_{x} & \text{if } y=x+z, \\ 
        \etah_y & \text{otherwise. }
    \end{cases}
\]
The generator $\LLWA_N$ act on  $f\in \mathcal{C}$ as  
\begin{equation}\label{eq:LLWA_N}
    \LLWA_N f(\etah) = \sum_{x\in \mathbb{T}^2_N} \sum_{ |z|=1 } \frac{v_0}{2}z^\top \mathbf{e}(\theta_x)  \eta^a_x \nabla_{x,x+z} f(\etah).
\end{equation}
\begin{remark}$\quad$
    \normalfont
    \begin{enumerate}
        \item The term $\eta^a_x$  in the generator $\LLWA_N$ 
        indicates  that the symmetric jump rate  $D_T$ at $x$
        is biased only if the site is occupied by an active particle.  
        Passive particles, on the other hand, only  execute the SSEP with jump rate $D_T$. 
        \item 
        Although in our setting, the angles of passive particles do not affect their random displacement,   
        each passive particle is still assigned an angle.  
        This ensures  our non-gradient method may cover a broader class of models, 
        such as multi-type particle systems with continuous orientations,  
        where each type follows its own distinct angle-dependent displacement. 
        \item While $\LLWA_N$ alone is not a markov generator, the combined generator 
        $\LLD_N$ is, provided $N$ is sufficiently  large. 
    \end{enumerate}
\end{remark}

Next, the rotational diffusion part is described as follows. 
Each particles's orientations undergoes independent Brownian motion on $\SSS$ with a 
\textit{rotational diffusion constant} $D_R$. 
The generator $\LLR_N$ of the rotational diffusion acts on  $f\in \mathcal{C}$  as follows: 
\[
\LLR_N f(\etah) = D_R \sum_{x\in \mathbb{T}^2_N} \eta_x \partial_{\theta_x}^2 f(\etah).
\]

Now, the complete generator of the active-passive lattice gas (APLG) model is the superposition of 
the accelerated  displacement dynamics 
and the rotational diffusion: 
\begin{equation}\label{eq:generator}
    L_N = N^2\LLD_N + \LLR_N =N^2 D_T \LLS_N + N \LLWA_N + \LLR_N.
\end{equation}
Here we accelerate the displacement dynamic on a diffusive time scale, 
so that in the hydrodynamic limit, the contributions from $\LLS_N$, $\LLWA_N$, and $\LLR_N$
are balanced, capturing the relevant physical features. 

In order to introduce the initial distribution of our process, 
we define the \textit{density profile} 
as a pair of maps 
$\bm{\hat{f}}=(\hat{f}^a,\hat{f}^p):\mathbb{T}^2 \to \mathbb{M}_1(\SSS)$, where 
\[ 
    \mathbb{M}_1(\SSS)=\Big\{ \bah=(\hat{\alpha}^a,\hat{\alpha}^p)\in \mathcal{M}_1(\SSS)^2 
    \Big| \int_{\SSS}  \hat{\alpha}^a(\dd \theta)+\int_{\SSS} \hat{\alpha}^p(\dd \theta)\leq 1
    \Big\},
\] 
is referred to as the set of \textit{grand canonical parameter}s. 
If, for each $u$, the images measure $\hat{f}^\sigma(u,\dd \theta)$ 
has a density $f^\sigma(u,\theta)\dd \theta$ with $\sigma\in \{ \text{a}, \text{p} \}$, 
we refer to the pair of density functions 
$\bm{f}=(f^a,f^p):\mathbb{T}^2\times \SSS \to \mathbb{R}_+^2$ as the \textit{density profile functions}. 

The initial configuration is chosen from $\mu_{\bm{\zeta}}^N$,   
the \textit{product measure associated with the density profile functions} 
$\bm{\zeta}=(\zeta^a,\zeta^p)$ on $\TN$, which is defined as follows. 
For macroscopic point $u\in \mathbb{T}^2$,  
the integral   
\begin{equation}
    \label{eq:rho_0}
    \rho_0^\sigma(u)=\int_{\SSS}\zeta^\sigma(u,\theta)\dd\theta
\end{equation}
represent the local particle density of type $\sigma\in \{ \text{a}, \text{p} \}$ at $u$.  
Under $\mu_{\bm{\zeta}}^N$, each site $x\in \TN$ is initially occupied by a type $\sigma\in \{ \text{a}, \text{p} \}$ particle 
with probability $\rho^\sigma(x/N)$, whose angle $\theta_x$ is then sampled 
from the probability distribution 
\begin{equation}\label{eq:zeta_theta}
    \frac{\zeta^\sigma(x/N,\theta)}{\rho_0^\sigma(x/N)}\dd\theta.
\end{equation}
Starting from $\mu^N=\mu_{\bm{\zeta}}^N$, we define $\mathbb{P}_{\mu^N}$ the path distribution of the Markov process 
$\etah^{[0,T]}\in D([0,T],\Sigma_N)$ driven by the generator $L_N$ defined in~\eqref{eq:generator}. 
The corresponding expectation is denoted by $\EN$. 

Although the invariant measures of our process are not known, the symmetric displacement 
happens at the fastest rate $N^2$. 
Consequently, one can expect that  the invariant measures of SSEP generator $\LLS$ 
could  serve as 
 important reference measures. 
 For any grand canonical parameter $\bah\in \mathbb{M}_1(\SSS)$, 
 we denote by $\mu_{\bm{\hat{\alpha}}}$ the
\textit{grand canonical measure with parameter} $\bm{\hat{\alpha}}$, which 
is a product measure associated with a constant density profile 
$\bm{\hat{\alpha}}\equiv (\hat{\alpha}^a,\hat{\alpha}^p)$.  
Specially, 
when the parameter $(\hat{\alpha}^a,\hat{\alpha}^p)$ corresponds to uniform distributions, 
\begin{equation}\label{eq:mustar}
        \hat{\alpha}^\sigma(\dd \theta)=\frac{\alpha^\sigma}{2\pi}\dd \theta,\quad \sigma \in \{ \text{a}, \text{p} \}, 
\end{equation}
where $\bm{\alpha}\coloneqq (\alpha^a,\alpha^p)\in [0,1]^2$ such that $\alpha\coloneqq \alpha^a+\alpha^p\leq 1$, 
this special class of  grand canonical measure is denoted by $\mu^*_{\bm{\alpha}}$. 
The expectations associated with $\mu_{\bm{\hat{\alpha}}}$ and $\mu^*_{\bm{\alpha}}$ 
are denoted by $\E_{\bm{\hat{\alpha}}}$ 
and $\E^*_{\bm{\alpha}}$, respectively. 

$\quad$

\noindent \textbf{Conventions. }

For clarity, unless otherwise specified, through this paper,  $\sigma$ 
is used to refer to  both active and passive particle cases in a given context. 
Notations without subscripts are understood to represent the sum of active and passive quantities, 
such as  
$\alpha= \alpha^a+\alpha^p$, $\eta= \eta^a+\eta^p$. 
Boldface notation is used to denote pairs of active and passive quantities, such as  
$\bm{\alpha}=(\alpha^a,\alpha^p)$, $\bm{f}=(f^a,f^p)$. 

\subsection{Hydrodynamic limit}
We now introduce our main result, the hydrodynamic limit of the orientation density for the APLG, whose
derivation is given in Sction~\ref{sec:proof}.

Let $\mathcal{M}(\mathbb{T}^2\times \SSS)$ denote the space of non-negative measures on the macroscopic configuration space 
endowed with the weak topology, and let $\mathcal{M}^{[0,T]}=D([0,T],\mathcal{M}(\mathbb{T}^2\times \SSS))$. 
Each trajectory $\etah^{[0,T]}\in D([0,T],\Sigma_N)$ of the process 
admits a projection $\bm{\pi}^N=(\pi^{a,N},\pi^{p,N})$,  
where $\pi^{\sigma,N}=(\pi^{\sigma,N}_t)\in \mathcal{M}^{[0,T]}$ denotes the trajectory of \textit{empirical measure} 
\[
\pi^{\sigma,N}_t(\hat{\eta}^{[ 0,T ]}) = \frac{1}{N^2}\sum_{x\in \mathbb{T}^2_N} 
\eta^\sigma_x(t)\delta_{x/N,\theta_x(t)}, 
\]
with $\delta_u$ denoting the Dirac measure concentrated at $u$. 
Furthermore, we  define $Q^N\in \mathcal{P}(\mathcal{M}^{[0,T]}\times \mathcal{M}^{[0,T]})$ as 
the distribution of  $\pi^N(\etah^{[0,T]})$.  

Heuristically, in the hydrodynamic limit, we will show that the trajectory $\bm{\pi}^N_t(\dd u,\dd\theta)$ of empirical measures
converges as $N\to \infty$, in probability, to a deterministic trajectory 
\begin{equation}\label{eq:pi_t}
    \bm{\pi}_t(\dd u,\dd\theta)=\big( f^a_t(u,\theta), f^p_t(u,\theta) \big)\dd u\dd\theta
\end{equation}
where $(f^a_t,f^p_t)$ represents the local densities of active and passive particles with angle  $\theta$.  
Formally, they can be interpreted  as
\[ 
f^\sigma_t(u,\theta) \dd\theta \approx 
\frac{1}{(2N\varepsilon+1)^2}\sum_{|x-Nu|\leq N\varepsilon}\eta^\sigma_x(t)\mathbbm{1}_{\{\theta_x(t)\in [\theta,\theta+d\theta]\}}, 
\]
where $\varepsilon$ is the side length of a small macroscopic box. 
We further  define the  \textit{polarization}  as
\[ 
\mathbf{p}^a_t(u)=\int_{\SSS}\mathbf{e}(\theta) f^a_t(u,\theta)\dd\theta,
\]
and similar to~\eqref{eq:rho_0}, the particle density as  
\begin{equation}\label{eq:rho_t}
     \rho_t^\sigma(u)=\int_{\SSS}f^\sigma_t(u,\theta)\dd\theta. 
\end{equation}

Following the notations in~\cite{mason2023exact}, we introduce two functions $\mathcal{D}$ and 
$s$, defined on $[0,1]$ as follows:  
\[ 
\mathcal{D}(\rho)=\frac{1-d_s(\rho)}{\rho},\quad \text{ and } \quad 
s(\rho)=\mathcal{D}(\rho)-1,
\]
where $d_s(\rho)$  is the self-diffusion coefficient for the SSEP in dimension $2$, 
defined in the Appendix~\ref{sec:crossdiffusion}. 
Note $\mathcal{D}(\rho)$ has been continuously extended to $[0,1]$, see Remark~\ref{remark:regularity}. 

The main result of this paper (cf. Theorem~\ref{thm:main})  states that density profile functions  
$(f^a_t,f^p_t)$, defined in~\eqref{eq:pi_t}, are
weak solutions  (cf. Definition~\ref{def:weak}), to the coupled partial differential equations (PDEs):
\begin{equation}\label{eq:pde}
  \begin{cases}
\partial_t f^a_t =D_T \nabla \cdot [d_s(\rho_t)\nabla f^a_t+ f^a_t \mathcal{D}(\rho_t)\nabla \rho_t ]- v_0 \nabla \cdot [f^a_t s(\rho_t) \mathbf{p}^a_t +  f^a_t d_s(\rho_t) \mathbf{e}] + D_R \partial_\theta^2 f^a_t,\\
\partial_t f^p_t =D_T \nabla \cdot [d_s(\rho_t)\nabla f^p_t+ f^p_t \mathcal{D}(\rho_t)\nabla \rho_t ]+ D_R \partial_\theta^2 f^p_t ,
\end{cases}  
\end{equation}
with initial values $(f^a_0,f^p_0)=(\zeta^a,\zeta^p)$ introduced above~\eqref{eq:rho_0}. 

\begin{definition}[Weak solution to~\eqref{eq:pde}]\label{def:weak}
    \normalfont
    We say  $\bm{\pi}^{[0,T]}=(\pi^{a},\pi^{p}) = \{ ( \pi^{a}_t,\pi^{p}_t ) \}_{t\in [0,T]}\in \mathcal{M}^{[0,T]}\times \mathcal{M}^{[0,T]}$  
    is a weak solution to~\eqref{eq:pde} if it satisfies the following conditions:
    \begin{enumerate}
        \item $\pi^a_0(du,d\theta)=\zeta^a(u,\theta)\dd u \dd \theta$ and $\pi^p_0(du,d\theta)=\zeta^p(u,\theta)\dd u \dd \theta$.
        \item For all $t\in [0,T]$, the measures $\pi^a_t,\pi^p_t$ 
        are absolutely continuous with respect to the lebesgue measure on $\mathbb{T}^2\times \SSS$, 
        i.e., there exists a pair of density profile functions  
        $(f^a,f^p):[0,T]\times\mathbb{T}^2\times \SSS\to \mathbb{R}_+\times \mathbb{R}_+$ 
        such that 
        \[ 
            \pi_t^a(\dd u,\dd \theta)=f^a_t(u,\theta)\dd u \dd \theta,\quad \text{ and }\quad 
            \pi_t^p(\dd u,\dd \theta)=f^p_t(u,\theta)\dd u \dd \theta. 
        \]
        \item The functions $\rho^a_t(u)$ and $\rho^p_t(u)$ as defined in~\eqref{eq:rho_t}  
        are in $L^2([0,T];H^1(\mathbb{T}^2))$. 
        \item For every  test function $H\in C^{1,2,2}([0,T]\times \mathbb{T}^2\times \SSS)$, the following holds:
        \begin{equation}\label{eq:weak1}
            \begin{aligned}
                &\big< f^a_T , H_T \big>-\big< f^a_0 , H_0 \big>=\int_0^T \big< f^a_t , \partial_t H+D_R \partial_\theta^2 H \big> \dd t \\
                &\qquad + \int_0^T \int_{\mathbb{T}^2\times \SSS} \Bigg\{  D_T 
                \Big[ \nabla^2 H d_s(\rho_t)f^a_t -\nabla H \cdot \nabla \rho_t  \big( \mathcal{D}(\rho_t) -d_s'(\rho_t) \big)f^a_t \Big]\\
                &\qquad +v_0\nabla H \cdot \big(   s(\rho_t) \mathbf{p}^a_t + d_s(\rho_t) \mathbf{e}(\theta) \big)f^a_t 
                \Bigg\} \dd u\dd \theta \dd t,
            \end{aligned}
        \end{equation}
        and 
        \begin{equation}\label{eq:weak2}
            \begin{aligned}
                &\big< f^p_T , H_T \big>-\big< f^p_0 , H_0 \big>=\int_0^T \big< f^p_t , \partial_t H +D_R \partial_\theta^2 H\big> \dd t \\
                &\quad+ \int_0^T \int_{\mathbb{T}^2\times \SSS} \Bigg\{  D_T 
                \Big[ \nabla^2 H d_s(\rho_t)f^p_t -\nabla H \cdot \nabla \rho_t  \big( \mathcal{D}(\rho_t) -d_s'(\rho_t) \big)f^p_t \Big]\Bigg\} \dd u\dd \theta \dd t,
            \end{aligned}
        \end{equation}
        Here, $\big< f , g \big>=\int_{\mathbb{T}^2\times \SSS}fg  \mathrm{d}  u \mathrm{d}\theta $. 
    \end{enumerate}
\end{definition}
We are now  state the main result.
\begin{theorem}
    \label{thm:main}
    The sequence $\big(Q^N \big)_{n\in \mathbb{N}}$ is weakly relatively compact, and
 any of its limit points $Q^*$ are concentrated on trajectories $\bm{\pi}^{[0,T]}=(\pi^{a},\pi^{p}) \in \mathcal{M}^{[0,T]}\times \mathcal{M}^{[0,T]}$, 
 which are weak solutions of~\eqref{eq:pde} in the sense of Definition~\ref{def:weak}.
\end{theorem}
\begin{remark}
    \normalfont 
    In Definition~\ref{def:weak}, we only require $\rho_t \in H^1$, 
        so the term $d_s(\rho)\nabla f^\sigma$ in the PDEs~\eqref{eq:pde} is understood as
        \begin{equation}\label{eq:rkintbyparts}
            d_s(\rho)\nabla f^\sigma=\nabla \big( d_s(\rho)f^\sigma \big)-d'_s(\rho)f^\sigma\nabla \rho
        \end{equation}
        This reformulation explains the second line of~\eqref{eq:weak1} and~\eqref{eq:weak2}.

\end{remark}

\section{Proof of Theorem~\ref{thm:main}}\label{sec:proof}
In this section we prove the main result of this paper Theorem~\ref{thm:main}.  
The strategy of proof mainly follows the classical entropy method presented 
in~\cite{Erignoux21,mason2023exact}.  
In Section~\ref{subsect:weaker}, 
we follow the route of~\cite{mason2023exact} to reduce the problem to considering the weaker solution of PDEs.  
In Section~\ref{subsect:dynkin}, we use Dynkin's formula to derive the  fluctuation of the microscopic system around the hydrodynamic limit.
The main novelty and difficulty arise in Section~\ref{subsect:nongrad}, 
where we explain a generalized non-gradient decomposition of typed currents 
and introduce  replacement lemmas. 
In Section~\ref{subsec:compactreg} we tackle  the compactness and regularity issues. 
In Section~\ref{subsec:secondint}, we apply the second discrete integration by parts.    
In Section~\ref{subsec:funcofem} , We convert  all relevant terms into functions of empirical measures, 
which are amenable for taking limit in Section~\ref{subsec:limit}.  

\subsection{Reduction to weaker solutions}\label{subsect:weaker}
Follow the Remark 5.8 in~\cite{mason2023exact}, which leverages the angular diffusion feature of the PDE to 
lift   
the weaker solution (cf. Definition~\ref{def:weaker}) whose angular part is measure 
to the weak solution (cf. Definition~\ref{def:weak}) whose angular part is absolutely continuous with respect to the Lebesgue
measure on $\SSS$.  
Taking advantage of this fact, it suffices to prove that  
the trajectory $\bm{\pi}^{[0,T]}=(\pi^a, \pi^p)$  
is a weaker solution of~\eqref{eq:pde} as defined in Definition~\ref{def:weaker}, 
which was introduced on page 19 of~\cite{Erignoux21}. 

\begin{definition}[Weaker solution of~\eqref{eq:pde}]\label{def:weaker}
    \normalfont
    We say  $\bm{\pi}^{[0,T]}=(\pi^{a},\pi^{p}) \in \mathcal{M}^{[0,T]}\times \mathcal{M}^{[0,T]}$  
    is a weaker solution to~\eqref{eq:pde} if it satisfies the following conditions:
    \begin{enumerate}
        \item $\pi^a_0(du,d\theta)=\zeta^a(u,\theta)\dd u \dd \theta$ and $\pi^p_0(du,d\theta)=\zeta^p(u,\theta)\dd u \dd \theta$.
        \item For all $t\in [0,T]$, the measures $\pi^a_t,\pi^p_t$ 
        are absolutely continuous with respect to the lebesgue measure on $\mathbb{T}^2$, 
        that is there exists a pair of density profiles 
        $(\hat{f}^a,\hat{f}^p):[0,T]\times\mathbb{T}^2 \to \mathcal{M}(\SSS)^2$ 
        such that 
        $$\pi_t^\sigma (du,d\theta)=\hat{f}^\sigma_t(u,\dd \theta)\dd u ,\quad \forall \sigma\in \{a,p\}.$$
        \item The function $\rho^\sigma_t(u)=\int_{\SSS} \hat{f}^\sigma_t(u,\dd \theta)$, 
        is in $L^2([0,T];H^1(\mathbb{T}^2))$ 
        \item For all test function $H\in C^{1,2,2}([0,T]\times \mathbb{T}^2\times \SSS)$,
        \begin{equation}\label{eq:weaka}
            \begin{aligned}
                &\left< \pi^a_T , H_T \right>-\left< \pi^a_0 , H_0 \right>
                =\int_0^T \left< \pi_t^a , \partial_t H + D_R \partial_\theta^2 H \right> dt \\
                &+ \int_0^T \int_{\mathbb{T}^2\times \SSS} \Bigg\{  D_T 
                \Big[ \nabla^2 H d_s(\rho_t)\hat{f}_t^a(u,\dd \theta) 
                -\nabla H \cdot \nabla \rho_t(u)  \big[ \mathcal{D}(\rho_t)
                -d_s'(\rho_t) \big]\hat{f}^a_t(u,\dd \theta) \Big]\\
                &\quad +v_0\nabla H \cdot 
                \Big[   s(\rho_t ) \mathbf{p}^a_t + d_s(\rho_t) \mathbf{e}(\theta) \Big]\hat{f}^a_t(u,\dd \theta)
                \Bigg\} \dd u \dd t,
            \end{aligned}
        \end{equation}
        and 
        \begin{equation}\label{eq:weakp}
            \begin{aligned}
                &\left< \pi^p_T , H_T \right>-\left< \pi^p_0 , H_0 \right>=
                \int_0^T \left< \pi_t^p , \partial_t H+D_R \partial_\theta^2 H   \right> \dd t\\ 
                &\qquad + \int_0^T \int_{\mathbb{T}^2\times \SSS} \Bigg\{  D_T 
                \Big[ \nabla^2 H d_s(\rho_t)\hat{f}_t^p(u,\dd \theta) \\
                &\qquad -\nabla H \cdot \nabla \rho_t(u)  \big[ \mathcal{D}(\rho_t) 
                -d_s'(\rho_t)\big]\hat{f}^p_t (u,\dd \theta) \Big]
                \Bigg\} \dd u\dd t.
            \end{aligned}
        \end{equation}
    \end{enumerate}
\end{definition}

By Remark 5.8 in~\cite{mason2023exact}, proving Theorem~\ref{thm:main} can be reduced to proving Theorem~\ref{thm:main2}, 
which becomes our main task in the rest of this paper. 
\begin{theorem}
    \label{thm:main2}
    The sequence $\big(Q^N \big)_{n\in \mathbb{N}}$ is weakly relatively compact, and
 any of its limit points $Q^*$ are concentrated on trajectories $\bm{\pi}^{[0,T]}=(\pi^{a},\pi^{p}) \in \mathcal{M}^{[0,T]}\times \mathcal{M}^{[0,T]}$, 
 which are weak solutions of~\eqref{eq:pde} in the sense of Definition~\ref{def:weaker}.
\end{theorem}
To make the exposition clearer, 
we divide the proof in several steps, detailed in the following subsections. 

\subsection{Dynkin's formula and instantaneous currents}\label{subsect:dynkin}
Fix a test function 
$H\in C^{1,2,2}([0,T]\times \mathbb{T}^2\times \SSS)$, and by Dynkin's formula we define a martingale
\begin{equation}\label{eq:dynkin}
    M^{\sigma,H,N}_t = \big< \pi^{\sigma,N}_t , H_t \big> - \big< \pi^{\sigma,N}_0 , H_0 \big>-
    \int_{0}^t \Big[ \big< \pi^{\sigma,N}_s ,\partial_s H_s \big>+L_N \big< \pi^{\sigma,N}_s , H_s \big> \Big]\dd s.
\end{equation}
By similar computations as those on page 21 of~\cite{Erignoux21},  
the quadratic variation $[M^{\sigma,H,N}]_t$  can be bounded as follows:   
\[ 
    [M^{\sigma,H,N}]_t\leq C \frac{t\|H\|_{L^\infty}}{N^2}.
\]
Applying the Burkholder-Davis-Gandy inequality~\cite{le2016brownian} and Markov inequality 
yields that 
\begin{equation}\label{eq:mgvanish}
    \lim_{N\to \infty}\PN \Bigg( \sup_{0\leq t\leq T}\big|
   M_t^{\sigma,H,N} \big|\geq \delta \Bigg)=0,\quad \forall T>0,\delta>0.
\end{equation}

Next using a periodic version of the Weierstrass approximation theorem proposed on page 74-75 of~\cite{Erignoux21}, 
we only need to consider test functions  of the following form: 
\begin{equation}\label{eq:testfunc}
    H_t(u,\theta)=G_t(u)\omega(\theta)
\end{equation}
where $G\in C^{1,2}([0,T]\times \mathbb{T}^2)$ and $\omega \in C^{2}( \SSS)$. 

In order to further expand the right hand side of~\eqref{eq:dynkin} and to 
get weak formulations~\eqref{eq:weaka} and~\eqref{eq:weakp}, we introduce some notations. 
For any angular function $\omega:\SSS\to \mathbb{R}$, any configuration $\etah$, we shorten 
\begin{equation}
    \eta_x^{\omega}=\omega(\theta_x)\eta_x, \quad \text{ and } \quad 
    \eta_x^{\sigma,\omega}=\omega(\theta_x)\eta^{\sigma}_x.
\end{equation}
For any test function $H$ of the form~\eqref{eq:testfunc}, 
applying discrete integration by parts  absorbing an $N$ factor from diffusive scaling $N^2$ yields: 
\begin{equation}\label{eq:dyna}
    \begin{aligned}
        L_N\big<\pi_t^{a,N},H_t\big>=&
        D_T \frac{1}{N^2}\sum_{x\in \mathbb{T}^2_N}\sum_{i=1,2}
        \tau_x\big(  N \joa_i+\ro_i \big)(t) 
\partial_{i,N}G_t (x/N)\\ 
&\quad + \big< \pi_t^{a,N} , D_R \partial_\theta^2 H_t \big>,
    \end{aligned}
\end{equation}
and 
\begin{equation}\label{eq:dynp}
    \begin{aligned}
        L_N\big<\pi_t^{p,N},H_t\big>=&
        D_T \frac{1}{N^2}\sum_{x\in \mathbb{T}^2_N}\sum_{i=1,2}
        \tau_x N \jop_i(t) 
\partial_{i,N}G_t (x/N)\\ 
&\quad + \big< \pi_t^{p,N} , D_R \partial_\theta^2 H_t \big>. 
    \end{aligned}
\end{equation}
Here  $\partial_{i,N}$ is the discrete partial derivative in the $i$-th direction: 
\[ 
    \partial_{i,N}G_t (x/N)=N\Big[ G_t\big( \frac{x+e_i}{N} \big)-G_t\big(\frac{x}{N}\big) \Big], 
\]
and $j^{\sigma,\omega}_i$ is 
the \textit{symmetric  current with intensity $\omega$} 
from $0$ to $e_i$:
\begin{equation}\label{eq:j}
    j^{\sigma,\omega}_i = \eta_0^{\sigma,\omega}(1-\eta_{e_i})-\eta_{e_i}^{\sigma,\omega}(1-\eta_{0}),
\end{equation}
and $r^{\omega}_i$ is 
the \textit{asymmetric  current with intensity $\omega$} 
from $0$ to $e_i$:
\begin{equation}\label{eq:r}
        r^{\omega}_i = \eta_0^{a,\omega\lambda_i}(1-\eta_{e_i})+\eta_{e_i}^{a,\omega\lambda_i}(1-\eta_{0}),
\end{equation}
where $\lambda_i$ is the normalized velocity: 
\begin{equation}\label{eq:lambda}
    \lambda_i(\theta) = \frac{v_0}{2 D_T }\mathbf{e}(\theta)^{\top}e_i. 
\end{equation}

\begin{remark}
    Define the edge availability function $a(x,y)= \eta_x(1-\eta_y)+\eta_y(1-\eta_x)$ 
indicating availability of particle transport through $(x,y)$  under the exclusion rule.  
This allows  compact reformulations of the currents in~\eqref{eq:j} and~\eqref{eq:r}:   
\[ 
    \begin{aligned}
        &j_i^{\sigma,\omega}=a(0,e_i)(\eta_0^{\sigma,\omega}-\eta_{e_i}^{\sigma,\omega}),\\
        &r_i^{\omega}=a(0,e_i)(\eta_0^{a,\omega\lambda_i}+\eta_{e_i}^{a,\omega\lambda_i})
    \end{aligned}
\]
These observations will be useful throughout this paper.
\end{remark}

Note that $r_i^\omega$ only appears in~\eqref{eq:dyna} 
which introduces additional  technical complexities in the treatments of~\eqref{eq:dyna} compared to~\eqref{eq:dynp}.  
Once we complete the proofs related to~\eqref{eq:dyna}, the results for~\eqref{eq:dynp} can be obtained similarly.  
Therefore, in the rest of the proof of Theorem~\ref{thm:main2}, 
we only need to focus on the active part of the PDEs. 

\subsection{Non-gradient decomposition and replacement lemma}\label{subsect:nongrad}

Note that in~\eqref{eq:dyna}, 
we must balance the remaining extra factor $N$ multiplying  $j^{a,\omega}_i$ 
by a second discrete integration by parts. Therefore, further expanding the current $j^{a,\omega}_i$ 
into a discrete gradient form is needed. 
This task is done in Lemma~\ref{lemma:nongradreplacement}, which is the main goal of this subsection. 
Before stating the lemma precisely, we provide a heuristic explanation of the underlying ideas. 

One main challenge in deriving the 
hydrodynamic limit of  our model 
comes from the non-gradient feature of the generator $L_N$.    
This means that we cannot further expand  $j^{a,\omega}_i$ 
into a discrete gradient form:  
\begin{equation}\label{eq:grad}
    \del_i h = \tau_{e_i}h-h,
\end{equation}
for some local function $h$. 
Without such discrete gradient expansion, 
we are unable to perform a second  discrete integration by parts, 
which absorbs the remaining factor $N$ in~\eqref{eq:dyna}. 
To simplify the notations, 
we define the \textit{local gradients} as follows: 
$$\del_i\eta^{\sigma,\omega}=\eta^{\sigma,\omega}_{e_i}-\eta^{\sigma,\omega}_0, $$ 
for any 
angular function $\omega$.  Similarly, 
$\del_i\eta^{\sigma}$, $\del_i\eta$ are defined in the same manner. 

Next we perform the non-gradient decomposition of  $j^{a,\omega}_i$, i.e., the current  
$j^{a,\omega}_i$, in a certain Hilbert space, can be approximated by a linear combination of local gradients 
$\{\bm{\delta}_i \eta^{a, \omega}, \del_i \eta\}$ 
and a function in the range of the SSEP generator $\LL$ defined in~\eqref{eq:LLS_N}:  
\begin{equation}\label{eq:localreplace}
    j_i^{a,\omega}=\lim_{n\to \infty}  \big( -d_s(\rho) \bm{\delta}_i \eta^{a, \omega}
    -\overline{\omega}^a \rho^a \mathcal{D}(\rho) 
    \del_i \eta 
    - \LL  f^{\omega}_{i,n} \big),
\end{equation}
for some local function $f^{\omega}_{i,n}$. For a more precise statement, see Lemma~\ref{lemma:uniform}.  

With the non-gradient decomposition~\eqref{eq:localreplace},   
the combination of symmetric and asymmetric currents $N j^{a,\omega}_i + r^\omega_i$ in~\eqref{eq:dyna}, 
after adding and subtracting $D_T^{-1}\LLWA f^{\omega}_{i,n}$, can be approximated by 
\begin{equation}\label{eq:localcomb}
    -N (d_s(\rho) \bm{\delta}_i \eta^{a, \omega}
    +\overline{\omega}^a \rho^a \mathcal{D}(\rho) 
    \del_i \eta 
    ) + (r^\omega_i + D_T^{-1}\LLWA f^{\omega}_{i,n})  - D_T^{-1} N  \LLD  f^{\omega}_{i,n} ,
\end{equation}
where $D_T^{-1} N\LLD f^{\omega}_{i,n}$  is a small perturbation which can be averaged out using classical martingale estimate, 
as detailed Lemma~\ref{lemma:diminishLD}. 

Furthermore, to close the microscopic equation~\eqref{eq:dynkin} and to take limits in~\eqref{eq:dyna} and~\eqref{eq:dynp}, 
we need to replace the local quantities in~\eqref{eq:localcomb} 
with their mesoscopic averages as follows:  
    \begin{equation}\label{eq:localdiff}
        d_s(\rho) \bm{\delta}_i \eta^{a, \omega}
        +\overline{\omega}^a \rho^a \mathcal{D}(\rho) 
        \del_i \eta 
         \approx  d_s(\rho_{\varepsilon N}) \bm{\delta}_i \rho^{a,\omega}_{\varepsilon N}
        + \rho^{a,\omega}_{\varepsilon N} \mathcal{D}(\rho_{\varepsilon N}) 
        \del_i \rho_{\varepsilon N} ,
    \end{equation}
and 
    \begin{equation}\label{eq:localdrift}
        r^\omega_i +D_T^{-1} \LLWA f^{\omega}_{i,n}  \approx 
        \E_{\bm{\hat{\rho}}_{\varepsilon N}}\big[ r^\omega_i + D_T^{-1} \LLWA f^{\omega}_{i,n} \big],
    \end{equation}
Here, we define \textit{empirical densities} $\del_i \rho^\sigma_{\epsilon N}$, $\del_i \rho_{\epsilon N}$, 
and $\del_i \rho^{\sigma,\omega}_{\epsilon N}$ as the averages of local gradients 
$\del_i \eta^{\sigma}$, $\del_i \eta $, and $\del_i \eta^{\sigma,\omega}$ over a mesoscopic box $B_{\epsilon N }$: 
\[ 
\begin{aligned}
    \rho^\sigma_{\varepsilon N} &=  \frac{1}{(2\varepsilon N + 1)^2} \sum_{x\in B_{\varepsilon N}} \eta^\sigma_x =
    \frac{1}{(2\varepsilon)^2}\big< \mathbbm{1}_{[-\varepsilon,\varepsilon]^2} , \pi^{\sigma,N} \big>+o_N(1), \\
    \rho^{\sigma,\omega}_{\varepsilon N} &=  \frac{1}{(2\varepsilon N + 1)^2} \sum_{x\in B_{\varepsilon N}} \eta^{\sigma,\omega}_x =
    \frac{1}{(2\varepsilon)^2}\big< \mathbbm{1}_{[-\varepsilon,\varepsilon]^2} \omega, \pi^{\sigma,N} \big>+o_N(1).
\end{aligned}
\]
The expectation $\E_{\bm{\hat{\rho}}_{\varepsilon N}}$ is taken 
with respect to the grand canonical measure with parameter 
$\widehat{\bm{\rho}}_{\varepsilon N}=(\hat{\rho}^a_{\varepsilon N},\hat{\rho}^p_{\varepsilon N})$, 
where $\hat{\rho}^\sigma_{\varepsilon N}$ is the \textit{empirical angular density}, 
\[  
    \hat{\rho}^\sigma_{\varepsilon N} = 
    \frac{1}{(2\varepsilon N + 1)^2} \sum_{x\in B_{\varepsilon N}} \eta^\sigma_x \delta_{\theta_x}\in \mathcal{M}_1(\SSS).
\]
In Lemma~\ref{lemma:X2} we can further calculate the expectation in~\eqref{eq:localdrift} 
and get the replacement: 
\begin{equation}\label{eq:localdrift2}
    r^\omega_i + D_T^{-1} \LLWA f^{\omega}_{i,n} \approx 
    2\Big( d_s(\rho_{\varepsilon N}) \rho^{a,\omega \lambda_i }_{\varepsilon N} 
    +\rho^{a,\omega}_{\varepsilon N}\rho^{a, \lambda_i }_{\varepsilon N} s(\rho_{\varepsilon N})
    \Big).
\end{equation}

In summary, combing the replacements~\eqref{eq:localdiff} and~\eqref{eq:localdrift2}, 
and discarding the small perturbation term  $D_T^{-1} N \LLD f^{\omega}_{i,n}$ in~\eqref{eq:localcomb}, 
we can finally replace  $N j^{a,\omega}_i + r^\omega_i $ by
\begin{equation}\label{eq:fullreplace}
        -N \Big(d_s(\rho_{\varepsilon N}) \bm{\delta}_i \rho^{a,\omega}_{\varepsilon N}
        + \rho^{a,\omega}_{\varepsilon N} \mathcal{D}(\rho_{\varepsilon N}) 
        \del_i \rho_{\varepsilon N}
        \Big)+2\Big( d_s(\rho_{\varepsilon N}) \rho^{a,\omega \lambda_i }_{\varepsilon N} 
        +\rho^{a,\omega}_{\varepsilon N}\rho^{a, \lambda_i }_{\varepsilon N} s(\rho_{\varepsilon N})
        \Big),  
\end{equation}
which is formalized by the following replacement lemma:
\begin{lemma}\label{lemma:nongradreplacement}
    Let 
\[ 
Y_i^{\varepsilon N}(\etah)=N\Big( j^{a,\omega}_i +  d_s(\rho_{\varepsilon N}) \bm{\delta}_i \rho^{a,\omega}_{\varepsilon N}
+ \rho^{a,\omega}_{\varepsilon N} \mathcal{D}(\rho_{\varepsilon N}) 
\del_i \rho_{\varepsilon N} \Big),
\]
and 
\[ 
Z_i^{\varepsilon N}(\etah) = r^\omega_i  -
2\Big( d_s(\rho_{\varepsilon N}) \rho^{a,\omega \lambda_i }_{\varepsilon N} 
+\rho^{a,\omega}_{\varepsilon N}\rho^{a, \lambda_i }_{\varepsilon N} s(\rho_{\varepsilon N})
\Big).
\]
Then for any $T>0$, $i=1,2$, and $G\in C^{1,2}([0,T]\times \TT)$,  
\[ 
\limsup_{\varepsilon\to 0} \limsup_{N\to \infty}\EN  
\Bigg[ \Bigg| 
\int_{0}^T \frac{1}{N^2}\sum_{x\in \TN} G_t(x/N)\tau_x  
\big( Y_i^{\varepsilon N}+Z_i^{\varepsilon N} \big)(\etah(t)) \dd t
\Bigg| \Bigg]=0.
\]
\end{lemma}
Before giving the proof of Lemma~\ref{lemma:nongradreplacement}, 
we now close the equation~\eqref{eq:dynkin}. 
For any $\delta>0$,  Lemma~\ref{lemma:nongradreplacement} and~\eqref{eq:mgvanish} 
together yield 
\begin{equation}\label{eq:limittilde}
    \limsup_{\varepsilon \to 0} \lim_{N\to \infty} 
    \PN \Big( \big| \widetilde{M}^{a,H,N,\varepsilon}_T \big|\geq \delta \Big) =0,
\end{equation}
where 
\begin{equation}\label{eq:mgtilde}
    \begin{aligned}
        \widetilde{M}^{a,H,N,\varepsilon}_T=&\big< \pi^{a,N}_T , H_T \big> - \big< \pi^{a,N}_0 , H_0 \big>-
        \int_{0}^T \big< \pi^{a,N}_t ,\partial_t H_t + D_R \partial_{\theta}^2 H_t\big>\dd t\\
        &+D_T\sum_{i=1,2}\int_{0}^T \dd t \frac{1}{N^2}\sum_{x\in \mathbb{T}^2_N}
        \tau_x\Big(  N(d_s(\rho_{\varepsilon N}) \del_i \rho^{a,\omega}_{\varepsilon N}
        + \rho^{a,\omega}_{\varepsilon N} \mathcal{D}(\rho_{\varepsilon N}) 
        \del_i \rho_{\varepsilon N}) \\
        &\qquad\qquad\qquad  -2( d_s(\rho_{\varepsilon N}) \rho^{a,\omega \lambda_i }_{\varepsilon N} 
        +\rho^{a,\omega}_{\varepsilon N}\rho^{a, \lambda_i }_{\varepsilon N} s(\rho_{\varepsilon N})
        ) \Big)(t) \partial_{i,N}G_t (x/N).
    \end{aligned}
\end{equation}

\begin{proof}[Proof of Lemma~\ref{lemma:nongradreplacement}]
Following the ideas illustrated in~\eqref{eq:localcomb},~\eqref{eq:localdiff}, and~\eqref{eq:localdrift2}, 
by adding and subtracting $N \LLD f^{\omega}_{i,n}$, we rewrite 
\[ 
Y_i^{\varepsilon N}+Z_i^{\varepsilon N} 
=\mathcal{V}_i^{f^{\omega}_{i,n},\varepsilon N} + \mathcal{R}_i^{f^{\omega}_{i,n},\varepsilon N}
-D_T^{-1} N\LLD f^{\omega}_{i,n}, \quad \forall n\in \mathbb{N},
\]
where 
\[ 
    \begin{aligned}
        &\mathcal{V}_i^{f,l}=N\big( j^{a,\omega}_i +  d_s(\rho_{l}) \bm{\delta}_i \rho^{a,\omega}_{l}
        + \rho^{a,\omega}_{l} \mathcal{D}(\rho_{l}) 
        \del_i \rho_{l} + \LL f \big),\\
        &\mathcal{R}_i^{f,l} = r_i^\omega + D_T^{-1} \LLWA f - 
    2\big( d_s(\rho_{l}) \rho^{a,\omega \lambda_i }_{l} 
    +\rho^{a,\omega}_{l}\rho^{a, \lambda_i }_{l} s(\rho_{l}) \big).
    \end{aligned}
\]
Now, the following Lemma~\ref{lemma:replaceV},~\ref{lemma:replaceR}, and~\ref{lemma:diminishLD} 
together conclude the proof of Lemma~\ref{lemma:nongradreplacement}.
\end{proof}
\begin{lemma}[Non-gradient replacement]\label{lemma:replaceV}
    For any $G\in C^{1,2}([0,T]\times \TT)$, and $\omega \in C^2(\SSS)$, 
    For each $i=1,2$, there exists a sequence  of cylinder functions $(f^{\omega}_{i,n})_{n\in \mathbb{N}}$ such that  
    \[ 
        \lim_{n\to \infty}\limsup_{\varepsilon\to 0} \limsup_{N\to \infty} \EN \Bigg[ \Bigg| 
    \int_{0}^T \frac{1}{N^2}\sum_{x\in \TN}G_t(x/N)\tau_x \mathcal{V}_i^{f^{\omega}_{i,n},\varepsilon N}(\etah(t)) \dd t
    \Bigg| \Bigg]=0.
    \]
\end{lemma}
Due to the involvement of the non-gradient decomposition, 
the proof of Lemma~\ref{lemma:replaceV} is the most technical one, 
and hence has to be handled carefully. We postpone it to Section~\ref{sec:replaceV}.  
\begin{lemma}[Drift term replacement]\label{lemma:replaceR}
    Let $G\in C^{1,2}([0,T]\times \TT)$,  $\omega \in C^2(\SSS)$, and  $(f^{\omega}_{i,n})_{n\in \mathbb{N}}$ 
    be as defined in Lemma~\ref{lemma:replaceV}. It holds that for $i=1,2$,
    \[ 
        \lim_{n\to \infty} \limsup_{\varepsilon\to 0} \limsup_{N\to \infty} \EN \Bigg[ \Bigg| 
    \int_{0}^T \frac{1}{N^2}\sum_{x\in \TN}G_t(x/N)\tau_x \mathcal{R}_i^{f^{\omega}_{i,n},\varepsilon N}(\etah(t)) \dd t
    \Bigg| \Bigg]=0.
    \]
\end{lemma}
The ideas of Lemma~\ref{lemma:replaceR} is the replacement~\eqref{eq:localdrift} 
and the computation of the expectation in Lemma~\ref{lemma:X2}. 
Therefore, Lemma~\ref{lemma:replaceR} follows directly from Lemma~\ref{lemma:X1} and~\ref{lemma:X2}. 
Its proof is shown in Section~\ref{sec:replaceR}. 
\begin{lemma}\label{lemma:diminishLD}
    Let $G\in C^{1,2}([0,T]\times \TT)$,  $\omega \in C^2(\SSS)$, and  $(f^{\omega}_{i,n})$ 
    be as defined in Lemma~\ref{lemma:replaceV}. It holds that for $i=1,2$, 
    \[ 
        \lim_{n\to \infty} \limsup_{N\to \infty}\EN \Bigg[ \Bigg| 
    \int_{0}^T \frac{1}{N^2}\sum_{x\in \TN}G_t(x/N)\tau_x \LLD f^{\omega}_{i,n}(\etah(t)) \dd t
    \Bigg| \Bigg]=0.
    \]
\end{lemma}
The proof of Lemma~\ref{lemma:diminishLD} 
follows  similar martingale estimate in Lemma 6.7.2 on page 122 of~\cite{Erignoux21}, we omit it for brevity. 

\subsection{Relative compactness and regularity of the densities}\label{subsec:compactreg}
The proof of the relative compactness of $\big(Q^N \big)_{N\in \mathbb{N}}$ is  standard 
and relies on Aldous criterion  which gives a necessary and sufficient condition for a sequence of measures
to be weak relative compact in the Skorokhod topology. Following the same arguments as those for the pure active exclusion process, 
see Proposition 5.2.2 on page 71 and Theorem 5.3.1 on page 76 of~\cite{Erignoux21}, we can have similar results.
\begin{lemma}[Relative compactness and regularity of the densities]\label{lemma:compactreg}
    The sequence $(Q^N)_{N\in \mathbb{N}}$  is relative compact, 
    and any of its accumulation points $Q^*$ is concentrated on trajectories that are
    \begin{enumerate}
        \item  absolutely continuous with respect to the  Lebesgue measure on $\TT$:
        \[ 
        Q^*\big( \big\{ \pi:
        \pi_t^\sigma(\dd u,\dd \theta)= \hat{f}^\sigma(u,\dd \theta) \dd u
        \quad \forall t \in [0,T],\sigma\in \{a,p\}
        \big\}  \big)=1.
        \]
        \item such that the local density 
        \[ 
        \rho_t^\sigma(u)=\int_{\SSS} \hat{f}_t^\sigma(u,\dd \theta), 
        \]
        is in $H^1([0,T]\times \TT )$, meaning  that, $Q^*$-a.s., for $i=1,2$,  there exists a function 
        $\partial_{u_i}\rho^\sigma\in L^2([0,T]\times \TT )$ such that for any  test function $H\in C^{\infty}([0,T]\times \TT )$, 
        \[
            \int_0^T \int_{\TT} \partial_{u_i}H_t(\theta)\rho_t^\sigma(u)\dd \theta \dd t=
            -\int_0^T \int_{\TT} H_t(\theta)\partial_{u_i}\rho_t^\sigma(u)\dd \theta \dd t.
        \]
    \end{enumerate}
\end{lemma}   

\subsection{The second integration by parts}\label{subsec:secondint} 
After replacing  $N j^{a,\omega}_i + r^\omega_i $ in~\eqref{eq:dyna} with 
\[ 
    -N \Big(d_s(\rho_{\varepsilon N}) \bm{\delta}_i \rho^{a,\omega}_{\varepsilon N}
    + \rho^{a,\omega}_{\varepsilon N} \mathcal{D}(\rho_{\varepsilon N}) 
    \del_i \rho_{\varepsilon N}
    \Big)+2\Big( d_s(\rho_{\varepsilon N}) \rho^{a,\omega \lambda_i }_{\varepsilon N} 
    +\rho^{a,\omega}_{\varepsilon N}\rho^{a, \lambda_i }_{\varepsilon N} s(\rho_{\varepsilon N})
    \Big),  
\]which has been done in Section~\ref{subsect:nongrad}, 
we are now ready to perform the second discrete integration by parts, absorbing the remaining factor $N$. 

Due to the regularity of $\rho_t\in H^1(\TT)$ obtained in Lemma~\ref{lemma:compactreg}, 
we can hope that the limits related to $ \del_i \rho_{\varepsilon N}$ can be handled relatively straightforwardly. 
However, since we lack regularity results for $f^\sigma$,  in the limit process,  
we need to circumvent dealing directly with  $\del_i \rho^{a,\omega}_{\varepsilon N}$. 
To this goal, we use a discrete version of~\eqref{eq:rkintbyparts} to write:  
\[ 
    d_s(\rho_{\varepsilon N}) \del_i\rho^{a,\omega}_{\varepsilon N}=
    \del_i \big( d_s(\rho_{\varepsilon N}) \rho^{a,\omega}_{\varepsilon N} \big)
    -\rho^{a,\omega}_{\varepsilon N}d_s'(\rho_{\varepsilon N}) \del_i \rho_{\varepsilon N}
    +o_N(1/N).
\]
Next we sum up all terms in front of $\del_i \rho_{\varepsilon N}$:  
\[ 
\begin{aligned}
    D^{\varepsilon N}_0 \coloneqq &  \rho^{a,\omega}_{\varepsilon N} \mathcal{D}(\rho_{\varepsilon N})
        - \rho^{a,\omega}_{\varepsilon N}d_s'(\rho_{\varepsilon N})\\
        =&  \rho^{a,\omega}_{\varepsilon N} \big[ \mathcal{D}- d_s' \big](\rho_{\varepsilon N}),  
\end{aligned}
\]
and denote $D_x^{\varepsilon N}=\tau_x D^{\varepsilon N}_0$, for any $x\in \TN$. 

After above preparations, the second discrete integration by parts is performed.   
We can now rewrite~\eqref{eq:mgtilde} as: 
\begin{equation}\label{eq:mgtilde2}
    \begin{aligned}
        \widetilde{M}^{a,H,N,\varepsilon}_T=&\big< \pi^{a,N}_T , H_T \big> - \big< \pi^{a,N}_0 , H_0 \big>-
        \int_{0}^T \big< \pi^{a,N}_t ,\partial_t H_t + D_R \partial_{\theta}^2 H_t\big>\dd t\\
        &\quad  -D_T\sum_{i=1,2}\int_{0}^T \big( I_i^1(t,\etah_t)-I_i^2(t,\etah_t) \big) \dd t 
        +o_N(1),
    \end{aligned}
\end{equation}
where 
\[ 
    \begin{aligned}
        I_i^1(t,\etah_t)=&\frac{1}{N^2}\sum_{x\in \TN} \tau_x 
        \Big[ 
        d_s(\rho_{\varepsilon N}) \rho^{a,\omega}_{\varepsilon N}  \partial_{i,N}^2 G_t(x/N)\\
        &\qquad \qquad +2\big( d_s(\rho_{\varepsilon N}) \rho^{a,\omega \lambda_i }_{\varepsilon N} 
        +\rho^{a,\omega}_{\varepsilon N}\rho^{a, \lambda_i }_{\varepsilon N} s(\rho_{\varepsilon N})
        \big)\partial_{i,N} G_t(x/N)\Big],
    \end{aligned}
\]
and
\[ 
I_i^2(t,\etah_t)=\frac{1}{N^2}\sum_{x\in \TN} \tau_x 
\big( N D^{\varepsilon N}_0 \del_i \rho_{\varepsilon N} \partial_{i,N} G_t(x/N) \big).
\]
Each term in $I^1_i$ has already been 
written in a function of empirical measures, and hence only waits taking limit.  
However, for $I^2_i$, some treatments are needed to 
utilize the regularity of $\rho\in H^1$ to absorb the remaining $N$.  
This task is done in the following lemma.
\begin{lemma}\label{lemma:I2}
 \[ 
\limsup_{\varepsilon\to 0}\limsup_{N\to \infty} 
\big| I_i^2(t,\etah) - \widetilde{I_i^2}(t,\etah) \big| = 0,\quad \text{uniformly in } \etah,
\]
where 
\[ 
    \widetilde{I_i^2}(t,\etah) = \frac{1}{N^2}\sum_{x\in \TN} 
    D_x^{\varepsilon N}
    \frac{\tau_{x+\varepsilon^3 N e_i} \rho_{\varepsilon N}-\tau_{x-\varepsilon^3 N e_i} \rho_{\varepsilon N}}
    {2\varepsilon^3}\partial_{i,N} G_t(x/N).
\]
\end{lemma}
The proof of Lemma~\ref{lemma:I2} follows similar arguments as those on page 130-131 of~\cite{Erignoux21}. 
For brevity, we omit the detailed proof here. 
\subsection{Conversion to function of empirical measures}\label{subsec:funcofem}
We now wish to express the right hand side of~\eqref{eq:mgtilde2}  
as an explicit function of empirical measures. To this end, we define a family of mollifiers 
$(\varphi_{\varepsilon})$ on $\TT$:
\[ 
\varphi_{\varepsilon}(\cdot)=\frac{1}{(2\varepsilon)^2}\mathbbm{1}_{[-\varepsilon,\varepsilon]^2}(\cdot).
\]
Furthermore, for any angular function $\omega$, and any $u\in \TT$, 
we define $\varphi^{\omega}_{\varepsilon,u}:\TT \times \SSS \to \mathbb{R}$ by 
\[ 
    \varphi^{\omega}_{\varepsilon,u}(v,\theta)=\varphi_{\varepsilon}(u-v)\omega(\theta).
\]
Specially, when the angular function is the constant function $\omega\equiv 1$, 
we omit the subscript and write $\varphi_{\varepsilon,u}\coloneqq \varphi_{\varepsilon,u}^1$. 

With these notations, we can therefore write: 
\begin{equation}\label{eq:funcofem}
    \begin{aligned}
        \tau_x \rho^\sigma_{\varepsilon N} &= 
        \big< \varphi_{\varepsilon,x/N} , \pi^{\sigma,N} \big>+o_N(1), \\
        \tau_x \rho^{\sigma,\omega}_{\varepsilon N} &= 
        \big< \varphi_{\varepsilon,x/N}^\omega, \pi^{\sigma,N} \big>+o_N(1).
    \end{aligned}
\end{equation}
Now we deduce from~\eqref{eq:mgtilde2},~\eqref{eq:funcofem} and 
Lemma~\ref{lemma:I2}  that 
\begin{equation}\label{eq:funcofem2}
    \begin{aligned}
        \widetilde{M}^{a,H,N,\varepsilon}_T=&\big< \pi^{a}_T , H_T \big> - \big< \pi^{a}_0 , H_0 \big>-
        \int_{0}^T \big< \pi^{a}_t ,\partial_t H_t + D_R \partial_{\theta}^2 H_t\big>\dd t\\
        &-D_T\sum_{i=1,2}\int_{0}^T \frac{1}{N^2}\sum_{x\in \TN} \widetilde{d}_{\varepsilon,x/N}(\pi_t^a, \pi_t^p) \partial_{i,N}^2 G_t(x/N) \dd t\\
        &-2D_T\sum_{i=1,2}\int_{0}^T \frac{1}{N^2}\sum_{x\in \TN} \widetilde{R}_{\varepsilon,x/N}(\pi_t^a, \pi_t^p) \partial_{i,N} G_t(x/N) \dd t\\
        &+D_T\sum_{i=1,2}\int_{0}^T \frac{1}{N^2}\sum_{x\in \TN} \widetilde{D}_{\varepsilon,x/N}(\pi_t^a, \pi_t^p) \partial_{i,N} G_t(x/N) \dd t
        + o_{N,\varepsilon}(1),  
    \end{aligned}
\end{equation}
where we define  
\[ 
\begin{aligned}
    \widetilde{d}_{\varepsilon,x/N}(\pi_t^a, \pi_t^p)\coloneqq&
    d_s\big(\big< \varphi_{\varepsilon,x/N} , \pi_t \big>\big)
    \big< \varphi_{\varepsilon,x/N}^{\omega} , \pi^a_t \big>,\\
    \widetilde{R}_{\varepsilon,x/N}(\pi_t^a, \pi_t^p)\coloneqq&
    d_s\big(\big< \varphi_{\varepsilon,x/N} , \pi_t \big>\big)
    \big< \varphi_{\varepsilon,x/N}^{\omega\lambda_i} , \pi^a_t \big>\\
    &\quad +\big< \varphi_{\varepsilon,x/N}^{\omega} , \pi^a_t \big>
    \big< \varphi_{\varepsilon,x/N}^{\lambda_i} , \pi^a_t \big>
    s\big(\big< \varphi_{\varepsilon,x/N} , \pi_t \big>\big),\\
    \widetilde{D}_{\varepsilon,x/N}(\pi_t^a, \pi_t^p)\coloneqq&
    \big< \varphi_{\varepsilon,x/N}^{\omega} , \pi^a_t \big>\big[ 
     \mathcal{D}
    - d_s' \big]\big(\big< \varphi_{\varepsilon,x/N} , \pi_t \big>\big)\\
    &\quad \times \Big< \frac{\varphi_{\varepsilon,x/N+\varepsilon^3 e_i}
    -\varphi_{\varepsilon,x/N-\varepsilon^3 e_i}}
    {2\varepsilon^3} , \pi_t \Big>.
\end{aligned}
\]

Since $G$ is a smooth function, one can replace
in~\eqref{eq:funcofem2} 
the discrete spacial derivatives $\partial_{i,N}$ 
with the continuous derivative $\partial_{u_i}$, 
the sums $N^{-2}\sum_{x\in \TN}$ 
with the integral $\int_{\TT}\dd u$, 
and the variables $x/N$ 
with $u$ after paying an error of $o_N(1)$. 
Thus,~\eqref{eq:limittilde} and~\eqref{eq:funcofem2} imply that 
\begin{equation}\label{eq:finallimit}
    \limsup_{\varepsilon\to 0} \limsup_{N\to \infty}  Q^N\Big( \big| N^{a,H,\varepsilon}_T(\bm{\pi}^{[0,T]}) \big|\geq \delta \Big) =0,\quad \forall \delta>0, 
\end{equation}
where $N^{a,H,\varepsilon}_T(\cdot)$ is a continuous function on $\mathcal{M}^{[0,T]}\times \mathcal{M}^{[0,T]}$: 
\begin{equation}
    \begin{aligned}
        N^{a,H,\varepsilon}_T(\bm{\pi}^{[0,T]})=&\big< \pi^{a}_T , H_T \big> - \big< \pi^{a}_0 , H_0 \big>-
        \int_{0}^T \big< \pi^{a}_t ,\partial_t H_t + D_R \partial_{\theta}^2 H_t\big>\dd t\\
        &-D_T\sum_{i=1,2}\int_{0}^T \int_{\TT}
         \widetilde{d}_{\varepsilon,u}(\pi_t^a, \pi_t^p) \partial_{u_i}^2 G_t(u) \dd u \dd t\\
        &-2D_T\sum_{i=1,2}\int_{0}^T \int_{\TT} \widetilde{R}_{\varepsilon,u}(\pi_t^a, \pi_t^p) \partial_{u_i} G_t(u)\dd u \dd t\\
        &+D_T\sum_{i=1,2}\int_{0}^T \int_{\TT} \widetilde{D}_{\varepsilon,u}(\pi_t^a, \pi_t^p) \partial_{u_i} G_t(u)\dd u \dd t.
    \end{aligned}
\end{equation}

\subsection{Proof of Theorem~\ref{thm:main2}}\label{subsec:limit} 
Having completed all preparatory steps, in this section 
we gather all the previous results  to prove Theorem~\ref{thm:main2}. 

\begin{proof}[Proof of Theorem~\ref{thm:main2}] 
Now we have  balanced out all the factors $N$, it left to take the limit. 

  We proved in Lemma~\ref{lemma:compactreg} 
that the sequence of distribution $(Q^N)_{N\in \mathbb{N}}$ is relatively compact. 
Since $N^{a,H}(\cdot)$ is a continuous function 
for product Skorohod's topology on $\mathcal{M}^{[0,T]}\times \mathcal{M}^{[0,T]}$, 
we deduce from~\eqref{eq:finallimit} that 
for any weak limit $Q^*$ of $(Q^N)_{N\in \mathbb{N}}$, and any $\delta>0$, 
\begin{equation}
    \begin{aligned}
        \limsup_{\varepsilon \to 0}Q^*\big( \big|N^{a,H,\varepsilon}_T(\bm{\pi}^{[0,T]}) \big| >\delta \big)=0,
    \end{aligned}
\end{equation}
Therefore, it remains to consider the limit $\varepsilon\to 0$. 
As detailed on page 134 of~\cite{Erignoux21}, 
using the regularity results in Lemma~\ref{lemma:compactreg}, $Q^*$-a.s., we have  
\[ 
    \begin{aligned}
        \Big< \frac{\varphi_{\varepsilon,u+\varepsilon^3 e_i}
        -\varphi_{\varepsilon,u-\varepsilon^3 e_i}}
        {2\varepsilon^3} , \pi_t \Big>\to \partial_{u_i}\rho_t(u), 
        \quad \text{ in }L^2([0,T]\times \TT),
    \end{aligned}
\]
and for each $t\in [0,T]$ and a.e. $u\in \TT$, 
\[ 
    \begin{aligned}
        &\big< \varphi_{\varepsilon,u} , \pi_t \big>\to \rho_t(u), \\   
        &\big< \varphi_{\varepsilon,u}^\omega , \pi_t^a \big>\to 
        \int_{\SSS}\omega(\theta)\hat{f}^a(u, \dd \theta).
    \end{aligned}
\]
Recall we take test functions $H_t(u,\theta)=G_t(u)\omega(\theta)$, and  $\lambda_i(\theta)$ in~\eqref{eq:lambda}, we get 
\begin{equation}\label{eq:39}
    \begin{aligned}
        Q^*\Bigg( \Bigg| & \big< \pi^{a}_T , H_T \big> - \big< \pi^{a}_0 , H_0 \big>-
        \int_{0}^T \big< \pi^{a}_t ,\partial_t H_t + D_R \partial_{\theta}^2 H_t\big>\dd t\\
        &-D_T\sum_{i=1,2}\int_{0}^T \int_{\TT\times \SSS} \partial_{u_i}^2 H_t(u,\theta)
     d_s(\rho_t)\hat{f}^a(u,\dd \theta) \dd u \dd t\\
     &-2D_T\sum_{i=1,2}\int_{0}^T \int_{\TT\times \SSS} \partial_{u_i} H_t(u,\theta) 
      d_s(\rho_t) \lambda_i(\theta)  \hat{f}^a(u,\dd \theta) \dd u \dd t\\
     &-2D_T\sum_{i=1,2}\int_{0}^T \int_{\TT\times \SSS}  \partial_{u_i} H_t(u,\theta)
      s(\rho_t) \int_{ \SSS}\lambda_i(\vartheta )\hat{f}^a(u,\dd \vartheta ) 
      \hat{f}^a(u,\dd \theta) \dd u \dd t\\
      &+D_T\sum_{i=1,2}\int_{0}^T \int_{\TT\times \SSS} \partial_{u_i} H_t(u,\theta)\partial_{u_i} \rho_t(u)
      \big[ \mathcal{D}- d_s' \big] (\rho_t)
      \hat{f}^a(u,\dd \theta) \dd u \dd t
        \Bigg|>\delta \Bigg)=0. 
    \end{aligned}
\end{equation}
Therefore,~\eqref{eq:39} completes the proof of Theorem~\ref{thm:main2}  and, consequently,  completes the proof of Theorem~\ref{thm:main}.
\end{proof}

\section{Non-gradient replacement}\label{sec:replaceV}
The main goal of this section is to prove Lemmas~\ref{lemma:replaceV} and~\ref{lemma:replaceR}. 
In Section~\ref{subsec:replaceV} we give a proof of Lemma~\ref{lemma:replaceV} using a serious lemmas. 
In Section~\ref{subsec:limitingvariance}, we introduce a semi-norm $\iip[\cdot,\cdot]_{\bah}$.  
In Section~\ref{subsec:Hilbertspace}, 
we introduce a Hilbert space $\mathcal{H}_{\bah}^\omega$ on which we perform the non-gradient decomposition. 
Finally, in Section~\ref{sec:replaceR}, we prove Lemma~\ref{lemma:replaceR}.

\subsection{Proof of lemma~\ref{lemma:replaceV}}\label{subsec:replaceV}
Before proving Lemma~\ref{lemma:replaceV}, we introduce some notation and definitions. 

We define the spatial average of the function $\varphi$ over $B_l(x)$ as
\[ 
    \langle \varphi \rangle_x^l = \frac{1}{(2l+1)^2} \sum_{y \in B_l(x)} \tau_y \varphi. 
\]
For $l>0$, let  $l_r\coloneqq l-r-1$, and  $l'\coloneqq l-1$. 
For a cylinder function $f\in \mathcal{C}$, let $s_f$ denote the diameter of $\supp(f)$, i.e., $\supp(f)\subset B_{s_f}$, and  
define $l_f\coloneqq l-s_f-1$. The purpose of defining $l_f$ is to ensure that 
$\sum_{x\in B_{l_f}}\tau_x f$ is measurable with respect to the sites in $B_l$. 
Next, for $r>0$, we introduce the set 
\[ 
E_{r}\coloneqq \Bigg\{\sum_{x \in B_r} \eta_x \leq |B_r| - 2 \Bigg\}, 
\]
indicating that at least two sites in $B_r$ are empty. With this we write 
\[ 
    \rho^{a,\omega,r}_{l}=\frac{1}{|B_l|}\sum_{x\in B_l} \eta_x^{a,\omega,r} \mathbbm{1}_{E_{r,x}}, 
    \quad \text{ and }\quad \overline{\rho}^{a,\omega,r}_{l}=\rho^{a,\omega}_{l}-\rho^{a,\omega,r}_{l}
\]
where  $E_{r,x}=\tau_x E_r$. 
\begin{proof}[Proof of lemma~\ref{lemma:replaceV}]
Recall 
\[ 
    \mathcal{V}_i^{f,\varepsilon N}=N\big( j^{a,\omega}_i +  d_s(\rho_{\varepsilon N}) \bm{\delta}_i \rho^{a,\omega}_{\varepsilon N}
        + \rho^{a,\omega}_{\varepsilon N} \mathcal{D}(\rho_{\varepsilon N}) 
        \del_i \rho_{\varepsilon N} + \LL f \big). 
\]
To prove  lemma~\ref{lemma:replaceV}, 
we decompose $\mathcal{V}_i^{f,\varepsilon N}$ into four parts: 
\[ 
    \mathcal{V}_i^{f,\varepsilon N} = 
    N \big( \mathcal{W}_{i,1}^{f,l}+ \mathcal{W}_{i,2}^{\varepsilon N,r}
     +\mathcal{W}_{i,3}^{l,\varepsilon N,r}+\mathcal{W}_{i,4}^{f,l,r} \big)
\]
where we define: 
$$
\mathcal{W}_{i,1}^{f,l} = 
\big( j_i^{a,\omega} + \mathcal{L} f \big) -  
\Big( \big< j_i^{a,\omega} \big>_0^{l'} + \big< \mathcal{L} f \big>_0^{l_f} \Big)
$$
as the difference between $j_i^{a,\omega}+\LL f$ and their local average, and  
$$
\begin{aligned}
    \mathcal{W}_{i,2}^{\varepsilon N,r} =
    d_s(\rho_{\varepsilon N}) \bm{\delta}_i \overline{\rho}^{a,\omega,r}_{\varepsilon N}
\end{aligned}
$$
as the full cluster part, 
and 
$$
\begin{aligned}
    \mathcal{W}_{i,3}^{l,\varepsilon N,r} =& 
    \Big( d_s(\rho_{\varepsilon N}) \bm{\delta}_i \rho^{a,\omega,r}_{\varepsilon N}
            + \rho^{a,\omega}_{\varepsilon N} \mathcal{D}(\rho_{\varepsilon N}) 
            \del_i \rho_{\varepsilon N} \Big)\\
            &\quad -\Big( d_s(\rho_{l}) \bm{\delta}_i \rho^{a,\omega,r}_{l}
            + \rho^{a,\omega}_{l} \mathcal{D}(\rho_{l}) 
            \del_i \rho_{l} \Big)
\end{aligned}
$$
as the difference between the microscopic and mesoscopic gradients, and 
$$
\mathcal{W}_{i,4}^{f,l,r} = 
\big< j_i^{a,\omega} \big>_0^{l'} + \big< \mathcal{L} f \big>_0^{l_f}+d_s(\rho_{l}) \bm{\delta}_i \rho^{a,\omega,r}_{l}
+ \rho^{a,\omega}_{l} \mathcal{D}(\rho_{l}) 
\del_i \rho_{l}
$$
as the difference between local averaged currents and microscopic gradients. 

The following Lemmas~\ref{lemma:W1},~\ref{lemma:W2},~\ref{lemma:W3}, 
and~Lemma~\ref{lemma:W4} together conclude the proof of 
Lemma~\ref{lemma:replaceV}. 
\end{proof}

\begin{lemma}\label{lemma:W1}
    For any $G\in C^{1,2}([0,T]\times \TT)$,  $\omega\in C^2(\SSS)$, and any cylinder function $f\in \mathcal{C}$,   
    it holds that 
    \[ 
    \limsup_{l\to \infty}\limsup_{N\to \infty}\EN \Bigg[ \Bigg| 
    \int_{0}^T \frac{1}{N^2}\sum_{x\in \TN}G_t(x/N) \tau_x\big( N \mathcal{W}_{i,1}^{f,l} \big) \dd t
    \Bigg| \Bigg]=0,\quad i=1,2.
    \]
\end{lemma}
\begin{lemma}\label{lemma:W2}
    For any $G\in C^{1,2}([0,T]\times \TT)$,  $\omega\in C^2(\SSS)$, it holds that 
    \[ 
    \lim_{r\to \infty}\limsup_{\varepsilon \to 0}\limsup_{N\to \infty}\EN \Bigg[ \Bigg| 
    \int_{0}^T \frac{1}{N^2}\sum_{x\in \TN}G_t(x/N) \tau_x\big( N \mathcal{W}_{i,2}^{\varepsilon N, r} \big) \dd t
    \Bigg| \Bigg]=0,\quad i=1,2.
    \]
\end{lemma}
\begin{lemma}\label{lemma:W3}
    For any $G\in C^{1,2}([0,T]\times \TT)$,  $\omega\in C^2(\SSS)$, and $r>1$, it holds that 
    \[ 
        \limsup_{l\to \infty} \limsup_{\varepsilon \to 0}\limsup_{N\to \infty}\EN \Bigg[ \Bigg| 
    \int_{0}^T \frac{1}{N^2}\sum_{x\in \TN}G_t(x/N) \tau_x\big( N \mathcal{W}_{i,3}^{l,\varepsilon N,r} \big) \dd t
    \Bigg| \Bigg]=0,\quad i=1,2.
    \]
\end{lemma}
\begin{lemma}\label{lemma:W4}
    For any $G\in C^{1,2}([0,T]\times \TT)$,  $\omega\in C^2(\SSS)$,  $i=1,2$, it holds that 
    \[ 
       \inf_{f\in \mathcal{C}} \lim_{r\to \infty} \limsup_{l\to \infty} \limsup_{\varepsilon \to 0}\limsup_{N\to \infty}\EN \Bigg[ \Bigg| 
    \int_{0}^T \frac{1}{N^2}\sum_{x\in \TN}G_t(x/N) \tau_x\big( N \mathcal{W}_{i,4}^{f,l,r} \big) \dd t
    \Bigg| \Bigg]=0.
    \]
    Furthermore, there exists a family of cylinder functions $(f_{i,n}^\omega)_{n\in \mathbb{N}}$ 
    such that 
    \[ 
       \lim_{n\to \infty} \lim_{r\to \infty} \limsup_{l\to \infty} \limsup_{\varepsilon \to 0}\limsup_{N\to \infty}\EN \Bigg[ \Bigg| 
    \int_{0}^T \frac{1}{N^2}\sum_{x\in \TN}G_t(x/N) \tau_x\big( N \mathcal{W}_{i,4}^{f_{i,n}^\omega,l,r} \big) \dd t
    \Bigg| \Bigg]=0.
    \]
\end{lemma}

The proofs of Lemmas~\ref{lemma:W1},~\ref{lemma:W2}, and~\ref{lemma:W3} 
are similar to the proofs in Section 6.2 (Page 85), Section 6.3 (Page 85), and Section 6.4 (Page 93) in~\cite{Erignoux21}, 
with the minor modification that all $\eta^\omega$ are replaced by $\eta^{a,\omega}$, and using the 
fact that $\eta^a\leq \eta$ in every related  estimates. Thus, their proofs are omitted. 
Lemma~\ref{lemma:W4} is the main difficulty, and its proof is given in the next subsection. 

\subsection{Limiting variance}\label{subsec:limitingvariance}
To prove Lemma~\ref{lemma:W4}, we begin by introducing some necessary notations.   
Let $l$ be a fixed positive integer. Denote by $K^a$ and $K^p$ the total number of active and passive particles in $B_l$, 
satisfying $K^a+K^p\leq |B_l|$.  
Let $\Theta^a=\{\theta_1^a, \ldots, \theta_{K^a}^a\}$ and $\Theta^p=\{\theta_1^p, \ldots, \theta_{K^p}^p\}$ 
represent the sets 
of angles of active and passive particles in $B_l$. Now  
we define  the \textit{canonical state} in $B_l$ as:
\begin{equation}\label{eq:K}
    \widehat{K} = (K^a, \Theta^a ,K^p, \Theta^p ). 
\end{equation}
The set of 
all possible canonical states in $B_l$ with at least two unoccupied sites is denote by:  
   \begin{equation*}
     \widetilde{\mathbb{K}}_l = 
     \big\{  \widehat{K} = (K^a, \Theta^a ,K^p, \Theta^p ) \big|  K^a+K^p\leq |B_l|-2 \big\}.
   \end{equation*}
For a fixed $\widehat{K} \in \widetilde{\mathbb{K}}_l $, 
we define the hyperplane of configurations on $B_l$ with the canonical state $\widehat{K}$ as: 
\[ 
\Sigma_l^{\widehat{K}} = \Bigg\{ (\etah_x)_{x\in B_l} \Bigg| 
\sum_{x\in B_l}\eta^a_x \delta_{\theta_x}=\sum_{k=1}^{K^a}\delta_{\theta_k^a},\, 
\sum_{x\in B_l}\eta^p_x \delta_{\theta_x}=\sum_{k=1}^{K^p}\delta_{\theta_k^p} \Bigg\}. 
\]  
Recall that  $\LL_l$ defined in~\eqref{eq:LLS_N} is the SSEP generator $\LLS$ confined on $B_l$. The reason for leaving two unoccupied sites  
is to make the SSEP with state space $\Sigma_l^{\widehat{K}}$  irreducible. 
Consequently,    
it admits  a unique ergodic invariant measure $\mu_{l,\widehat{K}}$ 
which is referred to as the \textit{canonical measure}. More precisely, 
 let $\mu_{\ba,l}^*$ represent  the grand canonical measure $\mu^*_{\ba}$ confined on $B_l$. 
 Then 
$\mu_{l,\widehat{K}}$ is the conditional measure of $\mu_{\ba,l}^*$ on the hyperplane $\Sigma_l^{\widehat{K}}$: 
\begin{equation}\label{eq:mu}
        \mu_{l,\widehat{K}}\big(\cdot\big) = \mu_{\ba,l}^*\big( \cdot \big| \Sigma_l^{\widehat{K}}\big).
\end{equation}
In the rest of this paper, we will denote the inner product in $L^2(\mu_{l,\widehat{K}})$ 
by $<\! \cdot , \cdot\! >_{l,\widehat{K}}$. 

By a similar argument as presented in Section 6.5 of~\cite{Erignoux21}, we can reduce the proof of 
Lemma~\ref{lemma:W4} to the proof of the following variance bound: 
\begin{equation}\label{eq:limitW}
    \inf_f \lim_{r \to \infty} \limsup_{l \to \infty} \sup_{\widehat{K} \in \widetilde{\mathbb{K}}_l } 
    (2l+1)^2
    \Big< (-\LL_l)^{-1} \mathcal{W}_{i,4}^{f,l,r}  , \mathcal{W}_{i,4}^{f,l,r} \Big>_{l,\widehat{K}} = 0,
\end{equation}
where we recall that 
$$
\mathcal{W}_{i,4}^{f,l,r} = 
\big< j_i^{a,\omega} \big>_0^{l'} + \big< \mathcal{L} f \big>_0^{l_f}+d_s(\rho_{l}) \bm{\delta}_i \rho^{a,\omega,r}_{l}
+ \rho^{a,\omega}_{l} \mathcal{D}(\rho_{l}) 
\del_i \rho_{l}.
$$

To proceed, for any $\Kh \in \Kt_l $, we define the grand canonical parameter 
$\bah_{\Kh}=(\ah^a_{\Kh}, \ah^p_{\Kh})$ where 
\begin{equation}\label{eq:bah}
    \ah^\sigma_{\Kh}=\frac{1}{|B_l|}\sum_{k=1}^{K^\sigma}\delta_{\theta_k^\sigma}\in  
    \mathcal{M}_1(\SSS), 
\end{equation}
and for any configuration $\etah$ on $B_l$, we  define 
$\bm{\hat{\rho}_l}=(\hat{\rho}_l^a,\hat{\rho}_l^p)$ where 
\[ 
    \hat{\rho}_l^\sigma(\etah) =  \frac{1}{(2l + 1)^2} \sum_{x\in B_l} \eta^\sigma_x \delta_{\theta_x}\in  
    \mathcal{M}_1(\SSS).
\]
Note that $\bm{\hat{\rho}_l}(\etah)\equiv \bah_{\hat{K}} $ for all $\etah\in \Sigma_l^{\widehat{K}}$, and 
\begin{equation}\label{eq:rhoh}
    \rho^\sigma_l=\int_{\SSS} \hat{\rho}_l^\sigma(\dd \theta)=\frac{1}{(2l + 1)^2} \sum_{x\in B_l} \eta^\sigma_x =
 \frac{K^\sigma}{(2l + 1)^2}  ,
\end{equation}
and also  
\begin{equation}\label{eq:rhoa}
    \rho^{a,\omega}_{l} = \frac{1}{(2l + 1)^2} \sum_{x\in B_l} \eta^a_x\omega(\theta_x)=
    \E_{\bah_{\hat{K}}}\big[ \eta^{a,\omega}_0 \big],
\end{equation}
where $\E_{\bah_{\hat{K}}}$ denotes  the expectation with respect to the 
grand canonical measure $\mu_{\bah_{\hat{K}}}$. 
Using~\eqref{eq:rhoa} and~\eqref{eq:rhoh}, for $\etah\in \Sigma_l^{\widehat{K}}$, 
we can write
\[ 
    \rho^{a,\omega}_{l} \mathcal{D}(\rho_{l}) = \mf[d](\bah_{\hat{K}}), 
\]
where we define  
\begin{equation}\label{eq:defofd}
    \mf[d](\bah) \coloneqq \E_{\bah}\big[ \eta^{a,\omega}_0  \big]\mathcal{D}(\alpha),  \quad \forall \bm{\hat{\alpha}}\in \mathbb{M}_1(\SSS).  
\end{equation} 

Now, following the similar argument as presented on page 110 of~\cite{Erignoux21}, to prove~\eqref{eq:limitW}, it suffices to show that 
\begin{equation}\label{eq:l}
    \inf_f  \limsup_{l \to \infty} \sup_{\widehat{K} \in \widetilde{\mathbb{K}}_l } 
    \frac{1}{(2l+1)^2}
    \Bigg< (-\LL_l)^{-1} \sum_{x\in B_{l_f}}\tau_x V^{f,l}_i( \bah_{\hat{K}})  , 
    \sum_{x\in B_{l_f}}\tau_x V^{f,l}_i( \bah_{\hat{K}}) \Bigg>_{l,\widehat{K}} = 0,
\end{equation}
where 
\[ 
V^{f,l}_i(\bah) = j_i^{a,\omega} + d_s(\alpha) \bm{\delta}_i \eta^{a,\omega}
+ \mf[d](\bah) \bm{\delta}_i \eta_0 + \LL f.
\]
\begin{remark}
    In the definition of  $\mathcal{W}_{i,4}^{f,l,r}$, the term  $\del_i \rho^{a,\omega,r}_{l}$ can be written as 
    $$\del_i \rho^{a,\omega,r}_{l} = \frac{1}{(2l+1)^2}\sum_{x\in B_l} \tau_x \del_i (\eta_0^{a,\omega} \mathbbm{1}_{E_r}). $$  
    Here in the definition of $V^{f,l}_i$, for clarity reason, 
    we replace  $\del_i (\eta_0^{a,\omega} \mathbbm{1}_{E_r})$  with $\del_i\eta^{a,\omega} $
    incurring an error term $\del_i (\eta_0^{a,\omega} \mathbbm{1}_{E_r^c})$. 
    However, this error is negligible in the limit, 
    for a detailed treatment of this negligible error, we refer readers to Corollary 6.6.5 on page 111 of~\cite{Erignoux21}. 
\end{remark}

To prove~\eqref{eq:l}, we start with introducing some notations. Recall 
that the set of cylinder functions is denoted by $\mathcal{C}$.  
Let $\mathcal{C}_0$ be the space of cylinder functions with mean $0$ with respect to
any canonical measure 
\[ 
\mathcal{C}_0 = \Big\{ \psi \in \mathcal{C} \big| \E_{s_\psi, \Kh}\big[ \psi \big] = 0 ,
\quad \forall \Kh \in \Kt_{s_\psi} \Big\}.
\]
where $s_\psi$ denote the diameter of $\supp(\psi)$. 
Recall that for any angular function $\Phi$ on $S$, the type-$\sigma$ symmetric current associated with $\Phi$ is give by  
$$j_i^{\sigma,\Phi} = \Phi(\theta_0) \eta_0^\sigma (1 - \eta_{e_i}) - \Phi(\theta_{e_i}) \eta_{e_i}^\sigma (1 - \eta_0).$$
Specially,  $j^\sigma=j_i^{\sigma,1}$ by taking $\Phi\equiv 1$. 
We define $J^*$ as the set of linear combinations of all such currents: 
\begin{equation}
J^* =  \Span \Big\{ j_i^{a,\Phi_i^a}, j_i^{p,\Phi_i^p} \Big|  \Phi_i^a,  \Phi_i^p \in C^2(\SSS),i=1,2 \Big\}\subset \mathcal{C}_0.
\end{equation}

Next, we introduce a function class that is rich enough to contain all functions of our interest, 
while being small enough to ensure a spectral gap estimate of order $n^{-2}$ still holds. 
Such function class is defined as   
\begin{equation}\label{eq:To}
    T^\omega = \Span \Bigg\{   f(\etah)=\sum_{x\in \mathbb{Z}^2}
    \big(a\eta^a_x+b\eta^p_x+ c\etaao_x + d\etapo_x  \big)F_x(\eta)\Big| 
     f \in \mathcal{C}, F_x \in \mathcal{S}, \forall x  \Bigg\}, 
\end{equation}
where $\mathcal{S}$ is the set of functions that are both angle-blind and type-blind, i.e. ITA-everywhere functions 
defined in Remark~\ref{rk:2}. 

We now restrict  $\mathcal{C}_0$ to the function class  $T^\omega$ and define 
$$\mathcal{T}_0^\omega = \mathcal{C}_0 \cap T^\omega.$$  
Similarly, we restrict $J^*$ to $T^\omega$ and define
\begin{align*}
J^\omega = J^* \cap T^\omega = \big\{ 
    u^\top j^a+v^\top j^p + m^\top j^{a,\omega }+ n^\top j^{p,\omega }\mid 
      u,v, m, n \in \mathbb{R}^2 \big\}\subset \mathcal{T}_0^\omega.
\end{align*}
where $j^\sigma=[j^\sigma_1,\, j^\sigma_2]^\top$ and $j^{\sigma,\omega}=[j^{\sigma,\omega}_1,\, j^{\sigma,\omega}_2]^\top$. 

Next our goal is to define a semi-norm on the space $\mathcal{T}_0^\omega + J^*+\LL \mathcal{C}$, 
to 
achieve this goal we first define $\iip[\cdot]_{\bah}$ on the set $J^*+\LL \mathcal{C}$. 
\begin{definition}\label{def:iip}
    \normalfont
 For any $\psi = j^* + \LL g \in J^*+\LL \mathcal{C}$, 
 where $j^*= \sum_{i=1,2}(j_i^{a,\Phi_i^a}+ j_i^{p,\Phi_i^p}) \in J^*$, 
  $g\in \mathcal{C}$, and any $\varphi\in \mathcal{T}_0^\omega$, 
 we define 
 \begin{enumerate}
    \item $<\!\!<\!\cdot\!>\!\!>_{\bah}$ on $J^*+\LL \mathcal{C}$: 
\begin{equation}\label{eq:JandLC}
    \iip[j^*+\LL g]_{\bah}=
    \sum_{\sigma\in \{a,p\}}\sum_{i=1,2} \mathbb{E}_{\bah} \big[ 
        \eta^\sigma_0(1-\eta_{e_i})(\Phi^\sigma_i(\theta_0) + \widetilde{\nabla}_{0,e_i}\Sigma_g  )^2\big].
\end{equation}   
    \item $<\!\!<\!\cdot,\, \cdot\!>\!\!>_{\bah}$ on $\mathcal{T}_0^\omega\times (J^*+\LL \mathcal{C})$: 
    \begin{equation}\label{eq:bilinear}
        \iip[\varphi,\, j^*+\LL g ]_{\bah} 
        = -\E_{\bah }\Big[\varphi \big( \Sigma_g + \sum_{i=1,2} \sum_{x \in \mathbb{Z}^2 } x_i(\eta_x^{a,\Phi^a_i}+\eta_x^{p,\Phi^p_i}) \big)\Big].
    \end{equation}
    \item $<\!\!<\!\cdot\!>\!\!>_{\bah}$ on $\mathcal{T}_0^\omega$:
    \begin{equation}
        \iip[\varphi]_{\bah} = \sup_{\substack{g\in \mathcal{T}_0^\omega \\ j \in J^\omega } }
        \Big\{ 2 \iip[\varphi,\, j+\LL g]_{\bah} - \iip[ j+\LL g]_{\bah} \Big\}.
    \end{equation}
    \item $<\!\!<\!\cdot\!>\!\!>_{\bah}$ on $\mathcal{T}_0^\omega + J^*+\LL \mathcal{C}$: 
    \begin{equation}
        \iip[ \varphi + j^*+\LL g  ]_{\bah} = \iip[ \varphi  ]_{\bah} 
        + \iip[  j^*+\LL g  ]_{\bah}
        +2 \iip[ \varphi ,\, j^*+\LL g  ]_{\bah}.
    \end{equation}
    \item $<\!\!<\!\cdot,\, \cdot\!>\!\!>_{\bah}$ on $(\mathcal{T}_0^\omega + J^*+\LL \mathcal{C})^2$: we use the polarization identity.
\end{enumerate}  
\end{definition}

To proceed, we define a metric on the set of grand-canonical parameters $\mathbb{M}_1(\SSS)$. 
Let $B^*$ be the unit ball in $(C^2(\SSS),\|\cdot\|_{W^{2,\infty}})$, 
and we endow a metric $d$ on $\mathbb{M}_1(\SSS)$ by 
\begin{equation}\label{eq:d}
    d (\bah, \bah')=\sup_{f\in B^*} \Bigg\{ \int_{\SSS} f\, \dd \big( \ah^a - \ah'^a \big)  \Bigg\}+
    \sup_{g\in B^*} \Bigg\{ \int_{\SSS} g\, \dd \big( \ah^p -  \ah'^p  \big)  \Bigg\}.
\end{equation}
Next, we present the main theorem of this section, whose proof can be found in Section~\ref{sec:main3}.

\begin{theorem}\label{thm:main3}
    Fix $\bah \in  \mathbb{M}_1(\SSS)$, 
    and a sequence $(\widehat{K}_l)_{l \in \mathbb{N}} $ such that $\widehat{K}_l \in \widetilde{\mathbb{K}}_l$ and 
    $d(\bah_{\widehat{K}_l}, \bah)\to 0$, 
    where $\bah_{\widehat{K}_l}$ is defined in~\eqref{eq:bah}. 
    For any functions $\psi, \varphi \in \mathcal{T}_0^\omega + J^* + \mathcal{LC}$,
    \begin{equation}
        \lim_{l \to \infty} \frac{1}{(2l+1)^2}
        \Big< (-\mathcal{L}_l^{-1}) \sum_{x \in B_{l,\psi}} \tau_x \psi,\, \sum_{x \in B_{l,\varphi}} \tau_x 
        \varphi\Big>_{l,\widehat{K}_l} = 
        \iip[\psi, \varphi]_{\bah}.
    \end{equation}with the convergence being uniform in $\bah$.
    Furthermore, 
    the map $\bah \mapsto \iip[\psi, \varphi]_{\bah}$ is continuous 
    in $\bah$. 
    In particular, for any $\psi \in \mathcal{T}_0^\omega + J^* + \mathcal{LC}$,
    \begin{equation}\label{eq:main3:2}
        \lim_{l \to \infty} \sup_{\widehat{K} \in \widetilde{\mathbb{K}}_l} \frac{1}{(2l+1)^2} 
        \Big< 
        (-\mathcal{L}_l^{-1}) \sum_{x \in B_{l,\psi}} \tau_x \psi,\,  \sum_{x \in B_{l,\psi}} \tau_x \psi
        \Big>_{l,\widehat{K}_l}  = \sup_{\bah \in \mathcal{M}_1(S)} \iip[\psi]_{\bah}.
    \end{equation}
\end{theorem}
Thanks to Theorem~\ref{thm:main3}, after similar argument on page 110 of~\cite{Erignoux21}, we can reduce  
the proof of~\eqref{eq:l} to showing that 
\begin{equation}\label{eq:main3}
    \inf_{f\in \mathcal{C}}  \sup_{\bah \in \mathbb{M}_1(\SSS) } 
   \iip[ j_i^{a,\omega} + d_s(\alpha) \bm{\delta}_i \eta^{a,\omega}
+ \mathfrak{d}(\bah) \del_i \eta_0 + \LL f]_{\bah} = 0, 
\end{equation}
which is proved in the next subsection. 
\subsection{\texorpdfstring{Hilbert space $\mathcal{H}_{\bah}^\omega$}{Hilbert space }}\label{subsec:Hilbertspace}
In the previous sections,  we mainly considered a relatively large space 
$\mathcal{T}_0^\omega + J^* + \mathcal{LC}$  for computational convenience. 
Next we focus on its subspace $\mathcal{T}_0^\omega$, in which most functions of interest in our model reside. 
Denote by $\mathcal{N}_{\bah}= \Ker \iip[\cdot]_{\bah}$ the null space.  
We define the Hilbert space  $\mathcal{H}_{\bah}^\omega$ by the completion of 
$\mathcal{T}_0^\omega/\mathcal{N}_{\bah}$ with respect to  $\iip[\cdot]_{\bah}^{1/2}$.

The first  result  of this section is the following Proposition~\ref{prop:structure}, which describes that 
$\mathcal{H}_{\bah}^\omega$ is the completion of $\LL \To / \mathcal{N}_{\bah}+ J^\omega$, i.e., 
all elements of $\mathcal{H}_{\bah}^\omega$ can be approximated by 
\[ 
    u^\top j^{a} +v^\top j^{p}+m^\top j^{a,\omega}+n^\top j^{p,\omega}+\LL g
\]
where $u,v,m,n\in \mathbb{R}^2$ and $g\in \mathcal{T}_0^\omega$. 
Since the proof of Proposition~\ref{prop:structure} follows the similar 
arguments in the proofs of Proposition 7.5.2 in~\cite{kipnis1999scaling} or  Proposition 6.6.6 in~\cite{Erignoux21}, 
here we omit it for brevity.  
\begin{proposition}[Structure of $\mathcal{H}_{\bah}^\omega$]\label{prop:structure}
    For any $\bah\in \mathbb{M}_1(\SSS)$, let 
    $\overline{\LL \To / \mathcal{N}_{\bah}}$ be the closure of $\LL \To / \mathcal{N}_{\bah}$ 
    with respect to $\iip[\cdot]_{\bah}^{1/2}$, we have the following direct sum decomposition
    \[ 
        \mathcal{H}_{\bah}^\omega = \overline{\LL \To / \mathcal{N}_{\bah}} \oplus J^\omega.
    \]
\end{proposition}
\begin{corollary}\label{cor:structure}
    For each $f\in \To$, there exists unique vectors $u,v,m,n\in \mathbb{R}^2$ such that 
    \[ 
    f + \big( u^\top j^{a} +v^\top j^{p}+m^\top j^{a,\omega}+n^\top j^{p,\omega} \big) \in \overline{\LL \To }.
\]
In other words, 
\begin{equation}\label{eq:structure}
    \inf_{g\in \To} \iip[f + u^\top j^{a} +v^\top j^{p}+m^\top j^{a,\omega}+n^\top j^{p,\omega} + \LL g  ]_{\bah} = 0.
\end{equation}
\end{corollary}

Now we define the set linear generated by local gradients as:   
\[ 
    \text{Grad}^\omega =  \big\{ 
        u^\top \grad^a+v^\top \grad^p + m^\top \grad^{a,\omega }+ n^\top \grad^{p,\omega };
         \quad u,v, m, n \in \mathbb{R}^2 \big\}
\]
where 
\[ 
\grad^\sigma=[\del_1 \eta^\sigma,\, \del_2 \eta^\sigma]^\top, \quad \text{and} \quad 
\grad^{\sigma,\omega}=[\del_1 \eta^{\sigma,\omega},\, \del_2 \eta^{\sigma,\omega}]^\top. 
\]
In the rest of this section, for any fixed $\bah \in \mathbb{M}_1(\SSS)$, and for any angular function $\omega$, 
we shorten:  
\[ 
\ob^\sigma = \E_{\bah}\big[ \omega(\theta_0) \big| \eta^\sigma_0=1  \big],\quad \ohw^\sigma=\omega - \ob^\sigma,\quad 
V^\sigma_{\bah}(\omega)=\E_{\bah}\big[ (\omega(\theta_0) 
- \ob^\sigma)^2 \big| \eta^\sigma_0=1  \big].
\]
We also define  the \textit{centered currents} $j^{\sigma,\oh^\sigma}_i$ and \textit{centered local gradients} $\del_i \eta^{\sigma,\oh^\sigma}$ as follows: 
$$
j^{\sigma,\oh^\sigma}_i=j^{\sigma,\omega}_i-\ob^\sigma j^a_i,\quad 
\del_i \eta^{\sigma,\oh^\sigma}=\del_i \eta^{\sigma,\omega}-\ob^\sigma \del_i \eta^\sigma. 
$$ 
By changing basis, the set of local currents $J^\omega$  
can be written as:   
\[ 
J^\omega =\Span \big\{ \jaoh_i , \jpoh_i, j_i^a, j_i^p\mid  i=1,2 \big\}.
\]
Similarly, the set of local gradients $\text{Grad}^\omega$ can be written as:  
\[ 
\text{Grad}^\omega = 
\Span \big\{ \del_i\etaaob, \del_i\etapob, \del_i\eta^a, \del_i\eta^p\mid  i=1,2  \big\}. 
\]

In the following, for any vectors \( g, j \in (\mathcal{H}_{\bah}^\omega)^2 \), 
their inner product is understood as a $2\times2$ matrix:
\[
\iip[g, j]_{\bah} \coloneqq 
\begin{bmatrix}
    \iip[g_1, j_1]_{\bah} & \iip[g_1, j_2]_{\bah} \\
    \iip[g_2, j_1]_{\bah} & \iip[g_2, j_2]_{\bah}
\end{bmatrix}\in \mathbb{R}^{2\times2}.
\]
The following 
Lemma~\ref{lemma:ortho} tells us that the currents $j^a$, $j^p$, $\jaoh$, and $\jpoh$  
form an orthogonal basis of 
$J^\omega$. Lemma~\ref{lemma:gradJ} 
provides some important inner products between local gradients and currents. 
Their proofs can be found in Appendix~\ref{asec:inner}.   
\begin{lemma}[Orthogonality of currents]\label{lemma:ortho}
    For any $\bah\in \mathbb{M}_1(\SSS)$ and $\omega$, the  currents $j^a$, $j^p$, $\jaoh$, and $\jpoh$
     are mutually orthogonal in $\mathcal{H}_{\bah}^\omega$.  Furthermore, 
        \begin{align*}
            &\iip[j^a, j^a ]_{\bah}=\alpha^a(1-\alpha)I_2,&&\iip[\jaoh, \jaoh ]_{\bah}=\alpha^a(1-\alpha) V^a_{\bah}(\omega)I_2,\\
            &\iip[j^p, j^p ]_{\bah}=\alpha^p(1-\alpha)I_2,&&\iip[\jpoh, \jpoh ]_{\bah}=\alpha^p(1-\alpha) V^p_{\bah}(\omega)I_2.
        \end{align*}
\end{lemma}

\begin{lemma}[Inner products associated with $\text{Grad}^\omega$ and $J^\omega$]\label{lemma:gradJ}
    For any $\bah\in \mathbb{M}_1(\SSS)$ and $\omega$, it holds that
        \begin{align*}
            &\iip[\grad^a, j^a ]_{\bah}=-\alpha^a(1-\alpha^a)I_2, 
               & &\iip[\grad^a, \jaoh ]_{\bah}=0,\\
            &\iip[\grad^a, j^p ]_{\bah}=\alpha^a\alpha^pI_2, 
               & &\iip[\grad^a, \jpoh ]_{\bah}=0,\\
            &\quad\\
            &\iip[\grad^p, j^a ]_{\bah}=\alpha^a\alpha^pI_2,
               & &\iip[\grad^p, \jaoh ]_{\bah}=0,\\
            &\iip[\grad^p, j^p ]_{\bah}=-\alpha^p(1-\alpha^p)I_2, 
               & &\iip[\grad^p, \jpoh ]_{\bah}=0,\\
            &\quad\\
            &\iip[\grad^{a,\oha}, j^a ]_{\bah}=0,
            &&\iip[\grad^{a,\oha}, \jaoh ]_{\bah}=- \alpha^a V^a_{\bah}(\omega)I_2,\\
            &\iip[\grad^{a,\oha}, j^p ]_{\bah}=0,
            &&\iip[\grad^{a,\oha}, \jpoh ]_{\bah}=0,\\
            &\quad\\
            &\iip[\grad^{p,\ohp}, j^a ]_{\bah}=0,
            &&\iip[\grad^{p,\ohp}, \jaoh ]_{\bah}=0,\\
            &\iip[\grad^{p,\ohp}, j^p ]_{\bah}=0,
            &&\iip[\grad^{p,\ohp}, \jpoh ]_{\bah}=- \alpha^p V^p_{\bah}(\omega)I_2.
        \end{align*}
\end{lemma}

In the following proposition, we project currents $J^\omega$ onto $\text{Grad}^\omega+\overline{\LL \To}$. 
\begin{proposition}[Decomposition of currents]\label{prop:currents}
    For each fixed $\bah \in \mathbb{M}_1(\SSS)$ 
    it holds that for $i=1,2$  
    \[ 
      \begin{aligned}
        &\inf_{g\in \To}\iip[
            j^{a}_i 
        + d_s(\alpha)\del_i \eta^a + \alpha^a \mathcal{D}(\alpha)\del_i \eta
        +\LL g]_{\hat{\alpha}}=0,\\
        &\inf_{g\in \To}\iip[
            j^{p}_i 
            +d_s(\alpha)\del_i \eta^p + \alpha^p \mathcal{D}(\alpha)\del_i \eta
        +\LL g]_{\hat{\alpha}}=0,\\
        &\inf_{g\in \To}\iip[
            j^{a,\oha}_i 
        + d_s(\alpha) \del_i \eta^{a,\oha}
        +\LL g]_{\hat{\alpha}}=0,\\
        &\inf_{g\in \To}\iip[
            j^{p,\ohp}_i 
        + d_s(\alpha) \del_i \eta^{p,\ohp}
        +\LL g]_{\hat{\alpha}}=0.
      \end{aligned}
    \]
    In other words, 
    \[ 
        \begin{aligned}
            &
            j^{a}_i 
            + d_s(\alpha)\del_i \eta^a + \alpha^a \mathcal{D}(\alpha)\del_i \eta
            \in \overline{\LL \To},  \\
            &
            j^{p}_i 
            +d_s(\alpha)\del_i \eta^p + \alpha^p \mathcal{D}(\alpha)\del_i \eta
            \in \overline{\LL \To},\\
            &
                j^{a,\oha}_i 
            + d_s(\alpha) \del_i \eta^{a,\oha}
            \in \overline{\LL \To},\\
            &
                j^{p,\ohp}_i 
            + d_s(\alpha) \del_i \eta^{p,\ohp}\in \overline{\LL \To}.
          \end{aligned}
    \]
\end{proposition}
\begin{proof}
    First we clear out the trivial cases. When 
    one of the cases $\alpha^a=0$ or $\alpha^p=0$ occurs, the problem is reduced to the one type particle model
    which was done in~\cite{Erignoux21}. When one of the cases $\alpha^a=1$ or $\alpha^p=1$ occurs, 
    all related quantities vanish.  When $V^\sigma_{\bah}(\omega)=0$, 
    the centered currents $j^{\sigma,\oh^\sigma}$ 
    and centered gradients $\grad^{\sigma,\oh^\sigma}$ all vanish, and hence have nothing to prove.  
    
Now we fix $\bah \in \mathbb{M}_1(\SSS)$ such that both $\alpha^a,\alpha^p \in (0,1)$, 
and fix angular function $\omega$ such that both $V^a_{\bah}(\omega),V^p_{\bah}(\omega)>0$.
We apply Corollary~\ref{cor:structure} to each normalized local gradients 
and get that in Hilbert space $\mathcal{H}_{\bah}^\omega$, 
\begin{equation}\label{eq:relation}
    \begin{aligned}
        &\del_i\eta^a + (\sum_{l=1,2} u^1_{i,l}j_l^a+ v^1_{i,l}j_l^p) + \LL g^1_i=0,\\
        &\del_i\eta^p + (\sum_{l=1,2} u^2_{i,l}j_l^a+ v^2_{i,l}j_l^p) + \LL g^2_i=0,\\
        &\del_i\etaaob+(\sum_{l=1,2}  m^3_{i,l}\jaoh_l ) + \LL g^3_i=0,\\
        &\del_i\etapob+(\sum_{l=1,2}  n^4_{i,l}\jpoh_l) + \LL g^4_i=0, 
    \end{aligned}
\end{equation}  
Here we for convenience reason assume that the infimum in~\eqref{eq:structure} is reached.   
This assumption is purely for convenience, 
and we can substitute at any point to $g_i^j$ 
a sequence $(g_{i,m}^k)_{m\in \mathbb{N}}$ 
such that the previous identity holds in the limit $m\to \infty$. 

In order to rewrite~\eqref{eq:relation} in a compact way, we define the $2\times 2$ 
coefficient matrices 
\[ 
    U^i 
       =\begin{bmatrix}
        u^i_{1,1}&u^i_{1,2} \\ u^i_{2,1}&u^i_{2,2}
       \end{bmatrix},\quad 
       V^i 
       =\begin{bmatrix}
        v^i_{1,1}&v^i_{1,2} \\ v^i_{2,1}&v^i_{2,2}
       \end{bmatrix},\quad i =1,2;  
\]
$M^3,\,N^4\in \mathbb{R}^{2\times2}$ are defined in the same fashion. Moreover, for each $i=1,2,3,4$, 
we define the vectors 
$\LL\bm{g}^i = [ \LL g^i_1,\, \LL g^i_2 ]^\top \in (\mathcal{H}_{\bah}^\omega)^2$. 
With these definitions, we can rewrite~\eqref{eq:relation} as
\begin{equation}\label{eq:compact}
    \begin{aligned}
        &\grad^a 
           + U^1 j^a + V^1 j^p +
           \LL \bm{g}^1=0,\\
        &\grad^p
           + U^2 j^a + V^2 j^p+
           \LL \bm{g}^2=0,\\
        &\grad^{a,\oha}
           + M^3\jaoh +
            \LL \bm{g}^3=0,\\
        &\grad^{p,\ohp}
           + N^4\jpoh +
           \LL \bm{g}^4=0.\\
    \end{aligned}
\end{equation}
We  further vectorize~\eqref{eq:compact} as 
\begin{align*}
    \begin{bmatrix}
        \grad^a \\
        \grad^p \\
        \grad^{a,\oha} \\
        \grad^{p,\ohp}
    \end{bmatrix}+
    \begin{bmatrix}
        \Psi\\
        &M^3& \\
        & &N^4
    \end{bmatrix}
    \begin{bmatrix}
        j^a \\
        j^p \\
        j^{a,\oha} \\
        j^{p,\ohp}
    \end{bmatrix}+
    \begin{bmatrix}
        \LL \bm{g}^1 \\
        \LL \bm{g}^2 \\
        \LL \bm{g}^3 \\
        \LL \bm{g}^4
    \end{bmatrix}=0,
\end{align*}
where the $4\times 4$ coefficient matrix is defined by 
\[ 
    \Psi=\begin{bmatrix}
        U^1 & V^1 \\ U^2 & V^2
    \end{bmatrix}\in \mathbb{R}^{4\times 4}.
\]
In order to represents normalized currents in terms of normalized gradients, 
our next goal is to prove the matrices 
$\Psi$, $M^3$, and $N^4$ all are invertible. 

We first show that $\Psi$ is invertible. To this end, we introduce the following $2\times 2$ matrices 
       \begin{equation}\label{eq:L_J}
        L_J^{i,\sigma} = \iip[\LL\bm{g}^i,\, j^\sigma ]_{\bah},\quad 
        G^{\sigma \sigma'}=\iip[\grad^\sigma, \grad^{\sigma'}]_{\bah},\quad 
        L^{ij}=\iip[\LL \bm{g}^i, \LL \bm{g}^j].
       \end{equation}
By taking the inner product of each term in the 
first identity in~\eqref{eq:compact} with vectors 
$j^a$, $j^p$, $\grad^a$, and $\LL \bm{g}^1$ respectively, using Lemma~\ref{lemma:ortho} and~\ref{lemma:gradJ}, we derive that 
\begin{equation}\label{eq:compact2}
    \begin{aligned}
        &-\alpha^a(1-\alpha^a)I_2
        + U^1 \alpha^a(1-\alpha)
        +L_J^{1,a}=0, \\
        &\alpha^a\alpha^pI_2 
        + V^1 \alpha^p(1-\alpha) 
        +L_J^{1,p}=0,\\
        &G^{aa} - U^1\alpha^a(1-\alpha^a)+ V^1\alpha^a\alpha^p=0,\\
        &U^1 (L_J^{1,a})^\top + V^1 (L_J^{1,p})^\top + L^{11} =0.
    \end{aligned}
\end{equation}
Similarly, by taking the inner product of each term in the 
second identity in~\eqref{eq:compact} with vectors 
$j^a$, $j^p$, $\grad^p$, and $\LL \bm{g}^2$ respectively, we derive that 
\begin{equation}\label{eq:compact3}
    \begin{aligned}
        &\alpha^a\alpha^pI_2 
        + U^2 \alpha^a(1-\alpha)
        +L_J^{2,a} =0,\\
        &-\alpha^p(1-\alpha^p)I_2 
        + V^2 \alpha^p(1-\alpha)
        +L_J^{2,p}=0,\\
        &G^{pp} + U^2\alpha^a\alpha^p- V^2\alpha^p(1-\alpha^p)=0,\\
        &U^2 (L_J^{2,a})^\top + V^2 (L_J^{2,p})^\top + L^{22} =0.
    \end{aligned}
\end{equation}
Combing~\eqref{eq:compact2} and~\eqref{eq:compact3}, we write in a more compact form:   
\begin{equation}\label{eq:L_Jeq}
    \begin{aligned}
    &G_J + \Psi \begin{bmatrix}
        \alpha^a(1-\alpha)I_2 & 0 \\ 0 & \alpha^p(1-\alpha)I_2
    \end{bmatrix}
    + L_J = 0, \\
    &G  + \Psi G_J^\top =0, \\
    & \Psi L_J^\top + L=0.
    \end{aligned}
\end{equation}
where we define the following $4\times4 $ matrices
\[ 
G_J=-\begin{bmatrix}
    \alpha^a(1-\alpha^a)I_2 &-\alpha^a\alpha^pI_2 \\ -\alpha^a\alpha^pI_2 & \alpha^p(1-\alpha^p)I_2
\end{bmatrix},
\]
and use~\eqref{eq:L_J} to define
\begin{equation*}
    L_J=\begin{bmatrix}
        L_J^{1,a} &L_J^{1,p} \\ L_J^{2,a} & L_J^{2,p}
    \end{bmatrix},
    G = \begin{bmatrix}
        G^{aa} & G^{ap} \\ G^{pa} & G^{pp}
    \end{bmatrix},
    L = \begin{bmatrix}
        L^{11} & L^{12} \\ L^{21} & L^{22}
    \end{bmatrix}.
\end{equation*}
Since $-G_J\succ 0$ and $G \succeq 0$, the second line of~\eqref{eq:L_Jeq} 
implies $\Psi \succeq 0$. Next, the first and the third lines of~\eqref{eq:L_Jeq} together imply that  
\begin{equation}
    L=\Psi\Bigg( G_J + \Psi \begin{bmatrix}
        \alpha^a(1-\alpha)I_2 & 0 \\ 0 & \alpha^p(1-\alpha)I_2
    \end{bmatrix} \Bigg).
\end{equation}
The facts $L\succeq 0$ and $\Psi \succeq 0$ yield  that 
\[ 
    \Psi \succeq \begin{bmatrix}
    (\alpha^a(1-\alpha))^{-1}I_2 & 0 \\ 0 & (\alpha^p(1-\alpha))^{-1}I_2
\end{bmatrix}
\begin{bmatrix}
    \alpha^a(1-\alpha^a)I_2 &-\alpha^a\alpha^pI_2 \\ -\alpha^a\alpha^pI_2 & \alpha^p(1-\alpha^p)I_2
\end{bmatrix}\succ 0, 
\]
which shows $\Psi$ is invertible. 

Next we prove $M^3$ and $N^4$ are invertible. 
By taking the inner product of each term in the 
third identity in~\eqref{eq:compact} with vectors 
$\jaoh$, $\grad^{a,\oha}$, and $\LL \bm{g}^3$ respectively, using Lemma~\ref{lemma:ortho} and~\ref{lemma:gradJ}, 
we derive that 
\begin{equation}\label{eq:compact4}
    \begin{aligned}
        &-\alpha^a V^a_{\bah}(\omega)I_2
        + M^3 \alpha^a(1-\alpha) V^a_{\bah}(\omega)
        +L_J^{3,a,\omega}=0,\\
        &G^{aa,\omega} - M^3\alpha^a V^a_{\bah}(\omega)=0,\\
        &  M^3 (L_J^{3,a,\omega})^\top  + L^{33} =0,
    \end{aligned}
\end{equation}
where similar to~\eqref{eq:L_J} we define 
\[ 
    L_J^{i,\sigma,\omega} = \iip[\LL\bm{g}^i,\, j^{\sigma,\oh^\sigma} ]_{\bah},\quad 
    G^{\sigma\sigma',\omega}=\iip[\grad^{\sigma,\oh^\sigma}, \grad^{\sigma',\oh^{\sigma'}}]_{\bah}.
\]
The second line in~\eqref{eq:compact4} implies $M^3\succeq 0$. The first and the third lines in~\eqref{eq:compact4} 
give 
\[ 
   L^{33} = M^3 (-\alpha^a V^a_{\bah}(\omega)I_2
   + M^3 \alpha^a(1-\alpha) V^a_{\bah}(\omega))^\top. 
\]
The facts $L^{33} \succeq 0 $ and $M^3\succeq 0$ together give 
\[ 
    M^3 \succeq \frac{\alpha^a}{\alpha^a(1-\alpha)} I_2 \succ 0,
\]
which shows $M^3$ is invertible. By similar argument, $N^4$ is also invertible.

Now since $\Psi\in \mathbb{R}^{4\times 4}$, $M^3$, $N^4$ are all invertible, we can get that  
\begin{equation}\label{eq:compact5}
    \begin{aligned}
        \begin{bmatrix}
            j^a \\
            j^p \\
            j^{a,\oha} \\
            j^{p,\ohp}
        \end{bmatrix}=-
        \begin{bmatrix}
            \Psi\\
            &M^3& \\
            & &N^4
        \end{bmatrix}^{-1}
        \begin{bmatrix}
            \grad^a \\
            \grad^p \\
            \grad^{a,\oha} \\
            \grad^{p,\ohp}
        \end{bmatrix}-
        \begin{bmatrix}
            \Psi\\
            &M^3& \\
            & &N^4
        \end{bmatrix}^{-1}
        \begin{bmatrix}
            \LL \bm{g}^1 \\
            \LL \bm{g}^2 \\
            \LL \bm{g}^3 \\
            \LL \bm{g}^4
        \end{bmatrix}.
    \end{aligned}
\end{equation}
To complete the proof, our final task is to find   
these inverse matrices and give explicit expressions for the currents. 

By introducing $C^1,C^2\in \mathbb{R}^{2\times 4}$, $C^3,C^4\in \mathbb{R}^{2\times 2}$ such that    
$$
\Psi^{-1}=C=\begin{bmatrix}
    C^1\\C^2
\end{bmatrix},\quad
(M^3)^{-1}=C^3,\quad 
(N^4)^{-1}=C^4, 
$$
we can rewrite~\eqref{eq:compact5} as  
\begin{equation}\label{eq:compact6}
    \begin{aligned}
        &j^a = -
        C^1
        \begin{bmatrix}
            \grad^a \\
            \grad^p
        \end{bmatrix} -
        \LL \bm{\tilde{g}}^1, \\
        &j^p = -
        C^2
        \begin{bmatrix}
            \grad^a \\
            \grad^p
        \end{bmatrix} -
        \LL \bm{\tilde{g}}^2, \\
        &j^{a,\oha} = -
        C^3
        \grad^{a,\oha} -
        \LL \bm{\tilde{g}}^3, \\
        &j^{p,\ohp} = -
        C^4
        \grad^{p,\ohp} -
        \LL \bm{\tilde{g}}^4, 
    \end{aligned}
\end{equation}
where  
$$
\begin{bmatrix}
    \LL \bm{\tilde{g}}^1 \\
    \LL \bm{\tilde{g}}^2 
\end{bmatrix}=C
\begin{bmatrix}
    \LL \bm{g}^1 \\
    \LL \bm{g}^2 
\end{bmatrix},\quad 
\LL \bm{\tilde{g}}^3 = C^3 \LL \bm{g}^3,\quad
\LL \bm{\tilde{g}}^4 = C^4 \LL \bm{g}^4.
$$

First, we find the matrix $C$.  The first two lines of~\eqref{eq:compact6} yield that for any $u,v\in \mathbb{R}^2$
\begin{equation}\label{eq:compact7}
    \inf_{g\in \To} \iip[u^\top j^a + v^\top j^p +u^\top C^1
    \begin{bmatrix}
        \grad^a \\
        \grad^p
    \end{bmatrix} +
     v^\top C^2
     \begin{bmatrix}
         \grad^a \\
         \grad^p
     \end{bmatrix}+\LL g ]=0.
\end{equation}
Taking the inner product of~\eqref{eq:compact7} with $u^\top j^a + v^\top j^p + \LL g$, 
and since local gradients are orthogonal to any $\LL g$, we get 
\begin{equation}\label{eq:compact8}
\begin{aligned}
    &\inf_{g\in \To}\iip[u^\top j^{a} +v^\top j^{p}+\LL g]_{\hat{\alpha}}\\
    =&-\iip[u^\top j^{a} +v^\top j^{p}, 
    u^\top C^1
    \begin{bmatrix}
        \grad^a \\
        \grad^p
    \end{bmatrix} +
        v^\top C^2
        \begin{bmatrix}
            \grad^a \\
            \grad^p
        \end{bmatrix}]_{\hat{\alpha}}\\
        =& \begin{bmatrix}
        u^\top & v^\top
    \end{bmatrix}
    (-G_J)^\top C^\top 
    \begin{bmatrix}
        u \\ v
    \end{bmatrix}. 
\end{aligned} 
\end{equation}
Moreover, thanks to Proposition~\ref{prop:cross}, we also have 
\begin{equation}\label{eq:compact9}
        \inf_{g\in \To}\iip[u^\top j^{a} +v^\top j^{p}+\LL g]_{\hat{\alpha}}=\begin{bmatrix}
            u^\top & v^\top
        \end{bmatrix} M 
        \begin{bmatrix}
            u\\v
        \end{bmatrix},
\end{equation}
where $M$ is the mobility matrix
\begin{equation}\label{eq:mobilitymatrix}
    M = \begin{bmatrix}
    \frac{\alpha^a\alpha^p}{\alpha}d_s(\alpha)I_2+\frac{(\alpha^a)^2(1-\alpha)}{\alpha} I_2 & - \frac{\alpha^a\alpha^p}{\alpha}d_s(\alpha)I_2+\frac{\alpha^a\alpha^p(1-\alpha)}{\alpha} I_2  \\
    - \frac{\alpha^a\alpha^p}{\alpha}d_s(\alpha)I_2+\frac{\alpha^a\alpha^p(1-\alpha)}{\alpha} I_2 & \frac{\alpha^a\alpha^p}{\alpha}d_s(\alpha)I_2+\frac{(\alpha^p)^2(1-\alpha)}{\alpha} I_2 
    \end{bmatrix}. 
\end{equation}
By defining $h'' \coloneqq G_J^{-1}$,~\eqref{eq:compact8} together with~\eqref{eq:compact9} yield   
\[ 
C = Mh'' = 
\begin{bmatrix}
\big( \frac{\alpha^p}{\alpha}d_s(\alpha)+\frac{\alpha^a}{\alpha} \big) I_2 & 
\frac{\alpha^a}{\alpha}(1-d_s(\alpha)) I_2  \\
\frac{\alpha^p}{\alpha}(1-d_s(\alpha)) I_2  &  
\big( \frac{\alpha^a}{\alpha}d_s(\alpha)+\frac{\alpha^p}{\alpha} \big) I_2
\end{bmatrix}.
\] 
Therefore, by defining $\grad = \grad^a+\grad^p$,~\eqref{eq:compact6} implies
\[ 
\begin{aligned}
    j^a =&- \Big( \frac{\alpha^p}{\alpha}d_s(\alpha)+\frac{\alpha^a}{\alpha} \Big)\grad^a - 
    \frac{\alpha^a}{\alpha}(1-d_s(\alpha))\grad^p - \LL \bm{\tilde{g}}^1\\
    =&-d_s(\alpha)\grad^a - \alpha^a\mathcal{D}(\alpha)\grad - \LL \bm{\tilde{g}}^1,  
\end{aligned}
\]
and  similarly
\[ 
    j^p = -d_s(\alpha)\grad^p - \alpha^a\mathcal{D}(\alpha)\grad  - \LL \bm{\tilde{g}}^2.
\]

Next, we find $C^3$ and $C^4$.  
The third line of~\eqref{eq:compact6} yields that for any $m \in \mathbb{R}^2$
\[ 
\inf_{g\in \To} \iip[m^\top j^{a,\oha}
+ 
 m^\top C^3
 \grad^{a,\oha} + \LL g ]=0.
\]
Similar arguments as in~\eqref{eq:compact8} and~\eqref{eq:compact9} give 
\begin{align*}
    &m^\top 
    \alpha^aV^a_{\bah}(\omega) d_s(\alpha) I_2
    m
    =\inf_{g\in \To}\iip[m^\top j^{a,\oha} + \LL g]_{\bah}\\
    =&-\iip[m^\top j^{a,\oha}, 
    m^\top C^3 \grad^{a,\oha}]_{\bah}\\
     =& m^\top
     \alpha^a V^a_{\bah}(\omega) (C^3)^\top 
    m, 
\end{align*} 
which implies that $C^3 = d_s(\alpha)I_2 $, 
and similarly $C^4 = d_s(\alpha) I_2$. Therefore,~\eqref{eq:compact6} implies 
\[ 
    \begin{aligned}
      &
          j^{a,\oha}_i =
      - d_s(\alpha) \del_i \eta^{a,\oha}+
      -\LL \bm{\tilde{g}}^3,\\
      &
          j^{p,\ohp}_i= 
      - d_s(\alpha) \del_i \eta^{p,\ohp}
      -\LL \bm{\tilde{g}}^4.
    \end{aligned}
  \]
\end{proof}
\begin{corollary}For each fixed $\bah \in \mathbb{M}_1(\SSS)$ 
    it holds that for $i=1,2$  
\begin{equation}\label{eq:compact10}
            \inf_{g\in \To}\iip[j_i^{a,\omega}+ d_s(\alpha) \del_i \eta^{a, \omega}  + 
            \mathfrak{d}(\bah) \del_i \eta + \LL g]_{\bah}=0.
\end{equation}
\end{corollary}
\begin{proof}
    Again,  we assume the infimum in Proposition~\ref{prop:currents} is attained:  
    \[ 
        \begin{aligned}
             &j^{a,\oha}_i =
          - d_s(\alpha) \del_i \eta^{a,\oha}
          -\LL g^1_i,\\
          &j^{a}_i =-d_s(\alpha)\del_i \eta^a - \alpha^a \mathcal{D}(\alpha) \del_i \eta-\LL g^2_i.
        \end{aligned}
      \]
    Then, note that $\oba\alpha^a = \E_{\bah}[\eta^{a,\omega}_0]$, 
    and recall the definition of $\mathfrak{d}$ in~\eqref{eq:defofd},  we obtain     
    \begin{align*}
        j_i^{a,\omega}=&\jaoh_i + \oba  j_i^a\\
         =& -d_s(\alpha) \del_i \eta^{a, \oha}  -\oba d_s(\alpha) \del_i \eta^a - 
         \oba \alpha^a \mathcal{D}(\alpha) \del_i \eta
         -\LL (g^1_i+\oba g^2_i)\\
         =& -d_s(\alpha) \del_i \eta^{a, \omega}  - 
         \oba \alpha^a \mathcal{D}(\alpha) \del_i \eta
         -\LL (g^1_i+\oba g^2_i)\\
         =& -d_s(\alpha) \del_i \eta^{a, \omega}  - 
         \mathfrak{d}(\bah) \del_i \eta
         -\LL (g^1_i+\oba g^2_i), 
    \end{align*}
    which concludes the proof of the corollary. 
\end{proof}
With all above preparations, we obtain  the following lemma which proves~\eqref{eq:main3} 
and constructs the sequence $(f_{i,n}^\omega)_{n\in \mathbb{N}}$. 
Consequently, Lemma~\ref{lemma:uniform} concludes the proof of Lemma~\ref{lemma:W4}. 
\begin{lemma}\label{lemma:uniform}
    There exists a sequence of cylinder functions $(f_{i,n}^\omega)_{n\in \mathbb{N}}$ such that 
    \[ 
    \lim_{n\to \infty} \sup_{\bah \in \mathbb{M}_1(\SSS)} 
    \iip[j_i^{a,\omega}+ d_s(\alpha) \del_i \eta^{a, \omega}  + 
    \mathfrak{d}(\bah) \del_i \eta + \LL f_{i,n}^\omega ]_{\bah}=0.
    \]
\end{lemma}
The proof of this lemma employs a compactness argument analogous to those 
 in the proof of Lemma 6.6.10 on page 120 of~\cite{Erignoux21} or  
Theorem 7.5.6 on page 176 of~\cite{kipnis1999scaling}, we omit the details for brevity. 
The only adaptation  required  is the use of the compactness of 
 $(\mathbb{M}_1(\SSS),d)$. 
This compactness issue is addressed in Proposition~\ref{prop:compact}. 

\subsection{Proof of Lemma~\ref{lemma:replaceR}}\label{sec:replaceR}
\begin{proof}[Proof of Lemma~\ref{lemma:replaceR}]
  Recall that 
\[ 
    \mathcal{R}_i^{f^{\omega}_{i,n},\varepsilon N} = r_i^\omega + D_T^{-1}\LLWA f^{\omega}_{i,n} - 
    2\big( d_s(\rho_{\varepsilon N}) \rho^{a,\omega \lambda_i }_{\varepsilon N} 
    +\rho^{a,\omega}_{\varepsilon N}\rho^{a, \lambda_i }_{\varepsilon N} s(\rho_{\varepsilon N}) \big).
\]
  By adding and subtracting $\E_{\bm{\hat{\rho}}_{\varepsilon N}}[ r^\omega_i + D_T^{-1} \LLWA f ]$,  
  we can split  $\mathcal{R}_i^{f^{\omega}_{i,n}}$ as follows: 
  \[ 
    \mathcal{R}_i^{f^{\omega}_{i,n},\varepsilon N}=X_1^{f^{\omega}_{i,n}, \varepsilon N } 
    +X_2^n(\bm{\hat{\rho}}_{\varepsilon N}),
  \]
  where 
  \[ 
  \begin{aligned}
    &X_1^{f,l} =  r_i^\omega + D_T^{-1}\LLWA f -
    \E_{\bm{\hat{\rho}}_{l}}[ r^\omega_i + D_T^{-1} \LLWA f ], \\
    &X_2^n(\bah) = \E_{\bah} [ r^\omega_i + D_T^{-1} \LLWA f^{\omega}_{i,n} ] -
     2\big( d_s(\alpha)\E_{\bah}[\eta_0^{a,\lambda_i \omega}] +\E_{\bah}[\eta_0^{a,\lambda_i}] 
    \E_{\bah}[\eta_0^{a,\omega}] s(\alpha) \big).
  \end{aligned}
  \]
Therefore, the following Lemma~\ref{lemma:X1} and~\ref{lemma:X2} together conclude the proof of Lemma~\ref{lemma:replaceR}.  
\end{proof}
\begin{lemma}[Replacement lemma]\label{lemma:X1}
    For any cylinder function $g$, $T>0$, and continuos $G\in C([0,T]\times \TT)$, 
    let $\mathcal{W}^{l}(\etah)\coloneqq  g(\etah) - \E_{\bm{\hat{\rho}}_{l}}[ g ]$, it holds that 
    \[ 
        \lim_{n\to \infty} \limsup_{\varepsilon\to 0} \limsup_{N\to \infty} \EN \Bigg[ \Bigg| 
    \int_{0}^T \frac{1}{N^2}\sum_{x\in \TN}G_t(x/N)\tau_x \mathcal{W}^{\varepsilon N}(\etah(t)) \dd t
    \Bigg| \Bigg]=0. 
    \]
\end{lemma}
This version of replacement lemma is relatively classical.  
Similar results can be widely found  in various hydrodynamic limit literatures. 
The proof of Lemma~\ref{lemma:X1} follows the usual strategy, 
which relies on the One-block and the Two-block estimates.  
For example, see Lemma 5.1.10 on page 77 of~\cite{kipnis1999scaling} or 
 Corollary 4.1.2 on page 50 of~\cite{Erignoux21}. 
However it requires several adaptations, which are addressed in Appendix~\ref{asec:replaceR}. 


\begin{lemma}\label{lemma:X2}
    Let $(f_{i,n}^\omega)_{n\in \mathbb{N}}$ be as defined in Lemma~\ref{lemma:replaceV}, 
     For $i=1,2$, we have the following limit  
    \[ 
    \lim_{n\to \infty}  
   \E_{\bah} \big[ r^\omega_i + D_T^{-1} \LLWA f^{\omega}_{i,n} \big] =
    2\Big( d_s(\alpha)\E_{\bah}\big[\eta_0^{a,\lambda_i \omega}\big] +\E_{\bah}\big[\eta_0^{a,\lambda_i}\big] 
    \E_{\bah}\big[\eta_0^{a,\omega}\big] s(\alpha) \Big),
    \]
    and the convergence is uniform in $\bah \in \mathbb{M}_1(\SSS)$. 
\end{lemma}
\begin{proof}[Proof of Lemma~\ref{lemma:X2}]
First we compute $\mathbb{E}_{\bah}\big[ r_i^\omega + D_T^{-1}\LLWA f \big]$ for any $\bah \in \mathbb{M}_1(\SSS)$.  

Recall the definition of $r_i^\omega$ in~\eqref{eq:r}, 
\[ 
    r^{\omega}_i = \eta_0^{a,\omega\lambda_i}(1-\eta_{e_i})+\eta_{e_i}^{a,\omega\lambda_i}(1-\eta_{0}). 
\]
We can calculate that 
\[ 
\mathbb{E}_{\bah}\big[ r_i^\omega \big]=
2\overline{\lambda_i \omega}^a \alpha^a(1-\alpha)=2\iip[j^{a,\lambda_i}_i,j^{a,\omega}_i]_{\bah}. 
\]
For any cylinder function $f$, using the translation invariance of $\mu_{\bah}$ and Definition~\ref{def:iip},  
\begin{align*}
    &\mathbb{E}_{\bah}\big[ D_T^{-1} \LLWA f \big]=
    \mathbb{E}_{\bah}\Bigg[ \sum_x \sum_{i=1,2}(\lambda_i(\theta_x)\eta_x^a(1-\eta_{x+e_i})-
    \lambda_i(\theta_{x+e_i})\eta_{x+e_i}^a(1-\eta_{x}))\widetilde{\nabla}_{x,x+e_i}f \Bigg]\\
    =&\sum_{i=1,2} \sum_x \mathbb{E}_{\bah}\Big[ \tau_x j_i^{a,\lambda_i} \widetilde{\nabla}_{x,x+e_i}f \Big]=
    -2\sum_{i=1,2} \mathbb{E}_{\bah}\Big[ j_i^{a,\lambda_i} \Sigma_f \Big] =
    2\iip[j_1^{a,\lambda_1}+j_2^{a,\lambda_2}, \LL f ]_{\bah}, 
\end{align*}
Combing, we obtain 
\[ 
    \mathbb{E}_{\bah}\big[ r_i^\omega + D_T^{-1}\LLWA f \big] = 
     2\iip[j^{a,\lambda_i}_i,j^{a,\omega}_i ]_{\bah}+2 \iip[j_1^{a,\lambda_1}+j_2^{a,\lambda_2},\LL f] _{\bah}.
\]
By Lemma~\ref{lemma:uniform}, 
we have the following limit uniformly in $\bah\in \mathbb{M}_1(\SSS)$, 
\[ 
    \lim_{n\to \infty} \LL f_{i,n}^\omega =- j_i^{a,\omega}- d_s(\alpha) \del_i \eta^{a, \omega}  - 
            \oba\alpha^a \mathcal{D}(\alpha) \del_i \eta, \quad \text{ in }   \mathcal{H}_{\bah}^\omega.
\]
Thus, uniformly in $\bah\in \mathbb{M}_1(\SSS)$, we have the limit 
\[ 
\begin{aligned}
        \lim_{n\to \infty}\E_{\bah}\big[ r_i^\omega + D_T^{-1}\LLWA f^{\omega}_{i,n} \big] 
        =& -2 \iip[j_1^{a,\lambda_1}+j_2^{a,\lambda_2},\,
        j^{a,\omega}_i+
        d_s(\alpha) \del_i \eta^{a, \omega}+
        \oba \alpha^a \mathcal{D}(\alpha) \del_i \eta]_{\bah} \\
        &\qquad + 2\iip[j^{a,\lambda_i}_i,j^{a,\omega}_i ]_{\bah}.
\end{aligned}
\]
Using Definition~\ref{def:iip}, we can calculate that 
\begin{align*}
    &\iip[j_k^{a,\lambda_k}, \bm{\delta}_i \eta^{a, \omega} ]_{\bah}
    = \delta_{i,k}\big( -\E_{\bah}\big[\eta_0^{a,\lambda_i \omega}\big] + \E_{\bah}\big[\eta_0^{a,\lambda_i}\big]\E_{\bah}\big[\eta_0^{a,\omega}\big] \big),\\
    &\iip[j_k^{a,\lambda_k},\del_i \eta ]_{\bah}=\delta_{i,k}\big( -\E_{\bah}\big[\eta_0^{a,\lambda_i}\big](1-\alpha) \big), 
\end{align*}
with which we obtain 
\begin{align*}
    &-2\sum_{k=1,2} \Big( \iip[j_k^{a,\lambda_k},\,
    j^{a,\omega}_i+
    d_s(\alpha) \del_i \eta^{a, \omega}+
    \oba \alpha^a \mathcal{D}(\alpha) \del_i \eta]_{\bah}  \Big) 
    + 2\iip[j^{a,\lambda_i}_i,j^{a,\omega}_i ]_{\bah}\\
    =&-2d_s(\alpha) \iip[j_i^{a,\lambda_i}, \bm{\delta}_i \eta^{a, \omega} ]_{\bah} 
    -2\oba \alpha^a \mathcal{D}(\alpha)  \iip[j_i^{a,\lambda_i},\bm{\delta}_i \eta ]_{\bah}\\
    =&2d_s(\alpha)\E_{\bah}\big[\eta_0^{a,\lambda_i \omega}\big] +2\E_{\bah}\big[\eta_0^{a,\lambda_i}\big] 
    \E_{\bah}\big[\eta_0^{a,\omega}\big] \big( (1-\alpha)\mathcal{D}(\alpha)-d_s(\alpha) \big)\\
    =&2d_s(\alpha)\E_{\bah}\big[\eta_0^{a,\lambda_i \omega}\big] +2\E_{\bah}\big[\eta_0^{a,\lambda_i}\big] 
    \E_{\bah}\big[\eta_0^{a,\omega}\big] s(\alpha).  
\end{align*}

\end{proof}

\section{Proof of Theorem~\ref{thm:main3}}\label{sec:5}
The main goal of this section is to prove Theorem~\ref{thm:main3}. 
In Section~\ref{sec:main3}, we prove Theorem~\ref{thm:main3} based on a serious lemmas. 
In Section~\ref{subsec:decomposition}, we introduce a decomposition of the set of germs of closed form based on a spectral gap estimate. 
In  Section~\ref{sec:Spectral gap}, we prove the spectral gap estimate. 
\subsection{Proof of Theorem~\ref{thm:main3}}\label{sec:main3}
\begin{proof}[Proof of Theorem~\ref{thm:main3}]
It suffices to prove 
\begin{equation}\label{eq:8.49}
    \lim_{l \to \infty} \frac{1}{(2l+1)^2} 
    \Big< 
    (-\mathcal{L}_l)^{-1} \sum_{x \in B_{l,\psi}} \tau_x \psi \cdot \sum_{x \in B_{l,\varphi}} \tau_x \varphi
    \Big>_{l,\widehat{K}_l} = 
    \iip[\psi, \varphi]_{\bah}
\end{equation}
in three cases:
\begin{enumerate}
    \item $\varphi = \psi \in \mathcal{LC} + J^*$.
    \item $\varphi = \psi \in \mathcal{T}_0^\omega$.
    \item $\varphi \in \mathcal{T}_0^\omega$ and $\psi \in \mathcal{LC} + J^*$. 
\end{enumerate}
This task is done by Lemmas~\ref{lemma:8.52},~\ref{lemma:8.51}, and~\ref{lemma:8.50} separately.
\begin{lemma}\label{lemma:8.52}
    Fix 
    $\psi = \mathcal{L}g + \sum_{i=1,2}( j_i^{a,\Phi^a_i} +j_i^{p,\Phi^p_i}  ) \in \mathcal{LC} + J^*$. 
    For any sequence $(\widehat{K}_l)$ such that $\bah_{\widehat{K}_l} \rightarrow \bah$, it holds that 
    \begin{equation}\label{eq:8.50}
        \begin{aligned}
        &\lim_{l \to \infty} \frac{1}{(2l+1)^2} \Big< 
        (- \mathcal{L}_l)^{-1} \sum_{x \in B_{l,\psi}} \tau_x \psi,\,\sum_{x \in B_{l,\psi}} \tau_x \psi
        \Big>_{l,\widehat{K}_l} = \iip[\psi]_{\bah},
        \end{aligned}
    \end{equation}
    where $\iip[\psi]_{\bah}$ is defined in Definition~\ref{def:iip}.
\end{lemma}
\begin{lemma}\label{lemma:8.51}
    Fix $\varphi \in \To$ and 
    $\psi = \mathcal{L}g + \sum_{i=1,2}( j_i^{a,\Phi^a_i} +j_i^{p,\Phi^p_i}  ) \in \mathcal{LC} + J^*$. 
    For any sequence $(\widehat{K}_l)$ such that $\bah_{\widehat{K}_l} \rightarrow \bah$, it holds that 
    \begin{equation}\label{eq:8.51}
        \begin{aligned}
        \lim_{l \to \infty} \frac{1}{(2l + 1)^2}
        \Big<
        (- \mathcal{L}_l)^{-1}\sum_{x \in B_{l,\psi}} \tau_x \psi,\, \sum_{x \in B_{l,\varphi}} \tau_x \varphi
        \Big>_{l,\widehat{K}_l}=\iip[\psi, \varphi]_{\bah},
        \end{aligned}
    \end{equation}
    where $\iip[\psi, \varphi]_{\bah}$ is defined in Definition~\ref{def:iip}.
\end{lemma}
\begin{lemma}\label{lemma:8.50}
    Fix $\varphi \in \To$. For any sequence $(\widehat{K}_l)$ such that $\bah_{\widehat{K}_l} \rightarrow \bah$, it holds that 
    \begin{equation}\label{eq:8.500}
        \begin{aligned}
        \lim_{l \to \infty} \frac{1}{(2l + 1)^2}
        \Big<
        (- \mathcal{L}_l)^{-1}\sum_{x \in B_{l,\varphi}} \tau_x \varphi,\, \sum_{x \in B_{l,\varphi}} \tau_x \varphi
        \Big>_{l,\widehat{K}_l}=\iip[\varphi]_{\bah},
        \end{aligned}
    \end{equation}
    where $\iip[\varphi]_{\bah}$ is defined in Definition~\ref{def:iip}.
\end{lemma}
Using the same technique as used in the proof of Lemma 8.5.1 on page 174 of~\cite{Erignoux21}, 
the proof of Lemmas~\ref{lemma:8.52} and~\ref{lemma:8.51} are relatively direct, and we postpone them to Appendix~\ref{appsec:main3}.
However, the proof of Lemma~\ref{lemma:8.50} is the main difficulty, 
and we will present it in the next Subsection~\ref{subsec:decomposition}.  

In order to complete the proof of Theorem~\ref{thm:main3}, 
we still need to prove that the convergence is uniform in $\bah$, to prove~\eqref{eq:main3:2}. Let us denote
$$
V_{l,\psi,\phi}(\bah_{\widehat{K}_l}) = 
\frac{1}{(2l+1)^2}  \Big< -\mathcal{L}_l^{-1} \sum_{x \in B_{l,\psi}} \tau_x \psi , \sum_{x \in B_{l,\phi}} \tau_x \phi \Big>_{l,\widehat{K}_l},
$$
and let us extend smoothly the domain of definition of $V_{l,\psi,\phi}$ to $\mathbb{M}_1(\SSS)$ which is compact see Proposition~\ref{prop:compact}. 
The three previous lemmas yield that $V_{l,\psi,\phi}(\bah_{\widehat{K}_l})$ converges to $\iip[\psi, \phi]_{\bah}$ 
as $l\to \infty$ for any sequence $\widehat{K}_l$ such that $d(\bah_{\widehat{K}_l}, \bah)\to 0$. In particular, 
$\lim_{l \to \infty} V_{l,\psi,\phi}(\bah_{l})=\iip[\psi, \phi]_{\bah}$ for any sequence $\bah_{l}\to \bah$. 
This implies that $\iip[\cdot]_{\bah}$ is continuous, 
and $V_{l,\psi,\phi}(\bah)$ converges uniformly in $\bah$ towards $\iip[\psi, \phi]_{\bah}$ as $l\to \infty$. 
This, combined with Lemmas~\ref{lemma:8.52}, \ref{lemma:8.51} and \ref{lemma:8.50}, complete the proof of Theorem~\ref{thm:main3}.
\end{proof}

\subsection{Decomposition of closed forms}\label{subsec:decomposition}
We introduce in this section the concept of discrete differential forms in the context
 of particle systems. A key point of the non-gradient method is that any translation
invariant closed form can be decomposed as the sum of a gradient of a translation
invariant function and the current forms.  

Recall that  
\[
\Sigma_\infty = \Bigg\{ \etah=(\etat_x, \theta_x)_{x\in \mathbb{Z}^2}\in 
\big( \{1^a,1^p,\zerot  \} \times \SSS \big)^{\mathbb{Z}^2}
\Big| \theta_x=0 \text{ if } \etat_x=\zerot  
\Bigg\}.
\]
We consider the graph $\mathcal{G} = (\Sigma_\infty, E)$ with oriented edge set $E$:  
$$
E \coloneqq \Big\{ 
    (\etah, \etah^{x,x+z}) \in \Sigma_\infty^2; \; 
     \text{ for some } x \in \mathbb{Z}^2, |z| = 1 \text{ and } \eta_x (1 - \eta_{x+z}) = 1 
    \Big\}.
$$
In other words, 
there is an edge from $\etah$ to $\etah'$ if and only if $\etah'$ can be transformed  
from $\etah$ by exerting exactly one particle jump.  
We endow $\mathcal{G}$ with the usual distance $d_g$ on graphs, i.e., 
$d_g(\widehat{\eta}, \widehat{\eta}')$ is the minimal number of particle jumps necessary to go from one 
configuration to the other. 
If $\etah'$ and $\etah$ is not connected, 
we will adopt the usual convention $d_g(\widehat{\eta}, \widehat{\eta}') = \infty$. 

We define \textit{differential form} on $(\mathcal{G}, d)$ as a collection of $L^2(\mu_{\bah})$ 
functions  $ (\mathfrak{u}_{x,x+z})_{x \in \mathbb{Z}^2, |z|=1}$, satisfying
$$
\mathfrak{u}_{x,x+z}(\widehat{\eta}) = \eta_x (1 - \eta_{x+z}) \mathfrak{u}_{x,x+z}(\widehat{\eta}),  
$$
i.e., $\mathfrak{u}_{x,x+z}(\widehat{\eta})=0$  if the jump from $x$ to $x+z$ cannot 
be performed in configuration $\widehat{\eta}$. 

We define \textit{path} as a finite sequence of nearest-neighbor jumps 
sites $\gamma=(x_i,x_i+z_i)_{i=1}^{n_\gamma}$. 
For any path $\gamma=(x_i,x_i+z_i)_{i=1}^{n_\gamma}$, and configuration $\etah$, by denoting  
$\etah^{(0,\gamma)}=\etah$, and $\etah^{(i+1,\gamma)}=(\etah^{(i,\gamma)})^{x_i,x_i+z_i}$,   
we define 
\[ 
    \Gamma(\etah)\coloneqq \Big\{ \gamma=(x_i,x_i+z_i)_{i=1}^{n_\gamma} \mid 
    \etah_{x_i}^{(i,\gamma)}(1-\etah_{x_i+z_i}^{(i,\gamma)})=1, \quad i=1,\ldots,n_\gamma \Big\}
\]
the set of 
\textit{licit paths}, i.e., paths such that 
starting from  $\etah$, all the nearest-neighbor jumps on the path can be executed one by one.   
We define  
\[ 
    \Gamma_c(\etah)\coloneqq \Big\{ \gamma=(x_i,x_i+z_i)_{i=1}^{n_\gamma}\in \Gamma(\etah)\mid  
    \big( \etah_{x_i}^{(n_\gamma,\gamma)} \big)^{x_{n_\gamma},x_{n_\gamma}+z_{n_\gamma}}  =\etah \Big\}, 
\]
the set of 
\textit{licit closed paths}, i.e.,  licit paths such that 
 after performing the last jump, the resulting configuration returns to to $\etah$. 
Given a differential form $\mathfrak{u}= (\mathfrak{u}_{x,x+z})_{x \in \mathbb{Z}^2, |z|=1}$, 
and a path $\gamma=(x_i,x_i+z_i)_{i=1}^{n_\gamma}$, 
we define  \textit{the integral of $\mathfrak{u}$ along $\gamma$} as 
\[ 
I_{\gamma,\mathfrak{u}}(\etah) = \mathbbm{1}_{\gamma \in \Gamma(\etah)}\sum_{i=1}^{n_\gamma}
 \mathfrak{u}_{x_i,x_i+z_i}(\etah^{(i,\gamma)}).  
\]

With all above preparations, 
we call a differential form $\mathfrak{u}$ the \textit{closed form} if 
$I_{\gamma,\mathfrak{u}}(\etah)=0, \forall \gamma \in \Gamma_c(\etah)$, for all $\etah\in \Sigma_\infty$,  
meaning that the integral along any closed path vanishes. 
And we call a differential form $\mathfrak{u}$ 
the \textit{exact form} if there exists a cylinder function $f$ such that  
$\mathfrak{u}_{x,x+z}(\etah)=\eta_x(1-\eta_{x+z})(f(\etah^{x,x+z})-f(\etah))$, for all $\etah\in \Sigma_\infty$. 

We may now introduce the main object of this section.
\begin{definition}[Germs of closed forms]
    \normalfont
    A pair of $L^2(\mu_{\bah})$ functions 
    $\mathfrak{U}=(\mathfrak{u}_1,\, \mathfrak{u}_2):\Sigma_\infty \to \mathbb{R}^2$ is called a \textit{germ of a closed form}, 
    if the closed form $\mathfrak{u}= (\mathfrak{u}_{x,x+z})_{x \in \mathbb{Z}^2, |z|=1}$ defined by 
    \[ 
        \mathfrak{u}_{x,x+e_i}(\etah) = \tau_x \mathfrak{u}_i(\etah),\quad \text{and} \quad 
        \mathfrak{u}_{x,x-e_i}(\etah) = - \mathfrak{u}_{x-e_i,x}(\etah^{x,x-e_i}),  
    \]
    for \( i = 1, 2 \). 
\end{definition}
\begin{definition}[Germs of exact forms]
    \normalfont
    A pair of $L^2(\mu_{\bah})$ functions 
    $\mathfrak{U}=(\mathfrak{u}_1,\, \mathfrak{u}_2):\Sigma_\infty \to \mathbb{R}^2$ is called a
 \textit{germ of an exact form}, if there exists a cylinder function $h$ such that 
    \[ 
        \big(\mathfrak{u}_1,\, \mathfrak{u}_2\big) = \bm{\nabla} \Sigma_h\coloneqq 
        \big( \nabla_{0,e_1} \Sigma_h, \nabla_{0,e_2} \Sigma_h \big).
    \]
    Moreover, we denote the set of all germs of exact forms associated with cylinder functions by 
    $$\mathfrak{E}^* =  \big\{ \bm{\nabla}\Sigma_h \big| h \in \mathcal{C} \big\}.$$ 

\end{definition}
Note germs of exact forms are special class of germs of closed forms. Next we define another special class of germs of closed forms. 
\begin{definition}[Germs of current forms]
    \label{def:current_form}
    \normalfont
    For any angular function $\omega$, and $i=1,2$, 
     the \textit{germs of current forms} $\mf[j]^{\sigma,i}, \mf[j]^{\sigma,\omega,i}$  are defined as 
    \[ 
        \mf[j]^{\sigma,i}_k=\delta_{i,k} \eta_0^\sigma(1-\eta_{e_i})\quad \text{ and }\quad 
        \mf[j]^{\sigma,\omega,i}_k=\delta_{i,k} \eta_0^{\sigma,\omega}(1-\eta_{e_i}),\quad \forall k=1,2.
    \]
    Moreover, we denote the set of all germs of current forms by 
    $$\mf[J]^\omega=\Span\big\{ \mf[j]^{\sigma,\Phi_i^\sigma,i}\big| \Phi_i^\sigma\in C^2(\SSS),i=1,2 \big\}.$$
\end{definition}
We now define a linear mapping $\mathfrak{F}$, which establishes a connection between
$\mathcal{LC} + J^*$ and $\mathfrak{E}^*+\mathfrak{I}^*$, as follows: 
\begin{equation}\label{eq:mapping}
    \begin{aligned}
        \mathfrak{F} : \mathcal{LC} + J^* &\quad \to \quad \mathfrak{E}^*+\mathfrak{I}^*  \\
        \mathcal{L}f + \sum_{i=1,2} j_i^{a,\Phi_a^i} + j_i^{p,\Phi_p^i} &\quad \mapsto \quad 
    \bm{\nabla} \Sigma_f + \sum_{i=1,2} \mf[j]^{a,\Phi_i^a,i} + \mf[j]^{p,\Phi_i^p,i}, 
    \end{aligned}
\end{equation}
\begin{remark}
    The linear mapping defined in~\eqref{eq:mapping} can be interpreted as a restriction of the following broader mapping:  
\begin{equation}\label{eq:mapping2}
\mathfrak{F}: f \mapsto \bm{\nabla} \mathcal{L}^{-1} \Sigma_f = 
\left( \nabla_{0,e_1} \mathcal{L}^{-1} \Sigma_f, \nabla_{0,e_2} \mathcal{L}^{-1} \Sigma_f \right), 
\end{equation}
whose domain includes all $f\in \mathcal{C}_0$ such that $\bm{\nabla} \mathcal{L}^{-1} \Sigma_f$ is well-defined. 

To illustrate this, we examine the effects of
$\mathfrak{F}$ defined in~\eqref{eq:mapping2}  on $J^*$ and $\mathcal{LC}$ separately.  
Formally, we define $$W_i^{\sigma,\Phi}(\etah)=\sum_{x}x_i \eta_x^{\sigma,\Phi},$$ and then calculate  
\[ 
    \mathcal{L} W_i^{\sigma,\Phi} =\sum_{x\in \mathbb{Z}^2}\sum_{j=1,2}a(x,x+e_j) 
    \widetilde{\nabla}_{x,x+e_j}W_i^{\sigma,\Phi}= \Sigma_{j_i^{\sigma,\Phi}}, 
\]
which yields 
$$\mathfrak{F}(j_i^{a,\Phi}) = \bm{\nabla} \mathcal{L}^{-1} \Sigma_{j_i^{a,\Phi}}=\bm{\nabla} W_i^{\sigma,\Phi}=\mf[j]^{a,\Phi,i}.$$   
Moreover, note that $\bm{\nabla}\mathcal{L}^{-1}\Sigma_{f}= \bm{\nabla}\Sigma_{\mathcal{L}^{-1}f} $, 
we obtain  $\mathfrak{F}(\mathcal{L}f) = \bm{\nabla} \Sigma_f$.  
\end{remark}

Since most of functions that we are concerned with are in $T^\omega$,  
we define the set of germs of closed forms with
components in $T^\omega$ as: 
\[ 
    \mathfrak{T}_0^\omega \coloneqq  
    \Big\{ \mathfrak{U}=(\mathfrak{u}_1,\, \mathfrak{u}_2)\mid 
    \mathfrak{U} \text{ is a germ of a closed form, } \mathfrak{u}_i \in T^\omega, i=1,2 \Big\}, 
\]
and define  the set of germs of exact forms associated with functions in $T^\omega$ as: 
\[ 
    \mathfrak{E}_0^\omega \coloneqq  
    \Big\{ \bm{\nabla} \Sigma_h \mid h \in \To \Big\}\subset  \mathfrak{T}_0^\omega.  
\]
We also define the set of current forms on $T^\omega$ as: 
\[ 
\mf[J]^\omega= \Bigg\{ 
\sum_{i=1,2} u_i \mf[j]^{a,i} +v_i \mf[j]^{p,i} + m_i \mf[j]^{a,\omega,i }+ n_i \mf[j]^{p,\omega,i } 
\Big| u,v, m, n \in \mathbb{R}^2 \Bigg\} \subset \mathfrak{T}_0^{\omega}.
\]
Finally, we let  $\mathfrak{T}^\omega_{\bah} =\overline{\mathfrak{T}_0^\omega }$ and  $\mathfrak{E}^\omega_{\bah} =\overline{\mathfrak{E}_0^\omega }$ 
be the closure of $\mathfrak{T}_0^\omega$ and $\mathfrak{E}_0^\omega$ in $L^2(\mu_{\bah})$ respectively. 

We now state the main result of this section. Similar to Proposition~\ref{prop:structure}, we have the following direct sum decomposition of $\mf[T]^{\omega}$.
\begin{proposition}
    \label{thm:1}
    \[ 
        \mf[T]^{\omega}= \mf[E]^{\omega} \oplus \mf[J]^\omega 
    \]
\end{proposition}
\begin{proof}
    Applying  classical arguments (see, for example, the proof of Proposition 8.2.5 in~\cite{Erignoux21} or 
    Theorem 4.14 in Appendix 3 of~\cite{kipnis1999scaling}), 
    we can reduce the proof of this proposition to 
    the proof of a spectral gap estimation, which is illustrated in Proposition~\ref{thm:3}.  
\end{proof}

We endow $\mathfrak{T}^\omega_{\bah}$ with the following $L^2(\mu_{\bah})$ norm:  
\[
\|\mf[U]\|_{2,\bah} =  \mathbb{E}_{\bah} \big[ \mf[u]_1^2 + \mf[u]_2^2 \big]^{1/2}.
\]
By denoting the current form 
\[ 
    \mf[j]^{u,v,m,n}\coloneqq \sum_{i=1,2} u_i \mf[j]^{a,i} +v_i \mf[j]^{p,i} + m_i \mf[j]^{a,\omega,i }+ n_i \mf[j]^{p,\omega,i }, 
    \quad \forall u,v,m,n \in \mathbb{R}^2, 
\]
Proposition~\ref{thm:1} implies that for any \( \mf[U] \in \mathfrak{T}^\omega \), we can rewrite its norm as:  
\begin{equation}
\| \mf[U] \|_{2,\bah}^2 = \sup_{\substack{g \in \To\\ u,v,m,n \in \mathbb{R}^2}} 
\Big\{ 2 \big< \mf[U] , \bm{\nabla} \Sigma_g + \mf[j]^{u,v,m,n} \big>_{\bah}  - 
\big\| \bm{\nabla} \Sigma_g + \mf[j]^{u,v,m,n} \big\|_{2,\bah}^2 \Big\}.
\end{equation}
Define $\Ker(\mf[F])$ as the kernel of $\|\mf[F](\cdot)\|_{2,\bah}$, 
we can equip $\overline{\To/\Ker(\mf[F])}$ with the norm $\iip[\cdot]_{\bah}^{1/2}$ 
induced by the mapping $\mathfrak{F}$,  defined as: 
\[
\iip[f]_{\bah} = \| \mathfrak{F}(f) \|_{2,\bah}^2 = 
\sup_{\substack{g \in \To \\ u,v,m,n \in \mathbb{R}^2}} 
\Big\{ 2 \big< \mathfrak{F}(f) , \bm{\nabla} \Sigma_g + \mf[j]^{u,v,m,n} \big>_{\bah}  - 
\big\| \bm{\nabla} \Sigma_g + \mf[j]^{u,v,m,n} \big\|_{2,\bah}^2 \Big\}\]
By the generalized integration by parts technique (see, for example, Lemma 8.3.1 in~\cite{Erignoux21}), 
\(\mathfrak{F}\) defines an isometric isomorphism: 
\[
\mathfrak{F} : \Big( \overline{\To/\Ker(\mf[F])} ,\; \iip[\cdot]_{\bah}^{1/2} \Big) 
\longrightarrow \big( \mathfrak{T}^\omega ,\; \|\cdot\|_{2,\bah} \big),
\]
which gives $\overline{\To/\Ker(\mf[F])}$  
the same structure as $\mathcal{H}_{\bah}^\omega=\overline{\To/\mathcal{N}_{\bah}}$ 
as stated in Proposition~\ref{prop:structure}. 

We now briefly carry on with our heuristics and explain why Lemma~\ref{lemma:8.50} holds. 
The proof is based on the generalized integration by parts technique,  which gives 
\begin{align*}
     &\Big<  (-\mathcal{L}_l^{-1})  \sum_{x \in B_{l_{\psi}}} \tau_x \psi , \sum_{x \in B_{l_{\psi}}} \tau_x \psi  \Big>_{l,\widehat{K}_l}\\
     = &\frac{1}{2}\sum_{x,x+z\in B_{l} }
     \Big<  \nabla_{x,x+z} (\mathcal{L}_l^{-1})  \sum_{x \in B_{l_{\psi}}} \tau_x \psi , 
     \nabla_{x,x+z} (\mathcal{L}_l^{-1}) \sum_{x \in B_{l_{\psi}}} \tau_x \psi  \Big>_{l,\widehat{K}_l}.
\end{align*}
In the limit, using equivalence of ensembles, we can replace $\mu_{l,\hat{K}_l}$ by the translation invariant 
grand canonical measure $\mu_{\bah}$, and obtain that 
\begin{equation}\label{eq:8.55}
    \begin{aligned}
        & \lim_{l \to \infty} \frac{1}{(2l + 1)^2}
        \Big<
        (- \mathcal{L}_l)^{-1}\sum_{x \in B_{l,\varphi}} \tau_x \varphi,\, \sum_{x \in B_{l,\varphi}} \tau_x \varphi
        \Big>_{l,\widehat{K}_l}\\
        =&\frac{1}{2}\lim_{l \to \infty} \frac{1}{(2l+1)^2} \sum_{x,x+z\subset B_{l} }
        \Big\| \nabla_{x,x+z} (\mathcal{L}_l^{-1})  \sum_{x \in B_{l_{\psi}}} \tau_x \psi   \Big\|_{l,\widehat{K}_l}^2\\
        =&\sum_{i=1,2}\Big\| \nabla_{0,e_i} (\mathcal{L}^{-1})  \Sigma_\psi   \Big\|_{2,\bah}^2=\big\| \mf[F](\psi)   \big\|_{2,\bah}^2
        =\iip[\psi]_{\bah}.
    \end{aligned}
\end{equation}

The direct sum decomposition of $\mf[T]^\omega$ in Proposition~\ref{thm:1}, 
is the main difficulty in the proof of Theorem~\ref{thm:main3}.  
After obtaining Proposition~\ref{thm:1}, the proof of Lemma~\ref{lemma:8.50} 
reduces to rigorously justifying~\eqref{eq:8.55} using similar arguments from 
Section 8.5 of~\cite{Erignoux21} or Section 7.4 of~\cite{kipnis1999scaling}.  
Therefore, we postpone this rigorous proof to Appendix~\ref{appsec:main3}. 

\subsection{Spectral gap}
\label{sec:Spectral gap}
This section aims to establish the spectral gap estimate of the symmetric exclusion generator 
on the function class $T^\omega$, which is of order $n^{-2}$. 
This is the reason why we focus on $T^\omega$.  
We hope that $T^\omega$ is rich enough to include  all functions of interest in our model, 
such as various kinds of local gradients and currents, 
while still being simple enough to ensure that the spectral gap estimate holds

Recall that in~\eqref{eq:To} we define 
\[ 
    T^\omega = \Span \Bigg\{   f(\etah)=\sum_{x\in \mathbb{Z}^2}
    \left(a\eta^a_x+b\eta^p_x+ c\etaao_x + d\etapo_x  \right)F_x(\eta)\Big| 
     f \in \mathcal{C}, F_x \in \mathcal{S},\forall x   \Bigg\}.
\]
Next, we define  $\mathcal{C}_n$ as the set of cylinder functions 
whose supports are contained in $B_n$. With this notation, we define $T^\omega_n=T^\omega\cap \mathcal{C}_n.$

Recall that from the definition in~\eqref{eq:K} we encoded in the canonical state $\widehat{K}\in \widetilde{\mathbb{K}}_n$ the
number and angles of each type of particles in $B_n$. For each canonical state $\widehat{K}$, we define 
\[
 \alpha^a = \frac{K^a}{|B_n|},\quad \alpha^p = \frac{K^p}{|B_n|}, 
\]
representing the densities of each type of particles.  
Finally, we define the Dirichlet form with respect to the canonical measure $\mu_{n,\widehat{K}}$ as 
\[ 
\Eu[D]_{n,\widehat{K}}(f) = \big< f , -\LL_n f \big>_{n,\widehat{K}},
\]
where $\LL_n$ is the generator of SSEP restricted to jumps with both extremities in $B_n$. 

We are now ready to state the main result of this section.
\begin{proposition}[The spectral gap estimate for the active-passive SSEP with angles]
    \label{thm:3}
    For a fixed $\alpha' \in [0,1)$, there exists a constant $C(\alpha')$ such that for any $\widehat{K}\in \mathbb{K}_n$ satisfying 
$\alpha^a+\alpha^p\leq \alpha'$, and for any $f \in T_n^{\omega}$ with $\mathbb{E}_{n,\widehat{K}}[f]=0$, the following holds: 
\[ 
\mathbb{E}_{n,\widehat{K}}[f^2]\leq C(\alpha')n^2\EuScript{D}_{n,\widehat{K}}(f).
\]
\end{proposition}
Before proving Proposition~\ref{thm:3}, we need some preparations. 
For any fixed $\widehat{K}\in \mathbb{K}_n$  such that $\alpha^a+\alpha^p\leq \alpha'$. 
We  shorten  
\[ 
    \ob^\sigma=\mathbb{E}_{n,\widehat{K}}[\omega(\theta_0)| \eta^\sigma_0=1],\quad \text{and}\quad \ohw^\sigma =\omega - \ob^\sigma. 
\]
For any $x\in \TT_n$, 
we define the centered  variables $\eta^{\sigma,\oh^\sigma}=\ohw^\sigma(\theta_x)\eta^\sigma_x$
and 
\begin{equation} \label{eq:2}
    \chi_x = \frac{\alpha^p}{\alpha}\eta^a_x - \frac{\alpha^a}{\alpha}\eta^p_x.
\end{equation}
The centered  variables $\eta^{a,\oha}$, $\eta^{p,\ohp}$ and $\chi$ will play a crucial role 
in the $L^2$ orthogonal decomposition of  functions in $T_n^{\omega}$. 
To achieve this goal, we introduce some definitions. 
\begin{definition} 
    \label{def:1}
    For any $x\in \mathbb{Z}^2$ fixed.$\quad$

    \begin{enumerate}
        \item We say a function $F$ is \textbf{independent of type and angle at $x$ (ITA-$x$)}, 
        if it can be written as
        \[ 
        F(\etah)=\eta_x f(\etah) + (1- \eta_x)g(\etah),
        \]
        where $f$ is independent of $\etah_x$.
        \item We say a function $F$ is \textbf{independent of  active  angle at $x$ (IAA-$x$)},
        if it can be written as
        \[ 
        F(\etah) = \eta_x^a f(\etah) + \eta_x^p h(\etah)+ (1- \eta_x)g(\etah), 
        \]
        where $f$ is independent of $\etah_x$.
        \item We say a function $F$ is \textbf{independent of  passive angle at $x$ (IPA-$x$)},
        if it can be written as
        \[ 
        F(\etah) = \eta_x^p h(\etah)+\eta_x^a f(\etah) +  (1- \eta_x)g(\etah), 
        \]
        where $h$ is independent of $\etah_x$.
    \end{enumerate}
\end{definition}
\begin{remark}
    \label{rk:2}
    $\quad$ 

    \begin{itemize} 
        \item If a function $F$ is ITA-$x$ for every site $x$, we say $F$ is \textbf{ITA-everywhere}. 
        If a function $F$ is both ITA-$x$ and ITA-$y$, we say $F$ is \textbf{ITA-$xy$}. 
        The same terminologies apply to the IAA and IPA cases. 
        \item ITA-everywhere is just another term for functions that are both angle-blind and type-blind.  
        Therefore, $\mathcal{S}$ is just the set of all ITA-everywhere functions.
        \item If $F$ is ITA-$x$, heuristically, it can only feel whether  the site $x$ is occupied or not, 
        but it cannot distinguish the type or angle of the occupying particle. 
        \item A function $F$ is IAA-$x$ (or IPA-$x$) means that 
        $F$ distinguish not only whether the site $x$ is occupied or not, but also the type of the occupying particle. 
        However, when the occupying particle is an active (or passive) particle 
        $F$ cannot see its angle.  
    \item Since we can decompose  $\eta_xf = \eta^a_xf+\eta^p_xf$, An ITA-$x$ function naturally is both IAA-$x$ and IPA-$x$ at the same time. 
    \end{itemize}
\end{remark}

The following lemma illustrates that the ITA-$x$ functions are  orthogonal to $\chi_x$, 
and that the IAA-$x$ (resp. IPA-$x$)  are orthogonal to $\eta^{a,\oha}_x$ (resp. $\eta^{p,\ohp}_x$). 
\begin{lemma}\label{cor:1} 
    For any fixed canonical state $\widehat{K} \in \widetilde{\mathbb{K}}_n$, the following hold: 
    \begin{enumerate}
        \item For any ITA-$x$ function $F$,  
        \[ 
            \mathbb{E}_{n,\widehat{K}}[\chi_x F]=0.
        \]
        \item For any IAA-$x$ function $G$, IPA-$x$ function $H$, and  integrable function $\omega(\theta)$,  
        \[ 
            \mathbb{E}_{n,\widehat{K}}[\eta^{a,\oha}_x G]=0,\quad \mathbb{E}_{n,\widehat{K}}[\eta^{p,\ohp}_x H]=0.
        \]
    \end{enumerate}
\end{lemma}
\begin{proof}
    For an ITA-$x$ function, we have 
    \[ 
    \E\big[\eta^\sigma_x F\big]=\frac{\alpha^\sigma}{\alpha}\E\big[  \eta_x F \big],
    \]
    so that 
    \[ 
    \E\big[\chi_x F\big]=\Big( \frac{\alpha^p\alpha^a}{\alpha^2} - \frac{\alpha^a\alpha^p}{\alpha^2} \Big)\E\big[  \eta_x F \big]=0.
    \]
    For the second part, by definition of IAA-$x$, we can write $G$ as 
    \[ 
        G(\etah)=\eta_x^a f(\etah) + \eta_x^p h(\etah)+ (1- \eta_x)g(\etah)
    \]
    where $f$ is independent of $\etah_x$. 
    Thus, 
    \begin{align*}
        \E[\eta^{a,\omega}_x G] =\oba \alpha^a \E[f] = \oba\E\big[ \eta^a_x  G   \big],
    \end{align*}
    and similarly $$ \E[\eta^{p,\omega}_x H] = \obp\E\big[ \eta^p_x  H   \big].$$ 
    Finally, using the fact that $\overline{\ohw^\sigma}^\sigma=0$, the proof of Lemma~\ref{cor:1} is complete.  
\end{proof}

Now we state the following $L^2$ orthogonal decomposition of functions in $T^\omega_n$. 
\begin{lemma}[$L^2$ orthogonal decomposition of $T^\omega_n$ functions]
    \label{lemma:1}
    For any fixed canonical state $\widehat{K}\in \mathbb{K}_n$, and for any centered $f\in T_n^\omega$ 
    taking the form 
    $$f(\etah)=\sum_{x\in B_n}\left(a\eta^a_x+b\eta^p_x+ c\etaao_x + d\etapo_x  \right)F_x(\eta),$$
    where $F_x\in \mathcal{S}\cap \mathcal{C}_n$ and $\etah\in \SK[n]$. 
    we have the $L^2(\mu_{n,\widehat{K}})$-orthogonal decomposition of $f$ as follows: 
     \[ 
     f = f_1+f_2+f_3+f_4, 
     \]
     where $f_1$ is centered and ITA-everywhere:
     \[ 
     f_1(\etah) =  \left( (a+c\oba)\frac{\alpha^a}{\alpha}+(b+ d\obp)\frac{\alpha^p}{\alpha} \right)\sum_{x\in B_n} \eta_x F_x(\eta), 
     \]
     $f_2$ is centered and both IAA-everywhere and IPA-everywhere:
     \[ 
     f_2(\etah) = \left( (a+c\oba) - (b+ d\obp) \right)\sum_{x\in B_n} \chi_xF_x(\eta),
     \]
     $f_3$ is centered and IPA-everywhere:
     \[ 
     f_3(\etah) =  c \sum_{x\in B_n} \etaaob_x  F_x(\eta),
     \] 
     $f_4$ is centered and IAA-everywhere:
     \[ 
      f_4(\etah) =  d \sum_{x\in B_n} \etapob_x  F_x(\eta).
     \] 
\end{lemma} 
\begin{proof}
    Let $\tilde{a}=(a+c\oba)$ and $\tilde{b}=(b+ d\obp)$.  We can rewrite $f$ as 
    \begin{align*}
        f &= \sum_{x\in B_n}\left( \tilde{a} \eta^a_x + \tilde{b} \eta^p_x  \right)F_x(\eta)+\sum_{x\in B_n}\left( c\etaaob_x + d\etapob_x  \right)F_x(\eta)\\
        &= \left( \tilde{a}\frac{\alpha^a}{\alpha}+\tilde{b}\frac{\alpha^p}{\alpha} \right)\sum_{x\in B_n} \eta_x F_x(\eta) + \left( \tilde{a} - \tilde{b} \right)\sum_{x\in B_n} \chi_xF_x(\eta)\\
        & \quad + c \sum_{x\in B_n} \etaaob_x  F_x(\eta) + d \sum_{x\in B_n} \etapob_x  F_x(\eta)\\
        & = f_1+f_2+f_3+f_4.
        \end{align*}
        Since for all $x$,  $F_x(\eta)$ is ITA-everywhere,  and $f$ is centered, 
        using Lemma~\ref{cor:1}, we can conclude that $f_1$, $f_2$, $f_3$, $f_4$ are all centered. 
        Furthermore, by Lemma~\ref{cor:1}, we can see that $f_1, f_2, f_3, f_4$ are mutually orthogonal. 
        Thus, the proof of Lemma~\ref{lemma:1} is complete. 
\end{proof}

Finally, we give a proof of Proposition~\ref{thm:3}. 
\begin{proof}[Proof of Proposition~\ref{thm:3}]
    
Using $L^2$ orthogonal decomposition from Lemma~\ref{lemma:1}, 
we can split the variance and Dirichlet form of a centered $f$:
\[ 
        \mathbb{E}_{n,\widehat{K}}\big[f^2\big] = \sum_{i=1}^4 \E_{n,\widehat{K}}\big[f_i^2\big],\quad 
        \EuScript{D}_{n,\widehat{K}}(f) = \sum_{i=1}^4 \EuScript{D}_{n,\widehat{K}}(f_i).
\]
Thus, to prove Proposition~\ref{thm:3},  
we can break the problem into smaller parts by showing: 
\begin{equation}\label{eq:final}
    \mathbb{E}_{n,\widehat{K}}[f_i^2]\leq C(\alpha')n^2 \EuScript{D}_{n,\widehat{K}}(f_i),\quad \forall i=1, 2, 3, 4. 
\end{equation}

Note that 
$f_1$ satisfies the spectral gap estimate~\eqref{eq:final}, which is a classical result, (see Lemma 8.1.4 in~\cite{Erignoux21} 
or Lemma 8.2 in~\cite{quastel1992diffusion}).  Since $f_3$ and $f_4$ are symmetric, 
Proposition~\ref{thm:3} is reduced to proving~\eqref{eq:final} for $f_2$, and $f_3$. Before we start, we need the following lemma.
\begin{lemma}\label{lemma:finalcomp}
For any fixed canonical state $\widehat{K}\in \mathbb{K}_n$, and for any ITA-$xy$ function $F$,  it holds that  
\[ 
\begin{aligned}
    &\mathbb{E}_{n,\widehat{K}}\big[\eta^a_x\eta^a_y F\big]=\frac{K^a(K^a-1)}{K(K-1)}\mathbb{E}_{n,\widehat{K}}\big[\eta_x\eta_y F\big],\\
    &\mathbb{E}_{n,\widehat{K}}\big[\eta^a_x\eta^p_y F\big]=\frac{K^aK^p}{K(K-1)}\mathbb{E}_{n,\widehat{K}}\big[\eta_x\eta_y F\big],\\
    &\mathbb{E}_{n,\widehat{K}}\big[\eta^{a,\omega_1}_x \eta^{a,\omega_2}_y F\big]=\frac{K^a}{K(K-1)}\big( K^a \overline{\omega_1}^a\overline{\omega_2}^a - \overline{\omega_1\omega_2}^a \big)\mathbb{E}_{n,\widehat{K}}\big[\eta_x\eta_y F\big],\\
    &\mathbb{E}_{n,\widehat{K}}\big[\eta^{a,\omega_1}_x \eta^{p,\omega_2}_y F\big]=\frac{K^a K^p}{K(K-1)} \overline{\omega_1}^a\overline{\omega_2}^p \mathbb{E}_{n,\widehat{K}}\big[\eta_x\eta_y F\big].
\end{aligned}
\]
\end{lemma}
\begin{proof}[Proof of Lemma~\ref{lemma:finalcomp}]
    The first two equalities hold because:  
    \[ 
    \begin{aligned}
        &\frac{\PP_{n,\widehat{K}}(\eta^a_x=\eta^a_y=1)}{\PP_{n,\widehat{K}}(\eta_x=\eta_y=1)}=\frac{K^a(K^a-1)}{K(K-1)},\\
        &\frac{\PP_{n,\widehat{K}}(\eta^a_x=\eta^p_y=1)}{\PP_{n,\widehat{K}}(\eta_x=\eta_y=1)}=\frac{K^aK^p}{K(K-1)}.
    \end{aligned}
    \]
    Next we note that
    \[ 
    \begin{aligned}
        \sum_{i=1}^{K^a}\omega_1(\theta_i^a)\sum_{j\neq i}\omega_2(\theta_j^a)
        =&K^a \sum_{i=1}^{K^a}\omega_1(\theta_i^a)
        \Big( \overline{\omega_2}^a-\frac{1}{K^a} \omega_2(\theta_i^a) \Big)\\
        =& (K^a )^2 \overline{\omega_1}^a\overline{\omega_2}^a -K^a \overline{\omega_1\omega_2}^a.  
    \end{aligned}
    \]
    Thus, we have 
    \[ 
        \begin{aligned}
            \mathbb{E}_{n,\widehat{K}}\big[\eta^{a,\omega_1}_x \eta^{a,\omega_2}_y F\big]=&
            \frac{1}{K(K-1)}\sum_{i=1}^{K^a}\omega_1(\theta_i^a)\sum_{j\neq i}\omega_2(\theta_j^a)
            \mathbb{E}_{n,\widehat{K}}\big[\eta_x \eta_y F\big]\\
            =& \frac{K^a}{K(K-1)}\big( K^a \overline{\omega_1}^a\overline{\omega_2}^a - \overline{\omega_1\omega_2}^a \big)\mathbb{E}_{n,\widehat{K}}\big[\eta_x\eta_y F\big]. 
        \end{aligned}
    \]
    Similarly, 
    \[ 
        \begin{aligned}
            \mathbb{E}_{n,\widehat{K}}\big[\eta^{a,\omega_1}_x \eta^{p,\omega_2}_y F\big]=&
            \frac{1}{K(K-1)}\sum_{i=1}^{K^a}\omega_1(\theta_i^a)\sum_{j=1}^{K^p}\omega_2(\theta_j^a)
            \mathbb{E}_{n,\widehat{K}}\big[\eta_x \eta_y F\big]\\
            =&\frac{K^a K^p}{K(K-1)} \overline{\omega_1}^a\overline{\omega_2}^p \mathbb{E}_{n,\widehat{K}}\big[\eta_x\eta_y F\big]. 
        \end{aligned}
    \]
    Thus, the proof of Lemma~\ref{lemma:finalcomp} is complete. 
\end{proof}
In the rest of proof we denote  $V^a(\omega)=\var_{n,\widehat{K}}\big( \omega(\theta_0) | \eta^a_0=1 \big)$. 
Using Lemma~\ref{lemma:finalcomp} and definition of related  quantities, direct computations yield the following corollary. 
\begin{corollary}\label{cor:final}
    For any fixed canonical state $\widehat{K}\in \mathbb{K}_n$, the following hold:
    \begin{enumerate}
        \item For any ITA-$x$ function $F$:  
        \[ 
        \begin{aligned}
            &\mathbb{E}_{n,\widehat{K}}\big[\chi_x^2 F\big]=\frac{\alpha^a\alpha^p}{\alpha^2}\mathbb{E}_{n,\widehat{K}}\big[\eta_x F\big],\\
            &\mathbb{E}_{n,\widehat{K}}\big[\big( \eta^{a,\oh^a}_x \big)^2 F\big]=V^a(\omega)\frac{\alpha^a}{\alpha}\mathbb{E}_{n,\widehat{K}}\big[\eta_x F\big].
        \end{aligned}
        \]
        \item For any ITA-$xy$ function $F$:  
        \[ 
            \begin{aligned}
                &\mathbb{E}_{n,\widehat{K}}\big[\chi_x\chi_y F\big]=\frac{-\alpha^a\alpha^p}{\alpha^2(K-1)}\mathbb{E}_{n,\widehat{K}}\big[\eta_x\eta_y F\big],\\
                &\mathbb{E}_{n,\widehat{K}}\big[\etaaob_x\etaaob_y  F\big]=-V^a(\omega)\frac{K^a}{K(K-1)}\mathbb{E}_{n,\widehat{K}}\big[\eta_x\eta_y F\big].
            \end{aligned}
        \]
    \end{enumerate}
\end{corollary}
With all above preparations, we now prove the spectral gap estimate~\eqref{eq:final} holds for $f_2$ and $f_3$.  

Note that multiplying  $f_i$ by any constant does not  affect the inequality~\eqref{eq:final}. 
Hence, without loss of generality, we can drop the constants ahead of $f_i$. Moreover, 
replacing $F_x$ by $F_x - F $ does not affect $f_i$. Therefore,  
we can further assume, without loss of generality, that $\sum F_x=0$ and that each $F_x$ vanishes if $\eta_x=0$. 

First, we do computations related to $f_2$. 
Corollary~\ref{cor:final} and $\sum F_x=0$ yield 
\begin{equation}\label{eq:final21}
    \begin{aligned}
        \mathbb{E}_{n,\widehat{K}}[f_2^2]&= \sum_{x\in B_n}\Big( \mathbb{E}_{n,\widehat{K}}\big[ \chi_x^2F_x^2\big] + 
        \sum_{y \neq x} \mathbb{E}_{n,\widehat{K}}\big[ \chi_x\chi_yF_xF_y\big]\Big)\\
        &=\frac{\alpha^a\alpha^pK}{\alpha^2(K-1)}\sum_{x\in B_n}\mathbb{E}_{n,\widehat{K}}\big[ F_x^2\big]. 
    \end{aligned}
\end{equation}
Now we turn our attention to $\EuScript{D}_{n,\widehat{K}}(f_2)$. 
Since we assumed that $F_x$ vanishes when the site $x$ is empty, we have  
\begin{equation}\label{eq:final23}
    \begin{aligned}
        \LL_n \big( \chi_x F_x \big)&=\chi_x \LL_n F_x + \sum_{|z|=1}a(x,x+z)F_x(\etah^{x,x+z}) \big( \chi_{x+z}-\chi_x \big)\\
    &=\chi_x \LL_n F_x + \sum_{|z|=1}\chi_{x+z}\big( 1-\eta_x \big)F_x(\etah^{x,x+z}) 
    \end{aligned}
\end{equation} 
By Corollary~\ref{cor:final}, it follows that  
\begin{equation}\label{eq:final22}
    \begin{aligned}
        &\EuScript{D}_{n,\widehat{K}}\big(f_2\big)=-\sum_{x,y\in B_n}\Big( \mathbb{E}_{n,\widehat{K}} \big[ \chi_x\chi_yF_x\LL_n F_y\big] + \sum_{|z|=1}\mathbb{E}_{n,\widehat{K}}\big[ \chi_x\chi_{y+z}(1-\eta_y)F_x(\eta) F_y(\etah^{y,y+z})\big] \Big)\\
        &\qquad =\frac{\alpha^a\alpha^pK}{\alpha^2(K-1)}\sum_{x\in B_n}\Big( 
            \mathbb{E}_{n,\widehat{K}}\big[ -F_x\LL_n F_x\big] + \sum_{|z|=1}\mathbb{E}_{n,\widehat{K}}\big[ (1-\eta_{x+z})F_x(\eta) F_{x+z}(\etah^{x,x+z})\big] \Big).
    \end{aligned}
\end{equation} 

Next, we do computations related to $f_3$. 
Again, Corollary~\ref{cor:final} and $\sum F_x=0$ yield 
\begin{equation}\label{eq:final31}
    \begin{aligned}
        \mathbb{E}_{n,\widehat{K}}\big[f_3^2\big]&= 
        \sum_{x\in B_n}\Big( \mathbb{E}_{n,\widehat{K}}\big[ \big( \etaaob_x \big)^2F_x^2\big] + \sum_{y \neq x} \mathbb{E}_{n,\widehat{K}}\big[ \etaaob_x\etaaob_y F_xF_y\big]\Big)\\
        &=V^a(\omega)\frac{K^aK}{K(K-1)}\sum_{x\in B_n}\mathbb{E}_{n,\widehat{K}}\big[ F_x^2 \big]
    \end{aligned}
\end{equation}
Similar to~\eqref{eq:final23}, we also have 
\begin{align*}
    \LL_n \big( \etaaob_x F_x \big)&=\etaaob_x \LL_n F_x + \sum_{|z|=1}a(x,x+z)F_x(\etah^{x,x+z}) \big( \etaaob_{x+z}-\etaaob_x \big)\\
&=\etaaob_x \LL_n F_x + \sum_{|z|=1}\etaaob_{x+z}\big( 1-\eta_x \big)F_x(\etah^{x,x+z}) 
\end{align*} 
Thus, Corollary~\ref{cor:final} gives 
\begin{equation}\label{eq:final32}
    \begin{aligned}
        &\EuScript{D}_{n,\widehat{K}}\big(f_3\big)\\
        =&-\sum_{x,y\in B_n}\Big( \mathbb{E}_{n,\widehat{K}}\big[ \etaaob_x\etaaob_yF_x\LL_n F_y\big] + \sum_{|z|=1}\mathbb{E}_{n,\widehat{K}}\big[ \etaaob_x\etaaob_{y+z}(1-\eta_y)F_x(\eta) F_y(\etah^{y,y+z})\big] \Big)\\
        =&V^a(\omega)\frac{K^aK}{K(K-1)}\sum_{x\in B_n}\Big( \mathbb{E}_{n,\widehat{K}}\big[ -F_x\LL_n F_x\big] + \sum_{|z|=1}\mathbb{E}_{n,\widehat{K}}\big[ (1-\eta_{x+z})F_x(\eta)F_{x+z}(\etah^{x,x+z})\big] \Big).
    \end{aligned} 
\end{equation}
Comparing~\eqref{eq:final21} with~\eqref{eq:final22} as well as~\eqref{eq:final31} with~\eqref{eq:final32}, 
the following Lemma~\ref{lemma:3} concludes the proof of Proposition~\ref{thm:3}. 
\end{proof}
\begin{lemma}\label{lemma:3}
    Let $(F_x)_{x\in B_n}\subset\mathcal{S}\cap \mathcal{C}_n$ be a family of ITA-everywhere
    functions such that $\sum_{x\in B_n} F_x=0$ and $\eta_xF_x=F_x$, then it holds that 
    \begin{align*}
    \sum_{x\in B_n}\mathbb{E}_{n,\widehat{K}}\big[ F_x^2\big]&\leq C(\alpha')n^2 \\
    &\times  \sum_{x\in B_n}\Big( \mathbb{E}_{n,\widehat{K}}\big[ -F_x\LL_n F_x\big] + 
    \sum_{|z|=1}\mathbb{E}_{n,\widehat{K}}\big[ (1-\eta_{x+z})F_x(\eta) F_{x+z}(\etah^{x,x+z})\big] \Big).
\end{align*}
\end{lemma}
The proof of this lemma is similar to the proof of inequality (8.5) on page 143 of~\cite{Erignoux21}, 
hence we omit it here for brevity.

\vspace{0.3cm}
\noindent{\bf Acknowledgments} 
The author would like to express his gratitude to Prof. Kai Du from Fudan University 
for valuable discussions, suggestions, and support during the course of this work. 
This work was supported by the National Science and Technology Major Project(2022ZD0116401).


    \bibliographystyle{plain}
    \bibliography{mybib.bib}

\begin{thebibliography}{10}

\bibitem{agudo2019active}
Jaime Agudo-Canalejo and Ramin Golestanian.
\newblock Active phase separation in mixtures of chemically interacting particles.
\newblock {\em Physical review letters}, 123(1):018101, 2019.

\bibitem{bruna2022phase}
Maria Bruna, Martin Burger, Antonio Esposito, and Simon~M Schulz.
\newblock Phase separation in systems of interacting active brownian particles.
\newblock {\em SIAM Journal on Applied Mathematics}, 82(4):1635--1660, 2022.

\bibitem{annurev-conmatphys-031214-014710}
Michael~E. Cates and Julien Tailleur.
\newblock Motility-induced phase separation.
\newblock {\em Annual Review of Condensed Matter Physics}, 6(Volume 6, 2015):219--244, 2015.

\bibitem{de1989invariance}
Anna De~Masi, Pablo~A Ferrari, Sheldon Goldstein, and William~David Wick.
\newblock An invariance principle for reversible markov processes. applications to random motions in random environments.
\newblock {\em Journal of Statistical Physics}, 55:787--855, 1989.

\bibitem{erignoux2016limite}
Cl{\'e}ment Erignoux.
\newblock {\em Limite hydrodynamique pour un dynamique sur r{\'e}seau de particules actives}.
\newblock PhD thesis, Universit{\'e} Paris Saclay (COmUE), 2016.

\bibitem{Erignoux21}
Cl{\'e}ment {Erignoux}.
\newblock {Hydrodynamic limit for an active exclusion process}.
\newblock {\em Mémoires de la SMF}, 169, 2021.

\bibitem{esposito1994diffusive}
R~Esposito, R~Marra, and HT1301374 Yau.
\newblock Diffusive limit of asymmetric simple exclusion.
\newblock {\em Reviews in Mathematical physics}, 6(05a):1233--1267, 1994.

\bibitem{frouvelle2012continuum}
Amic Frouvelle.
\newblock A continuum model for alignment of self-propelled particles with anisotropy and density-dependent parameters.
\newblock {\em Mathematical Models and Methods in Applied Sciences}, 22(07):1250011, 2012.

\bibitem{frouvelle2012dynamics}
Amic Frouvelle and Jian-Guo Liu.
\newblock Dynamics in a kinetic model of oriented particles with phase transition.
\newblock {\em SIAM Journal on Mathematical Analysis}, 44(2):791--826, 2012.

\bibitem{guo1988nonlinear}
Mao~Zheng Guo, George~C Papanicolaou, and SR~Srinivasa Varadhan.
\newblock Nonlinear diffusion limit for a system with nearest neighbor interactions.
\newblock {\em Communications in Mathematical Physics}, 118(1):31--59, 1988.

\bibitem{kipnis1999scaling}
Claude Kipnis and Claudio Landim.
\newblock {\em Scaling limits of interacting particle systems}, volume 320.
\newblock Springer Science \& Business Media, 1999.

\bibitem{kipnis1986central}
Claude Kipnis and SR~Srinivasa Varadhan.
\newblock Central limit theorem for additive functionals of reversible markov processes and applications to simple exclusions.
\newblock {\em Communications in Mathematical Physics}, 104(1):1--19, 1986.

\bibitem{kourbane2018exact}
Mourtaza Kourbane-Houssene, Cl{\'e}ment Erignoux, Thierry Bodineau, and Julien Tailleur.
\newblock Exact hydrodynamic description of active lattice gases.
\newblock {\em Physical review letters}, 120(26):268003, 2018.

\bibitem{landim2001symmetric}
C~Landim, S~Olla, and SRS Varadhan.
\newblock Symmetric simple exclusion process:{\P} regularity of the self-diffusion coefficient.
\newblock {\em Communications in Mathematical Physics}, 224:307--321, 2001.

\bibitem{le2016brownian}
Jean-Fran{\c{c}}ois Le~Gall.
\newblock {\em Brownian motion, martingales, and stochastic calculus}.
\newblock Springer, 2016.

\bibitem{mason2023exact}
James Mason, Cl{\'e}ment Erignoux, Robert~L Jack, and Maria Bruna.
\newblock Exact hydrodynamics and onset of phase separation for an active exclusion process.
\newblock {\em Proceedings of the Royal Society A}, 479(2279):20230524, 2023.

\bibitem{mason2023macroscopic}
James Mason, Robert~L Jack, and Maria Bruna.
\newblock Macroscopic behaviour in a two-species exclusion process via the method of matched asymptotics.
\newblock {\em Journal of Statistical Physics}, 190(3):47, 2023.

\bibitem{mason2024dynamical}
James Mason, Robert~L Jack, and Maria Bruna.
\newblock Dynamical patterns in active-passive particle mixtures with non-reciprocal interactions: Exact hydrodynamic analysis.
\newblock {\em arXiv preprint arXiv:2408.03932}, 2024.

\bibitem{mccandlish2012spontaneous}
Samuel~R McCandlish, Aparna Baskaran, and Michael~F Hagan.
\newblock Spontaneous segregation of self-propelled particles with different motilities.
\newblock {\em Soft Matter}, 8(8):2527--2534, 2012.

\bibitem{ni2014crystallizing}
Ran Ni, Martien A~Cohen Stuart, Marjolein Dijkstra, and Peter~G Bolhuis.
\newblock Crystallizing hard-sphere glasses by doping with active particles.
\newblock {\em Soft Matter}, 10(35):6609--6613, 2014.

\bibitem{doi:10.1126/science.1230020}
Jeremie Palacci, Stefano Sacanna, Asher~Preska Steinberg, David~J. Pine, and Paul~M. Chaikin.
\newblock Living crystals of light-activated colloidal surfers.
\newblock {\em Science}, 339(6122):936--940, 2013.

\bibitem{quastel1992diffusion}
Jeremy Quastel.
\newblock Diffusion of color in the simple exclusion process.
\newblock {\em Communications on pure and applied mathematics}, 45(6):623--679, 1992.

\bibitem{ridgway2023motility}
Wesley~JM Ridgway, Mohit~P Dalwadi, Philip Pearce, and S~Jonathan Chapman.
\newblock Motility-induced phase separation mediated by bacterial quorum sensing.
\newblock {\em Physical Review Letters}, 131(22):228302, 2023.

\bibitem{solon2013revisiting}
Alexandre~P Solon and Julien Tailleur.
\newblock Revisiting the flocking transition using active spins.
\newblock {\em Physical review letters}, 111(7):078101, 2013.

\bibitem{spohn1990tracer}
Herbert Spohn.
\newblock Tracer diffusion in lattice gases.
\newblock {\em Journal of Statistical Physics}, 59:1227--1239, 1990.

\bibitem{key1354152m}
S.~R.~S. Varadhan.
\newblock Nonlinear diffusion limit for a system with nearest neighbor interactions. {II}.
\newblock In K.~D. Elworthy and Nobuyuki Ikeda, editors, {\em Asymptotic problems in probability theory: stochastic models and diffusions on fractals}, volume 283 of {\em Pitman Research Notes in Mathematics}, pages 75--128. Longman Scientific \& Technical, Harlow, 1993.
\newblock (Sanda/Kyoto, 1990). MR:1354152. Zbl:0793.60105.

\bibitem{Varadhan1994}
S.~R.~S. Varadhan.
\newblock {\em Regularity of Self-Diffusion Coefficient}, pages 387--397.
\newblock Birkh{\"a}user Boston, Boston, MA, 1994.

\bibitem{vicsek1995novel}
Tam{\'a}s Vicsek, Andr{\'a}s Czir{\'o}k, Eshel Ben-Jacob, Inon Cohen, and Ofer Shochet.
\newblock Novel type of phase transition in a system of self-driven particles.
\newblock {\em Physical review letters}, 75(6):1226, 1995.

\bibitem{VICSEK201271}
Tamás Vicsek and Anna Zafeiris.
\newblock Collective motion.
\newblock {\em Physics Reports}, 517(3):71--140, 2012.
\newblock Collective motion.

\bibitem{wysocki2016propagating}
Adam Wysocki, Roland~G Winkler, and Gerhard Gompper.
\newblock Propagating interfaces in mixtures of active and passive brownian particles.
\newblock {\em New journal of physics}, 18(12):123030, 2016.

\end{thebibliography}


\begin{thebibliography}{1}

\bibitem{Erignoux21}
Cl{\'e}ment {Erignoux}.
\newblock {Hydrodynamic limit for an active exclusion process}.
\newblock {\em Mémoires de la SMF}, 169, 2021.

\bibitem{kipnis2013scaling}
Claude Kipnis and Claudio Landim.
\newblock {\em Scaling limits of interacting particle systems}, volume 320.
\newblock Springer Science \& Business Media, 2013.

\bibitem{mason2023exact}
James Mason, Cl{\'e}ment Erignoux, Robert~L Jack, and Maria Bruna.
\newblock Exact hydrodynamics and onset of phase separation for an active exclusion process.
\newblock {\em Proceedings of the Royal Society A}, 479(2279):20230524, 2023.


\end{thebibliography}

\appendix         

\section*{Appendix}

\section{Cross diffusion}\label{sec:crossdiffusion}
Consider on $\mathbb{Z}^2$, an initial configuration where each site is initially occupied with probability $\alpha\in [0,1]$, 
and with a tagged particle at the origin. Each particle then follows a 
symmetric exclusion process. 
\begin{proposition}[Self-diffusion coefficient]\label{prop:crossdiffusion}
    \normalfont
    Given $\bm{X}_t=(X_t^1,X_t^2)$ the position of the tagged particle at time $t$, 
    the  2-dimensional self diffusion matrix $D_s(\alpha)$ is defined as~\cite{kipnis1986central}  
    \[ 
    x^\top D_s(\alpha)x = \lim_{t\to \infty}\frac{\E\big[ \big( x^\top \bm{X}_t \big)^2  \big]}{t},\quad \forall x\in \mathbb{R}^2. 
     \]
     Moreover, $D_s(\alpha)$  is diagonal, $D_s(\alpha)=d_s(\alpha)I_2$~\cite{de1989invariance}, and is characterized by the variational formula~\cite{spohn1990tracer} 
     \[ 
       x^\top D_s(\alpha) x = \inf_{f\in \mathcal{S}}\Bigg\{ \sum_{i=1,2}\mathbb{E}_{\hat{\alpha}}\Big[ (1-\eta_{e_i})\left( x_i + \tau_{e_i} f(\eta^{0,e_i}) - f\right)^2 +
        \sum_{y\neq 0,e_i}\big( \nabla_{0,e_i} \tau_y f \big) ^2  \Big] \Bigg\}.
     \]
\end{proposition}
The following results is about the regularity of the self-diffusion coefficient $d_s(\alpha)$ 
\begin{proposition}[Regularity of $d_s(\alpha)$]
    \normalfont
    The self diffusion $d_s$ is $C^\infty([0,1])$~\cite{landim2001symmetric}, 
    and for some constant $C>0$, we can write~\cite{Varadhan1994}
    \[ 
    \frac{1}{C}(1-\alpha)\leq d_s(\alpha) \leq C (1-\alpha).
    \]
    Moreover, $d_s(\alpha)$ has the following asymptotic expansion~\cite{mason2023macroscopic},
    \begin{equation}\label{eq:asym}
        d_s(\alpha)=\begin{cases}
            1-(1+\gamma)\alpha+O(\alpha^2) & \text{as } \alpha\to 0,\\
            \frac{1}{2\gamma+1}(1-\alpha) + O((1-\alpha)^2) & \text{as  } \alpha\to 1,
        \end{cases}
    \end{equation}
    and the minimal cubic polynomial approximation
    \[ 
    d_s(\alpha)=(1-\alpha)\Big(1-\gamma\alpha+\frac{\gamma(2\gamma-1)}{2\gamma+1}\alpha^2\Big), 
    \]
    where $\gamma=\pi/2-1$. 
\end{proposition}
\begin{remark}\label{remark:regularity}
    Using~\eqref{eq:asym}, we can see that 
    $$\mathcal{D}(\alpha)=\frac{1-d_s(\alpha)}{\alpha}$$
    can be continuously extended to $[0,1]$. 
\end{remark}
\begin{proposition}\label{prop:cross}
    Fix grand parameter $\bah\in \mathbb{M}_1(\SSS)$, and angular function $\omega \in C(\SSS)$. 
    For any vectors $u,v,m,n\in \mathbb{R}^2$, 
    let  $x=[u^\top,v^\top,m^\top,n^\top]^\top$, then it holds that 
    \begin{align*}
        x^\top Q^\omega_{\bah} x=\inf_{g\in \To}
        \iip[u^\top j^{a} +v^\top j^{p}+m^\top j^{a,\hat{\omega}}+n^\top j^{p,\hat{\omega}}+\LL g]_{\bah}, 
    \end{align*}
    where 
    \[ 
        Q^\omega_{\bah} =  \begin{bmatrix}
        M  &   &   \\
          &   & \alpha^aV^a_{\bah}(\omega) d_s(\alpha)I_2 &   \\
          &   &   & \alpha^pV^p_{\bah}(\omega) d_s(\alpha) I_2
        \end{bmatrix}, 
\]
and $M$ is the mobility matrix defined in~\eqref{eq:mobilitymatrix}.
\end{proposition}
\begin{proof}
    The proof is analogous to the proof of Theorem3.2 in~\cite{quastel1992diffusion}. 
    By Definition~\ref{def:iip}, we have 
    \begin{align*}
        &\iip[u^\top j^{a} +v^\top j^{p}+m^\top j^{a,\hat{\omega}}+
        n^\top j^{p,\hat{\omega}}+\LL g]_{\bah}\\
        =&\mathbb{E}_{\bah}\Bigg[ \sum_{i=1,2}\Big( m_i\etaaob_0(1-\eta_{e_i})+n_i\etapob_0(1-\eta_{e_i})
        \\ &\qquad\qquad +u_i\eta_0^a(1-\eta_{e_i})+v_i\eta_0^p(1-\eta_{e_i})
         +\nabla_{0,e_i} \Sigma_{g} \Big)^2 \Bigg]\\
         =&\mathbb{E}_{\bah}\Bigg[ \sum_{i=1,2}\Big( m_i\etaaob_0(1-\eta_{e_i})
         +n_i\etapob_0(1-\eta_{e_i})
         \\ &\qquad +(u_i-v_i)\chi_0(1-\eta_{e_i}) +(\frac{\alpha^a}{\alpha}u_i+ \frac{\alpha^p}{\alpha}v_i)\eta_0(1-\eta_{e_i})
         +\nabla_{0,e_i} \Sigma_{g} \Big)^2 \Bigg].
    \end{align*}
Note that for any finite supported $g$, $\tilde{\nabla}_{0,e_i}$ is a finite sum, 
$$\tilde{\nabla}_{0,e_i} \Sigma_{g} = \sum_{x\in A_g}\tau_x \tilde{\nabla}_{-x,-x+e_i} g
=\tilde{\nabla}_{0,e_i} \sum_{x\in A_g}\tau_x g , $$ 
where $A_g \coloneqq \big\{x\in \mathbb{Z}^2 ; \{-x,-x+e_i\}\cap \supp(g)\neq \emptyset \big\}$. 
Thus we can obtain the following orthogonal relationship,  
\begin{align*}
    &\mathbb{E}_{\bah}\big[ \eta_0(1-\eta_{e_i})\nabla_{0,e_i} \Sigma_{g}  \big]=
    \mathbb{E}_{\bah}\big[ a(0,e_{i}) \eta_0\tilde{\nabla}_{0,e_i} \Sigma_{g}  \big]\\
    =&\frac{1}{2}\mathbb{E}_{\bah}\big[ a(0,e_{i}) (\eta_0-\eta_{e_i})\tilde{\nabla}_{0,e_i} \Sigma_{g}  \big]=
    \frac{1}{2}\mathbb{E}_{\bah}\big[  (\eta_0-\eta_{e_i})\tilde{\nabla}_{0,e_i} \Sigma_{g}  \big]\\
    =&\sum_{x\in A_g} \mathbb{E}_{\bah}\big[  (\eta_{e_i}-\eta_0) \tau_x g  \big]
    =\sum_{x\in A_g} \mathbb{E}_{\bah}\big[  (\eta_{-x+e_i}-\eta_{-x}) g \big]=0, 
\end{align*}
because the variables within the support of $g$ cancel.  
Thus, using orthogonal relationships, we have  
\begin{align*}
    &\iip[u^\top j^{a} +v^\top j^{p}+m^\top j^{a,\hat{\omega}}+n^\top j^{p,\hat{\omega}}+\LL g]_{\bah}\\
     =&\mathbb{E}_{\bah}\Bigg[ \sum_{i=1,2}\Big( m_i\etaaob_0(1-\eta_{e_i})+ n_i\etapob_0(1-\eta_{e_i})\\
     &\qquad\qquad \qquad\qquad  +(u_i-v_i)\chi_0(1-\eta_{e_i})+\nabla_{0,e_i} \Sigma_{g} \Big)^2  \Bigg]\\
     &\quad  +\sum_{i=1,2} \Big( \frac{\alpha^a}{\alpha}u_i+ \frac{\alpha^p}{\alpha}v_i \Big)^2 \alpha(1-\alpha).
\end{align*}
Since the infimum is taken over $g\in \To$, thanks for the orthogonality,  without loss of generality, 
$g$ can be chosen in the form $g=\sum (\chi_x+c\etaaob_x+d\etapob_x)F_x(\eta)$. 
Again using the fact that $g$ is finite supported, by defining $F'=\sum_y \tau_{-y}F_y$, we can write   
\[ 
     \Sigma_{g} = \sum_{y,y'\in \mathbb{Z}^2}\tau_{y'}\big[ (\chi_y+c\etaaob_y+d\etapob_y)F_y(\eta) \big]
     = \sum_{y\in \mathbb{Z}^2}\big(\chi_y+c\etaaob_y+d\etapob_y\big) \tau_y F'(\eta). 
\]
Therefore, orthogonality gives 
\[ 
\begin{aligned}
    &\mathbb{E}_{\hat{\alpha}}\Bigg[ 
        \sum_{i=1,2}\Big( m_i\etaaob_0(1-\eta_{e_i})+n_i\etapob_0(1-\eta_{e_i})+
        (u_i-v_i)\chi_0(1-\eta_{e_i})+\nabla_{0,e_i} \Sigma_{g} \Big)^2  \Bigg]\\
    =&I_1 + I_2+I_3,
\end{aligned}
\]
where 
\begin{align*}
    &I_1=\sum_{i=1,2}\mathbb{E}_{\hat{\alpha}}\Bigg[ \Big( m_i\etaaob_0(1-\eta_{e_i}) +
    \nabla_{0,e_i} \sum_{y\in \mathbb{Z}^2}\etaaob_y \tau_y cF'(\eta) \Big)^2  \Bigg],\\
    &I_2=\sum_{i=1,2}\mathbb{E}_{\hat{\alpha}}\Bigg[ \Big( n_i\etapob_0(1-\eta_{e_i}) +
    \nabla_{0,e_i} \sum_{y\in \mathbb{Z}^2}\etapob_y \tau_y dF'(\eta) \Big)^2  \Bigg],\\
    &I_3=\sum_{i=1,2}\mathbb{E}_{\hat{\alpha}}\Bigg[ \Big( (u_i-v_i)\chi_0(1-\eta_{e_i}) +
    \nabla_{0,e_i} \sum_{y\in \mathbb{Z}^2}\chi_y \tau_y F'(\eta) \Big)^2  \Bigg].
\end{align*}

For term $I_1$,  elementary computations yield 
\begin{align*}
    &\nabla_{0,e_i}\etaaob_0 \tau_0 cF'(\eta)=-\etaaob_0(1-\eta_{e_i})cF'(\eta),\\
    &\nabla_{0,e_i}\etaaob_{e_i} \tau_{e_i} cF'(\eta)=\etaaob_0(1-\eta_{e_i})\tau_{e_i} cF'(\eta^{0,e_i}),\\
    &\nabla_{0,e_i}\etaaob_{y} \tau_{y} cF'(\eta)=\etaaob_y \nabla_{0,e_i} \tau_{y} 
    cF'(\eta^{0,e_i}),\quad \forall y\neq 0,e_i, 
\end{align*}
thus  
\begin{align*}
    I_1=\sum_{i=1,2}\mathbb{E}_{\hat{\alpha}}\Bigg[ \Big( \etaaob_0(1-\eta_{e_i})\left( m_i + \tau_{e_i} cF'(\eta^{0,e_i}) - cF'(\eta)\right) +
   \sum_{y\neq 0,e_i}\etaaob_y \nabla_{0,e_i} \tau_y cF'(\eta) \Big)^2  \Bigg].
\end{align*}
Note that 
\[ 
    \mathbb{E}_{\hat{\alpha}}\Big[ \etaaob_x\etaaob_y F(\eta )\Big]=
    \delta_{x,y} \frac{\alpha^aV^a_{\hat{\alpha}}}{\alpha}\mathbb{E}_{\hat{\alpha}}\big[ \eta_y F(\eta)\big]
\]
thus 
\begin{align*}
    I_1
   =&\frac{\alpha^aV^a_{\hat{\alpha}}}{\alpha}\sum_{i=1,2}\mathbb{E}_{\hat{\alpha}}
   \Bigg[ \eta_0(1-\eta_{e_i})\Big( m_i + \tau_{e_i} cF'(\eta^{0,e_i}) - cF'(\eta)\Big)^2 \\
   &\qquad\qquad\qquad\qquad\qquad\qquad\qquad +\sum_{y\neq 0,e_i}\eta_y \Big( \nabla_{0,e_i} \tau_y cF'(\eta) \Big) ^2  \Bigg].
\end{align*}
Denote by $\eta \setminus \eta_0=\left\{ \eta_x \right\}_{x \neq 0, x \in \mathbb{Z}^2} $, and 
$$f(\eta \setminus \eta_0)=\mathbb{E}_{\hat{\alpha}}[F'|\eta_0=1](\eta \setminus \eta_0)=\alpha^{-1}\mathbb{E}_{\hat{\alpha}}[\eta_0 F'|\eta \setminus \eta_0]$$
with this definition, we still write $f(\eta)$ but here $f$ is considered independent of $\etah_0$. Hence 
\begin{align*}
    I_1
   =\alpha^aV^a_{\hat{\alpha}}\sum_{i=1,2}\mathbb{E}_{\hat{\alpha}}
   \Bigg[ (1-\eta_{e_i})\Big( m_i + \tau_{e_i} cf(\eta^{0,e_i}) - cf(\eta)\Big)^2 +
   \sum_{y\neq 0,e_i}\Big( \nabla_{0,e_i} \tau_y cf(\eta) \Big) ^2  \Bigg].
\end{align*}
Using the same procedure, and the fact that 
\begin{align*}
    &\mathbb{E}_{\hat{\alpha}}\big[ \etapob_x\etapob_y F(\eta )\big]=
    \delta_{x,y} \frac{\alpha^pV^p_{\hat{\alpha}}}{\alpha}\mathbb{E}_{\hat{\alpha}}\big[ \eta_y F(\eta)\big],\quad 
    \mathbb{E}_{\hat{\alpha}}\big[ \chi_x\chi_y F(\eta )\big]=
    \delta_{x,y} \frac{\alpha^a\alpha^p}{\alpha^2}\mathbb{E}_{\hat{\alpha}}\big[ \eta_y F(\eta)\big]
\end{align*}
we obtain that 
\begin{align*}
    &I_2
   =\alpha^pV^p_{\hat{\alpha}}\sum_{i=1,2}\mathbb{E}_{\hat{\alpha}}
   \Bigg[ (1-\eta_{e_i})\Big( n_i + \tau_{e_i} df(\eta^{0,e_i}) - df(\eta)\Big)^2 +
   \sum_{y\neq 0,e_i}\Big( \nabla_{0,e_i} \tau_y df(\eta) \Big) ^2  \Bigg],\\
   &I_3
   =\frac{\alpha^a\alpha^p}{\alpha}\sum_{i=1,2}\mathbb{E}_{\hat{\alpha}}\Bigg[ (1-\eta_{e_i})
   \Big( (u-v)_i + \tau_{e_i} f(\eta^{0,e_i}) - f(\eta)\Big)^2 +
   \sum_{y\neq 0,e_i}\Big( \nabla_{0,e_i} \tau_y f(\eta) \Big) ^2  \Bigg].
\end{align*}
Combing, for any $g \in \To $  taking the form $g=\sum_x (\chi_x+c\etaaob_x+d\etapob_x)F_x(\eta)$,
\begin{align*}
    &\iip[u^\top j^{a} +v^\top j^{p}+m^\top j^{a,\hat{\omega}}+n^\top j^{p,\hat{\omega}}+\LL g]_{\hat{\alpha}}\\
    =&\alpha^aV^a_{\hat{\alpha}}\sum_{i=1,2}\mathbb{E}_{\hat{\alpha}}
    \Bigg[ (1-\eta_{e_i})\Big( m_i + \tau_{e_i} cf(\eta^{0,e_i}) - cf(\eta)\Big)^2 +
    \sum_{y\neq 0,e_i}\Big( \nabla_{0,e_i} \tau_y cf(\eta) \Big) ^2  \Bigg]\\
    &+\alpha^pV^p_{\hat{\alpha}}\sum_{i=1,2}\mathbb{E}_{\hat{\alpha}}
    \Bigg[ (1-\eta_{e_i})\Big( n_i + \tau_{e_i} df(\eta^{0,e_i}) - df(\eta)\Big)^2 +
    \sum_{y\neq 0,e_i}\Big( \nabla_{0,e_i} \tau_y df(\eta) \Big) ^2  \Bigg]\\
    &+\frac{\alpha^a\alpha^p}{\alpha}\sum_{i=1,2}\mathbb{E}_{\hat{\alpha}}\Bigg[ (1-\eta_{e_i})\left( (u-v)_i + \tau_{e_i} f(\eta^{0,e_i}) - f(\eta)\right)^2 +
    \sum_{y\neq 0,e_i}\Big( \nabla_{0,e_i} \tau_y f(\eta) \Big) ^2  \Bigg]\\
    &+\alpha(1-\alpha)\sum_{i=1,2} \Big( \frac{\alpha^a}{\alpha}u_i+ \frac{\alpha^p}{\alpha}v_i \Big)^2 .
\end{align*}

Finally, taking infimum over $g \in \To$  is equivalent to taking infimum over 
$f\in \mathcal{S}$, according Proposition~\ref{prop:crossdiffusion},  
we have 
\begin{align*}
    &\inf_{g\in T^\omega}\iip[u^\top j^{a} +v^\top j^{p}+m^\top j^{a,\hat{\omega}}+n^\top j^{p,\hat{\omega}}+\LL g]_{\hat{\alpha}}\\
    =&m^\top \alpha^aV^a_{\hat{\alpha}} d_s(\alpha)I_2 m +
     n^\top \alpha^pV^p_{\hat{\alpha}} d_s(\alpha)I_2 n +
     \frac{\alpha^a\alpha^p}{\alpha} (u-v)^\top d_s(\alpha)I_2 (u-v) \\
    &+\alpha(1-\alpha) 
    \Big( \frac{\alpha^a}{\alpha}u+ \frac{\alpha^p}{\alpha}v \Big)^\top \Big( \frac{\alpha^a}{\alpha}u+ \frac{\alpha^p}{\alpha}v \Big) \\
    =& x^\top Q^\omega_{\bah} x ,
\end{align*} 
where $x=[u^\top,v^\top,m^\top,n^\top]^\top$. 
\end{proof}

\section{Space of grand-canonical parameters}
\subsection{Equivalence of ensembles}
\begin{proposition}[Equivalence of ensembles]\label{prop:equivalence}
    Let $g$ be a cylinder function , we have
    \[
    \limsup_{l \to \infty} \sup_{\widehat{K} \in \mathbb{K}_l} \Big| 
    \mathbb{E}_{l,\widehat{K}}\big[ g \big] - \mathbb{E}_{\bah_{\widehat{K}}}\big[ g \big] \Big| \to 0,
    \]
    where the first measure is the canonical measure $\mu_{l,\widehat{K}}$ defined in~\eqref{eq:mu},  
    and the second is the grand-canonical measure $\mu_{\bah_{\widehat{K}}}$ with parameter 
    $\bah_{\widehat{K}}$ defined in~\eqref{eq:bah}.
\end{proposition}
\begin{proof}
The proof is analogous to the proof of Proposition C.1.1 in~\cite{Erignoux21}. 

For any cylinder function $g$ fixed, which monitors $M$ positions on $B_l$.
We consider two samplings of $M$ variables, chosen among $L=(2l+1)^2$ fixed possible values
\[
    \xi^L = \big\{\widehat{\eta}^1, \ldots, \widehat{\eta}^L\big\} \in 
    \Sigma_1^L \coloneqq \Big\{(\etat, \theta) \in \{0, 1\}^2 \times \SSS, \; 
    \theta = 0 \text{ if } \etat = \tilde{0}\Big\}^L.
\]
The first sampling is made without replacement to represent the canonical measure $\mu_{l,\widehat{K}}$, 
and the sampled items will be denoted $\bm{X}^M\coloneqq (X_1, \ldots, X_M)$.  
The second sampling is made with replacement to represent the grand-canonical measure 
$\mu_{\bah_{\widehat{K}}}$, and the sampled items will be denoted $\bm{Y}^M\coloneqq( Y_1, \ldots, Y_M)$. 

We denote by $\mathbb{E}_{\xi^L}$ the expectation w.r.t. the two samplings $(X_i)$ and $(Y_i)$ given $\xi^L$. 
Further denote by $I_{L,M} = \{1,\ldots,L\}^M$ be the `name set' of all samplings. 
Each ${\bf i} = (i_1,\ldots,i_M)\in I_{L,M}$ encodes the `names' of sampled items, and represents a sampled trajectory. 
Moreover, we denote by 
\[ D_{L,M} = \{(i_1,\ldots,i_M) \in I_{L,M} \mid i_1 \neq \cdots \neq i_M\}\] 
the set of all samplings without replacement, and by 
\[  C_{L,M} = I_{L,M} \setminus D_{L,M} \]
the set of samplings having at least one repeated item.  

Then  we have
\begin{align*}
    \Big|\mathbb{E}_{\xi^L}[g(\bm{X}^M)] - \mathbb{E}_{\xi^L}[g(\bm{Y}^M)]\Big| 
    \leq& \|g\|_\infty \sum_{{\bf i} \in I_{L,M}} 
    \Big| \mathbb{P}_{\xi^L}\big(\bm{X}^M = (\widehat{\eta}^{i_1},\ldots,\widehat{\eta}^{i_M})\big)  \\
    &\qquad \qquad \qquad 
     - \mathbb{P}_{\xi^L}\big(\bm{Y}^M = (\widehat{\eta}^{i_1},\ldots,\widehat{\eta}^{i_M})\big) \Big| \\
    =& \|g\|_\infty \sum_{{\bf i} \in D_{L,M}}
    \Big| \mathbb{P}_{\xi^L}\big(\bm{X}^M = (\widehat{\eta}^{i_1},\ldots,\widehat{\eta}^{i_M})\big)  \\
    &\qquad \qquad \qquad  
    - \mathbb{P}_{\xi^L}\big(\bm{Y}^M = (\widehat{\eta}^{i_1},\ldots,\widehat{\eta}^{i_M})\big) \Big|  \\
    &\quad + \|g\|_\infty \sum_{{\bf i} \in C_{L,M}} \mathbb{P}_{\xi^L}\big(\bm{Y}^M = (\widehat{\eta}^{i_1},\ldots,\widehat{\eta}^{i_M})\big).
\end{align*}
The sum on the last line is the probability that at least two indexes among the $M$ we chosen uniformly in $\{1,\ldots,L\}$ are equal. This probability is
\[
\sum_{{\bf i} \in C_{L,M}} \mathbb{P}_{\xi^L}\big(\bm{Y}^M  = (\widehat{\eta}^{i_1},\ldots,\widehat{\eta}^{i_M})\big) = 1 - \frac{L(L-1)\cdots(L-M+1)}{L^M},
\]
which for $M$ fixed vanishes uniformly in $\xi^L$ as $L \to \infty$. We now take a look at the other term, for which we write

\begin{align*}
&\sum_{{\bf i} \in D_{L,M}} \Big| \mathbb{P}_{\xi^L}\big(\bm{X}^M  = (\widehat{\eta}^{i_1},\ldots,\widehat{\eta}^{i_M})\big) 
- \mathbb{P}_{\xi^L}\big(\bm{Y}^M  = (\widehat{\eta}^{i_1},\ldots,\widehat{\eta}^{i_M})\big) \Big| \\
=& \sum_{{\bf i} \in D_{L,M}} \Big| \frac{1}{L(L-1)\cdots(L-M+1)} - \frac{1}{L^M} \Big| 
= 1 - \frac{L(L-1)\cdots(L-M+1)}{L^M},
\end{align*}
which also vanishes uniformly in $\xi^L$ as $L \to \infty$. We can therefore write for any bounded function $g$ depending on $M$ sites
\[
\sup_{\xi^L \in \Sigma_L^1} \Big| \mathbb{E}_{\xi^L}[g(\bm{X}^M )] - 
\mathbb{E}_{\xi^L}[g(\bm{Y}^M )] \Big| \leq \|g\|_\infty C(M) o_L(1). 
\]

\end{proof}
\subsection{regularity of the grand-canonical measures}
\begin{proposition}\label{prop:Lipschitz}
    Consider the set of grand canonical parameters $\mathbb{M}_1(\SSS)$ equipped with the 
    metric $d(\cdot,\cdot)$ defined in~\eqref{eq:d}.  
    Then, given a cylinder function $g$, the mapping $\Phi_g(\bah)= \mathbb{E}_{\bah}[g]$
    is Lipschitz continuous on $(\mathcal{M}_1(\mathbb{S}), d)$ with Lipschitz constant depending on the function $g$.
\end{proposition}
\begin{proof}
    The proof is analogous to the proof of Proposition C.2.1 in~\cite{Erignoux21}.
    Let us consider a cylinder function $g$ depending only on vertices $x_1, \dots, x_M$, and we prove this proposition by induction on $M$. 
    
    For $M=1$, we can write $g$ as 
    \[ 
    g(\etah_{x_1})=\eta^a_{x_1}g((1^a, \theta_{x_1}))+\eta^p_{x_1}g((1^p, \theta_{x_1}))+(1-\eta_{x_1})g((\tilde{0},0)),
    \]
    hence 
    \begin{align*}
       \big| \mathbb{E}_{\bm{\hat{\alpha}}}[g] - \mathbb{E}_{\bm{\hat{\alpha}}'}[g]\big|
       =& \Bigg|\int g((1^a, \theta_{x_1})) (\hat{\alpha}^a - (\hat{\alpha}^{a})')(d\theta_{x_1}) +\int g((1^p, \theta_{x_1})) (\hat{\alpha}^p - (\hat{\alpha}^{p})')(d\theta_{x_1})\Bigg|\\
    \leq& \|g\|_{W^{2,\infty}} d(\bm{\hat{\alpha}}, \bm{\hat{\alpha}}').  
    \end{align*}

Assuming now that the proposition is true for any function depending on $M-1$ sites. Next, we consider a function $g$ depending on $M$ vertices. 
Here we temporarily introduce a notation $\etah_{x_1}^c=\left\{ \etah_{x_2}\dots \etah_{x_M} \right\}$. 
Note that 
\begin{align*}
    \mathbb{E}_{\bm{\hat{\alpha}}}[g] - \mathbb{E}_{\bm{\hat{\alpha}}'}[g] 
     =&\mathbb{E}_{\etah_{x_1} \sim \bm{\hat{\alpha}}}\big[ \E_{\etah_{x_1}^c \sim \bm{\hat{\alpha}}} [g| \etah_{x_1}] \big] - \mathbb{E}_{\etah_{x_1} \sim\bm{\hat{\alpha}}'}\big[ \E_{\etah_{x_1}^c\sim \bm{\hat{\alpha}}'} [g|\etah_{x_1} ] \big] \\
     =& \mathbb{E}_{\etah_{x_1} \sim \bm{\hat{\alpha}}}\big[ \E_{\etah_{x_1}^c \sim \bm{\hat{\alpha}}} [g| \etah_{x_1}] - \E_{\etah_{x_1}^c\sim \bm{\hat{\alpha}}'} [g|\etah_{x_1} ] \big] \\
     &+ \mathbb{E}_{\etah_{x_1}^c \sim\bm{\hat{\alpha}}'}\big[ \E_{\etah_{x_1}\sim \bm{\hat{\alpha}}} [g|\etah_{x_1}^c ] - \E_{\etah_{x_1}\sim \bm{\hat{\alpha}}'} [g|\etah_{x_1}^c ] \big].
\end{align*}
Hence, using induction hypothesis in two difference terms, we get
\begin{align*}
    \Big|\mathbb{E}_{\bm{\hat{\alpha}}}[g] - \mathbb{E}_{\bm{\hat{\alpha}}'}[g]\Big| 
     \leq&  \mathbb{E}_{\etah_{x_1} \sim \bm{\hat{\alpha}}}\big[ \big|\E_{\etah_{x_1}^c \sim \bm{\hat{\alpha}}} [g| \etah_{x_1}] - \E_{\etah_{x_1}^c\sim \bm{\hat{\alpha}}'} [g|\etah_{x_1} ]\big| \big] \\
     &+ \mathbb{E}_{\etah_{x_1}^c \sim\bm{\hat{\alpha}}'}\big[ \big|\E_{\etah_{x_1}\sim \bm{\hat{\alpha}}} [g|\etah_{x_1}^c ] - \E_{\etah_{x_1}\sim \bm{\hat{\alpha}}'} [g|\etah_{x_1}^c ]\big| \big]\\
     \leq & \big( \mathbb{E}_{\etah_{x_1} \sim \bm{\hat{\alpha}}}\big[ C(g,\etah_{x_1}) \big] + \mathbb{E}_{\etah_{x_1}^c \sim\bm{\hat{\alpha}}'}\big[  C(g,\etah_{x_1}^c) \big] \big)d(\bm{\hat{\alpha}}, \bm{\hat{\alpha}}').
\end{align*}

\end{proof}

\subsection{Compactness of  grand canonical parameter set}
Recall that on the set of grand-canonical parameters 
\[ 
    \mathbb{M}_1(\SSS)=\Big\{ \bah=(\hat{\alpha}^a,\hat{\alpha}^p)\in \mathcal{M}_1(\SSS)^2 
    \Big| \int_{\SSS}  \hat{\alpha}^a(\dd \theta)+\int_{\SSS} \hat{\alpha}^p(\dd \theta)\leq 1
    \Big\},
\] 
we define a metric 
\[ 
d (\bah, \bah')=\sup_{f\in B^*} \Bigg\{ \int_{\SSS} f\, \dd \big( \ah^a - \ah'^a \big)  \Bigg\}+
\sup_{g\in B^*} \Bigg\{ \int_{\SSS} g\, \dd \big( \ah^p -  \ah'^p  \big)  \Bigg\}
\]
where  $B^*$ is the unit ball in $(C^2(\SSS),\|\cdot\|_{W^{2,\infty}})$. 

The next result shows that $(\mathbb{M}_1(\SSS), d)$ is a compact metric space. 
\begin{proposition}\label{prop:compact}
    The metric space $(\mathbb{M}_1(\SSS), d)$ is totally bounded and Cauchy complete, and hence is compact.  
\end{proposition}
\begin{proof}
The proof is analogous to the proof of Proposition C.3.1 in~\cite{Erignoux21}.
We first show  the Cauchy-completeness.  
Consider a Cauchy sequence $(\bah_k)_{k \in \mathbb{N}}$, 
then by definition of the metric $d$, for any $g \in B^*$, 
the sequence $(\int_S g(\theta) \widehat{\alpha}_k^\sigma(d\theta))_k$ are two real Cauchy sequences,  
and therefore converge, and we can let
$$
\int_{\SSS} g d \widehat{\alpha}_*^\sigma = 
\lim_{k \to \infty} \int_{\SSS} g d \widehat{\alpha}_k^\sigma .
$$
This definition can be extended to any $C^2(S)$ function $g$ by 
$$
\int_{\SSS} g d \widehat{\alpha}_*^\alpha = 
\|g\|_{W^{2,\infty}} \lim_{k \to \infty} \int_S \frac{g}{\|g\|_{W^{2,\infty}}} d \widehat{\alpha}_k^\sigma.
$$
This defines a measure $\bah_*$ on $\SSS$, whose total mass is given by
$$
\int_{\SSS} \widehat{\alpha}_*^a(d\theta)+\int_{\SSS} \widehat{\alpha}_*^p(d\theta) = 
\lim_{k \to \infty} 
\int_{\SSS} \widehat{\alpha}_k^a(d\theta) +\int_{\SSS} \widehat{\alpha}_k^a(d\theta) \in [0, 1],
$$
which proves the Cauchy completeness of $(\mathbb{M}_1(\SSS), d)$.

Next, we  prove that $(\mathbb{M}_1(\SSS), d)$ is totally bounded. 
For any integer $n$, we are going to construct a finite set 
$\mathcal{M}_{1,n} \subset \mathbb{M}_1(\SSS)$ such that
$$
\sup_{\bah \in \mathbb{M}_1(\SSS)} \inf_{\bah' \in \mathcal{M}_{1,n}} 
d(\bah,\bah')\leq \frac{1}{n}.
$$
For any $n \in \mathbb{N}$ and any $j \in \{0,1,\ldots,n-1\}$, 
we shorten $\theta_{j,n} = 2\pi j / n$, and $\theta_{n,n} = \theta_{0,n} = 0$. 
We can now define
$$
\mathcal{M}_{1,n} = 
\Bigg\{ \Big( \frac{1}{n^2}\sum_{j=0}^{n-1} k_j^a \delta_{\theta_{j,n}},\, \frac{1}{n^2}\sum_{j=0}^{n-1} k_j^a \delta_{\theta_{j,n}}  \Big)
 \Big|  k_j^a,k_j^p \in \{0,\dots,n^2\}, \quad \sum_j k_j^a+k_j^p \leq n^2 \Bigg\}.
$$
The inclusion $\mathcal{M}_{1,n} \subset \mathbb{M}_1(\SSS)$ 
is trivial thanks to the condition $\sum_j k_j^a+k_j^p \leq n^2$, 
and $\mathcal{M}_{1,n}$ is finite since the $k_j^\sigma$'s can each take only a finite number of values. 
We now prove that any $\bah \in \mathbb{M}_1(\SSS)$ is at distance at most $1/n$ 
of an element $\bah_n \in \mathcal{M}_{1,n}$. 

Fix $\bah \in \mathbb{M}_1(\SSS)$, and let
$$
k_j^\sigma = \lfloor n^2 \widehat{\alpha}^\sigma\big( [\theta_{j,n}, \theta_{j+1,n}) \big ) \rfloor.
$$
The facts $k_j^\sigma \in \{0,\dots,n^2\}$ 
and $\sum_j k_j^a+k_j^p \leq n^2$ can be easily verified. We now let
$$
\bah_n = \Bigg( \frac{1}{n^2}\sum_{j=0}^{n-1} k_j^a \delta_{\theta_{j,n}},\, \frac{1}{n^2}\sum_{j=0}^{n-1} k_j^a \delta_{\theta_{j,n}}  \Bigg),
$$
Now our task is to prove that $d(\bah,\bah_n)\leq 4/n$. 
Fix a function $g \in B^*$, we can write
$$
\begin{aligned}
    \sum_{\sigma=a,p} \int_\SSS g(\theta) (\widehat{\alpha}^\sigma - \widehat{\alpha}_n^\sigma)(d\theta)
    =& \sum_{\sigma=a,p} 
    \sum_{j=0}^{n-1} \Bigg( 
        \int_{[\theta_{j,n}, \theta_{j+1,n})} 
        g(\theta) \widehat{\alpha}^\sigma(d\theta) - \frac{k_j^\sigma}{n^2} g(\theta_{j,n}) \Bigg)\\
    =&\sum_{\sigma=a,p} \sum_{j=0}^{n-1} 
    \Big( \widehat{\alpha}^\sigma\big([\theta_{j,n}, \theta_{j+1,n})\big) g(\theta_{j,n}) - \frac{k_j^\sigma}{n^2} g(\theta_{j,n}) \Big)\\
    & + \sum_{\sigma=a,p} \sum_{j=0}^{n-1}  
    \int_{[\theta_{j,n}, \theta_{j+1,n})} (g(\theta) - g(\theta_{j,n})) \widehat{\alpha}^\sigma(d\theta)\\
    \leq & \sum_{\sigma=a,p} \sum_{j=0}^{n-1} 
     \|g\|_\infty \underbrace{\Big| 
        \widehat{\alpha}^\sigma\big([\theta_{j,n}, \theta_{j+1,n})\big)- \frac{k_j}{n^2} \Big|}_{\leq 1/n^2}\\
        &+\sum_{\sigma=a,p} \sum_{j=0}^{n-1} 
        \|g'\|_\infty\underbrace{\big|\theta_{j+1,n} - \theta_{j+1,n}\big|}_{\leq 1/n} 
        \int_{[\theta_{j,n}, \theta_{j+1,n})} \widehat{\alpha}^\sigma(d\theta)\\
        \leq & 2\frac{\|g\|_\infty + \|g'\|_\infty}{n} \leq 4/n
\end{aligned}
$$
The proof is complete. 
\end{proof}

\section{Replacement lemma~\ref{lemma:X1}}\label{asec:replaceR}

Before stating the proof, we introduce some notation. 
Let $P_t^{N}$ be the semi group of the active-passive exclusion process driven by the generator $L_N$ defined in~\eqref{eq:generator}, 
and $\mu_t^{N}=\mu^{N}P_t^{N}$ the distribution 
of the configuration at time $t$, where $\mu^{N}=\mu^{N}_{\bm{\zeta}}$ 
is the initial distribution defined in~\eqref{eq:zeta_theta} with density  
\begin{equation}\label{eq:f^N}
        f_0^N(\etah) \coloneqq\frac{d\mu^{N}}{d\mu^{*}_{\bm{\alpha}}}(\etah)=\prod_{x\in \TN} \Big[ \eta^a_x \frac{\zeta^a(x/N,\theta_x)}{\alpha^a/2\pi}
        +\eta^p_x\frac{\zeta^p(x/N,\theta_x)}{\alpha^p/2\pi} 
        +(1-\eta_x)\frac{1-\rho_0(x/N) }{1-\alpha} 
      \Big]
\end{equation}
The density $f_t^N$  is the solution of the forward Kolmogorov equation 
\begin{equation}\label{eq:kolmogorov}
    \begin{cases}
        \partial_t f_t^N = L_N^* f_t^N,\\
        f_0^N = \frac{d\mu^{N}}{d\mu^{*}_{\ba}},\\
    \end{cases}
\end{equation}
where $L_N^{*}$ is the adjoint of $L_N$ in $L^2(\mu^{*}_{\ba})$.  

The Dirichlet form is defined by $\EuScript{D}_{\bah}(h)=\big< h , -\LL h \big>_{\bah}$. 
If there is no ambiguity, we will omit the dependency in $\bah$, and simply denote it by $\Eu[D]$. 
We introduce 
$D(h)=\Eu[D](\sqrt{h})$. For a density $f$ w.r.t. 
    the translation invariant measure $\mu_{\ba}^{*}$ we denote by $H(f)=\E_{\ba}^*[f\log f]$ the entropy of the density
Given $\varphi$ a function on $\Sigma_N$, we denote by 
\[ 
    \big< \varphi \big>^l_x \coloneqq \frac{1}{|B_l|}\sum_{y\in B_l(x)}\tau_y \varphi
\]

By classical arguments, see Chapter 4 in~\cite{Erignoux21} or Chapter 5 in~\cite{kipnis1999scaling}, 
the proof of Lemma~\ref{lemma:X1} can be reduced to check $\overline{f}_T^N=1/T\int_0^T f_t^N(\etah)\dd t$ satisfies  
the condition ({\romannumeral 1}) and ({\romannumeral 2}) of the following one-block and two-block estimates, which will be done in the next subsections.   
\begin{lemma}[One-block estimate]\label{lemma:oneblock}
    Consider $\ba=(\alpha^a,\alpha^p)$ such that $\alpha=\alpha^a+\alpha^p  \in (0,1)$ and a density $f$ w.r.t. 
    the translation invariant measure $\mu_{\ba}^{*}$ satisfying
    \begin{enumerate}[(i)]
        \item There exists a constant $K_0$ such that for any $N$
    \[ H(f) \leq K_0 N^2 \quad \text{and} \quad D(f) \leq K_0. \]
    \item \begin{equation}
    \lim_{p \to \infty} \lim_{N \to \infty} \mathbb{E}_{\ba}^{*} \Bigg[f \frac{1}{N^2} \sum_{x \in \mathbb{T}_N^2} \mathbbm{1}_{E_p,x}\Bigg]   = 0.
    \end{equation}
    \end{enumerate}
    Then, for any cylinder function $g$,
    \[ \limsup_{l \to \infty} \limsup_{N \to \infty} 
    \mathbb{E}_{\ba}^{*} \Bigg[ f \frac{1}{N^2} \sum_{x \in \mathbb{T}_N^2} \tau_x \mathcal{V}^l \Bigg]  = 0, \]
    where $\mathcal{V}^l=\big< g \big>^l_0-\E_{\widehat{\bm{\rho}}_l}[g].$ 
\end{lemma}
\begin{lemma}[Two-block estimate]\label{lemma:twoblock}
    For any $\ba=(\alpha^a,\alpha^p)$ such that $\alpha \in (0,1)$ and any density $f$ satisfying conditions ({\romannumeral 1}) 
    and ({\romannumeral 2}) of Lemma~\ref{lemma:oneblock},
\begin{equation*}
    \limsup_{l \to \infty} \limsup_{\varepsilon \to 0} \limsup_{N \to \infty} \sup_{y \in B_{\varepsilon N}} 
    \mathbb{E}_{\alpha}^* 
    \Bigg[ f \frac{1}{N^2}\sum_{x \in \mathbb{T}_N^2} d\big(\tau_{x+y}\widehat{\bm{\rho}}_l,\tau_x\widehat{\bm{\rho}}_{\varepsilon N}\big) 
    \Bigg]= 0
\end{equation*}
\end{lemma}
The proof of Lemma~\ref{lemma:oneblock} is similar to the proof of Lemma 4.2.1 in~\cite{Erignoux21}, 
with adaption that in our proof, the new version of equivalence of ensembles in Proposition~\ref{prop:equivalence} is used.  
The proof of Lemma~\ref{lemma:twoblock} is similar to the proof of Lemma 4.2.2 in~\cite{Erignoux21}, 
with adaption that in our proof, 
the Lipschitz-continuity of the mapping $\Phi_{g}(\bah)=\E_{\bah}[g]$ in the newly defined metric space 
$(\mathbb{M}_1(\SSS),d)$ in Proposition~\ref{prop:Lipschitz} is used. 

\subsection{Control on the entropy and the Dirichlet form }
In this section we show that $\overline{f}_T^N=1/T\int_0^T f_t^N(\etah)\dd t$ satisfies  
the condition ({\romannumeral 1})  of Lemma~\ref{lemma:oneblock}. 
\begin{proposition}\label{prop:entropy} For any time $t > 0$, there exists a constant $K_0(t,\bm{\zeta},\ba,D_T,v_0)$ such that
    \[
    H\left(\frac{1}{t}\int_0^t f_s^N ds\right) \leq K_0 N^2 \quad \text{and} \quad D\left(\frac{1}{t}\int_0^t f_s^N ds\right) \leq K_0.
    \]
\end{proposition}
\begin{proof}
The relative entropy of $\mu^N_t$ w.r.t. reference measure $\mu^*_{\ba}$ is given by 
\[ 
H(\mu^N_t | \mu^*_{\ba}) =  H(f_t^N)=\mathbb{E}^*_{\ba}\big[ f_t^N \log f_t^N \big],
\]
according to the forward Kolmogorov equation~\eqref{eq:kolmogorov}, its time derivative is  
\[ 
\partial_t H(f_t^N) =
\mathbb{E}^*_{\ba}\big[ f_t^N L_N\log f_t^N \big].
\]
Using Lemma~\ref{lemma:A1} we have 
\[ 
    L_N \log f_t^N\leq \frac{2}{\sqrt{f_t^N}} L_N \sqrt{f_t^N}.
\]
Thus we split the full generator and get 
\[ 
    \partial_t H(f_t^N)\leq 
    -2D_T N^2 D(f_t^N)+
    2N \mathbb{E}_{\ba}^* \Big[ \sqrt{f_t^N} \LLWA \sqrt{f_t^N }\Big] 
    + \mathbb{E}_{\ba}^* \Big[ \sqrt{f_t^N} \LLR \sqrt{f_t^N }\Big]
\]
Integrating from time $0$ to $t$, we get 
\begin{equation}\label{eq:entropy_production}
        H(\mu^N_t | \mu^*_{\ba}) +
        2D_T N^2\int_{0}^{t}D(f_s^N)ds \leq H(\mu^N | \mu^*_{\ba})+
        2\int_{0}^{t}\mathbb{E}_{\ba}^* 
        \Big[ \sqrt{f_s^N} \big( N \LLWA +\LLR \big) \sqrt{f_s^N }\Big]ds 
\end{equation}
To bound  the entropy production of the asymmetric part, note that  
\begin{align*}
        \mathbb{E}_{\ba}^* &\big[ \sqrt{f_s^N}  \LLWA  \sqrt{f_s^N }\big]=
        \mathbb{E}^*_{\ba}\Big[ \sqrt{f_s^N} \sum_{x,|z|=1}\frac{v_0}{2} \mathbf{e}(\theta_x)\cdot z 
        \eta^a_x\nabla_{x,x+z}\sqrt{f_s^N}  \Big].
\end{align*}
Using elementary inequality
\[ 
\int fg\leq \frac{\gamma}{2}\int f^2 + \frac{1}{2\gamma}\int g^2,\quad \forall \gamma>0,
\]
we have 
\begin{align*}
    \mathbb{E}_{\ba}^* \Big[ \sqrt{f_s^N}  \LLWA  \sqrt{f_s^N }\Big] \leq
    \sum_{x,|z|=1}\frac{1}{2\gamma}\mathbb{E}^*_{\ba}\big[ 
        (\frac{v_0}{2} \mathbf{e}(\theta_x)\cdot z)^2f_s^N(\hat{\eta})\big]
    + \frac{\gamma}{2}\mathbb{E}^*_{\ba}\Big[ 
        \Big( \nabla_{x,x+z}\sqrt{f_s^N} \Big)^2  \Big].
\end{align*}
The first term can be bounded by 
\begin{equation}\label{eq:bound1}
        \sum_{x,|z|=1}\frac{1}{2\gamma}\mathbb{E}^*_{\ba}\big[ 
            (\frac{v_0}{2}\mathbf{e}(\theta_x)\cdot z)^2f_s^N(\hat{\eta})\big]\leq 
            \frac{N^2v_0^2}{2\gamma}, 
\end{equation}
and the second term is exactly $\gamma D(f_s^N)$. 
Finally, we let $\gamma=D_T N/2$ and derive 
\[ 
    2N\mathbb{E}_{\ba}^* \Big[ \sqrt{f_s^N}  \LLWA  \sqrt{f_s^N }\Big]
    \leq \frac{2v_0^2 N^2}{D_T} + D_T N^2D(f_s^N).
\] 
Since the angular dynamic is a diffusion, 
\begin{equation}\label{eq:bound2}
    \begin{aligned}
        \E_{\ba}^* \Big[ \sqrt{f_s^N}  \LLR  \sqrt{f_s^N }&\Big]=
        \sum_{x\in \TN}\E_{\ba}^* \Big[ \sqrt{f_s^N} \eta_x \partial_{\theta_x}^2 \sqrt{f_s^N }\Big]\\
        =&-\alpha\sum_{x\in \TN}\E_{\ba}^* \Big[ \Big( \partial_{\theta_x} \sqrt{f_s^N }\Big)^2\big| \eta_x=1 \Big]\leq 0.
    \end{aligned}
\end{equation}
Plugging~\eqref{eq:bound1} and~\eqref{eq:bound2} in~\eqref{eq:entropy_production}, we get 
\[ 
    H(\mu^N_t | \mu^*_{\ba}) +D_T N^2\int_{0}^{t}D(f_s^N)ds \leq H(\mu^N | \mu^*_{\ba})+t \frac{2v_0^2}{D_T} N^2.
\]
By Lemma~\ref{lemma:A1}, we have $H(\mu^N | \mu^*_{\ba})\leq C(\bm{\zeta},\bm{\alpha}) N^2$, thus 
\[ 
    H(\mu^N_t | \mu^*_{\ba}) + D_T N^2\int_{0}^{t}D(f_s^N)ds \leq \Big(C(\bm{\zeta},\bm{\alpha})+t\frac{2v_0^2 }{D_T}\Big) N^2.
\] 
Finally, by convexity of $H$ and $D$, we complete the proof of this proposition. 
\end{proof}

    \begin{lemma}
        \label{lemma:A2}
        Adopting the notations of Proposition~\ref{prop:entropy}, there exists a constant 
        $K(\bm{\zeta},\bm{\alpha})$ such that, 
        \[ 
        H(\mu^N|\mu^*_{\bm{\alpha}})\leq K(\bm{\zeta},\bm{\alpha})N^2.
        \]
    \end{lemma}
    \begin{proof}
        Recall that the initial density profile $\bm{\zeta}$ is continuous on the compact space 
        $\TT\times \SSS$, thus  
    \[ 
        \Big\|\eta^a_x \frac{\zeta^a(x/N,\theta_x)}{\alpha^a/2\pi}
        +\eta^p_x\frac{\zeta^p(x/N,\theta_x)}{\alpha^p/2\pi} 
        +(1-\eta_x)\frac{1-\rho_0(x/N) }{1-\alpha} 
          \Big\|_{L^{\infty}(\Sigma_N)}\leq  K(\bm{\zeta},\bm{\alpha}).
    \]
    Using this, we derive from~\eqref{eq:f^N} that 
    \[ 
        \Big\|\log \frac{d\mu^N}{d\mu^*_{\bm{\alpha}}}\Big\|_{L^{\infty}(\Sigma_N)}\leq 
         K(\bm{\zeta},\bm{\alpha})N^2. 
    \]
    \end{proof}
    \begin{lemma}
        \label{lemma:A1}
        Adopting the notations of Proposition~\ref{prop:entropy}, it holds that 
        \[ 
            L_N \log f^N_t \leq \frac{2}{\sqrt{f^N_t}}L_N \sqrt{f^N_t}.  
        \]
    \end{lemma}
    \begin{proof}In this proof we abbreviate the notation $f^N_t$ as $f$. 
        Thanks to the elementary inequality 
        \[ 
        \log b - \log a \leq \frac{2}{\sqrt{a}}(\sqrt{b}-\sqrt{a}),
        \]
        we can easily control the dynamic part of the full generator 
        \[ 
        \LLD \log f \leq \frac{2}{\sqrt{f}} \LLD \sqrt{f}.
        \]
        It remains to consider the rotation part of the full generator. 
        First, we observe that 
        \[ 
        2\partial_{\theta}^2 \sqrt{f}=
        -\frac{1}{2}\Big( \frac{1}{\sqrt{f}} \Big)^3 
        \big( \partial_{\theta} f \big)^2 + \frac{1}{\sqrt{f}} \partial_{\theta}^2 f,
        \]
        and  that 
        \[ 
        \begin{aligned}
            \partial_{\theta}^2 \log f&=-\Big( \frac{1}{\sqrt{f}} \Big)^4 
            \big( \partial_{\theta} f \big)^2 + 
            \Big( \frac{1}{\sqrt{f}} \Big)^2 \partial_{\theta}^2 f\\
            &\leq \frac{1}{\sqrt{f}}\Big( -\frac{1}{2} \Big( \frac{1}{\sqrt{f}} \Big)^3 
            \big( \partial_{\theta} f \big)^2 + \frac{1}{\sqrt{f}} \partial_{\theta}^2 f \Big)\\
            &= \frac{2}{\sqrt{f}}\partial_{\theta}^2 \sqrt{f}, 
        \end{aligned}
        \]
        with which we calculate that 
        \[ 
            \LLR \log f = \sum_{x\in \mathbb{T}^2_N} \eta_x \partial_{\theta_x}^2 \log f\leq 
            \frac{2}{\sqrt{f}} \sum_{x\in \mathbb{T}^2_N} \eta_x \partial_{\theta_x}^2 
            \sqrt{f}=\frac{2}{\sqrt{f}} \LLR \sqrt{f}.
        \]
        Combing dynamic part and rotation part the proof is completed. 
    \end{proof}

\subsection{High density estimate}

In this section we show that $\overline{f}_T^N=1/T\int_0^T f_t^N(\etah)\dd t$ satisfies  
the condition ({\romannumeral 2})  of Lemma~\ref{lemma:oneblock},  
i.e., large microscopic boxes are rarely fully occupied under the dynamics. 

Let us denote by \( E_{p,x} \) the event
\begin{equation}
E_{p,x} = \left\{ \sum_{y \in B_p(x)} \eta_y \leq |B_p(x)| - 2 \right\},
\end{equation}
on which the box of size \( p \) around \( x \) contains at least two empty sites. 
\begin{proposition}[Control on full clusters]\label{prop:full_cluster}
    For any positive time $T$,
\begin{equation}
    \lim_{p \to \infty} \lim_{N \to \infty} \mathbb{E}_{\mu_N^{\lambda}} \left( \int_0^T \frac{1}{N^2} \sum_{x \in \mathbb{T}_N^2} \mathbbm{1}_{E_p^c,x}(t) dt \right) = 0.
\end{equation}
\end{proposition}
\begin{proof}[Sketch of the proof of Proposition~\ref{prop:full_cluster}]
    Since the proof is analogous to the proof of Proposition 3.3.2 in~\cite{Erignoux21}, 
    we only give a sketch  to point out the main difference. 

    To prove that the box of microscopic size $p$ is not full, setting $p' = (2p + 1)^2$ the cardinal of $B_p$, it is enough to prove thanks to the microscopic setting that
    \[
    \int_{[0,T] \times \mathbb{T}^2} \rho_t^{p'}(u) du dt \underset{p' \to \infty}{\longrightarrow} 0,
    \]
    where $\rho_t(u)$ denotes the macroscopic density in $u$ at time $t$.
    
    We expect the total density $\rho$ to follow the partial differential equation
    \begin{equation*}
    \partial_t \rho = D_T \Delta \rho - v_0  \nabla \cdot [\mathbf{p}^a (\rho^a s(\rho)  +  d_s(\rho))],
    \end{equation*}
    which is obtained by integrating the PDE~\eqref{eq:pde} w.r.t. $\theta$.  
By letting $\phi(\rho) = 1/(1-\rho)$, and $\lambda=\frac{v_0}{2D_T}$, 
and note that 
\[ 
    \rho^a s(\rho)  +  d_s(\rho) \leq  \rho s(\rho)  +  d_s(\rho) = 1-\rho,
\]
we can therefore formally write, 
\begin{align*}
    \frac{1}{D_T}\partial_t \int_{\mathbb{T}^2} \phi(\rho_t) du &= 
    \int_{\mathbb{T}^2} \phi'(\rho_t) 
    \left[ \Delta \rho_t - 2\lambda \nabla \cdot [\mathbf{p}^a_t (\rho^a_t s(\rho_t)  +  d_s(\rho_t))]  \right] du \\
    &= \int_{\mathbb{T}^2} \phi''(\rho_t) 
    \left[ -  (\nabla \rho_t)^2 + 2\lambda  [\mathbf{p}^a_t (\rho^a_t s(\rho_t)  +  d_s(\rho))] \nabla \rho_t \right] du \\
    &\leq \int_{\mathbb{T}^2} \phi''(\rho_t) 
    \left[ -  (\nabla \rho_t)^2 + \frac{(\nabla \rho_t)^2}{2} + 2\lambda^2  [\mathbf{p}^a_t (\rho^a_t s(\rho_t)  +  d_s(\rho_t))]^2 \right] du \\
    &\leq \int_{\mathbb{T}^2} \phi''(\rho_t) 2\lambda^2 
    \|\mathbf{p}^a_t\|_\infty^2 (1-\rho_t)^2 du = 
    4\lambda^2  \|\mathbf{p}^a_t\|_\infty^2 \int_{\mathbb{T}^2} \phi(\rho_t) du.
\end{align*}
Applying Gronwall's inequality to obtain that for any time $t$,
\[
    \int_{\mathbb{T}^2} \phi(\rho_t) du \leq e^{4D_T \lambda^2  \|\mathbf{p}^a_t\|_\infty^2 t} \int_{\mathbb{T}^2} \phi(\rho_0) du.
\]
Furthermore, for any time $t$,
\[
\int_{\mathbb{T}^2} \phi(\rho_t) du \geq \frac{1}{\delta} \int_{\mathbb{T}^2} \mathbbm{1}_{\{\rho_t \geq 1-\delta\}} + \int_{\mathbb{T}^2} \mathbbm{1}_{\{\rho_t \leq 1-\delta\}} = \frac{1-\delta}{\delta} \int_{\mathbb{T}^2} \mathbbm{1}_{\{\rho_t \geq 1-\delta\}} + 1,
\]
therefore, for any time $t$,
\[
 \quad \int_{\mathbb{T}^2} \mathbbm{1}_{\{\rho_t \geq 1-\delta\}} \leq \frac{\delta}{1-\delta} \left[ e^{4D_T \lambda^2 \|\mathbf{p}^a_t\|_\infty^2 t} \int_{\mathbb{T}^2} \phi(\rho_0) du - 1 \right] \underset{\delta \to 0}{\longrightarrow} 0.
\]
As a consequence, for any time $t$, we could therefore write
\[
 \quad \iint_{[0,T] \times \mathbb{T}^2} \rho_t^{p'}(u) du dt \leq T (1-\delta)^{p'} + \iint_{[0,T] \times \mathbb{T}^2} \mathbbm{1}_{\{\rho_t \geq 1-\delta\}}.
\]
The first term in the right-hand side vanishes for any fixed $\delta$ as $p' \to \infty$, whereas the second becomes as small as needed letting $\delta \to 0$.
\end{proof}

\section{Proofs in Section~\ref{sec:main3}}\label{appsec:main3}
\begin{proof}[Proof of Lemma~\ref{lemma:8.52}]   
    Let  $\tilde{B}_l^i = \{x \in B_l ; x_i \leq l-1\}$, 
    using the identity
    $$
    \sum_{x \in \widetilde{B}_l^i} \tau_x j_i^{a,\Phi^a_i} = \mathcal{L}_l \sum_{x \in B_l} x_i \eta_x^{a,\Phi^a_i}.
    $$
    we obtain that 
\begin{equation*}
        \begin{aligned}
            \sum_{x \in B_{l_\psi}} \tau_x \psi =& 
                \sum_{x \in B_{l_\psi}} \tau_x \LL_l g + 
                \sum_{i=1,2} \sum_{x \in \widetilde{B}_l^i}\tau_x \big( j_i^{a,\Phi^a_i} + j_i^{p,\Phi^p_i} \big)\\
            &\quad -\sum_{i=1,2} \sum_{x \in \widetilde{B}_l^i \setminus B_{l_\psi}}\tau_x \big( j_i^{a,\Phi^a_i} + j_i^{p,\Phi^p_i} \big)\\
            =&\LL_l   \Big(
                  \sum_{x \in B_{l_\psi}} \tau_x g + 
                  \sum_{i=1,2} \sum_{x \in B_{l}}x_i \big( \eta_x^{a,\Phi^a_i} + \eta_x^{p,\Phi^p_i} \big)
                   \Big)\\
            &\quad -\sum_{i=1,2} \sum_{x \in \widetilde{B}_l^i \setminus B_{l_\psi}}
            \tau_x \big( j_i^{a,\Phi^a_i} + j_i^{p,\Phi^p_i} \big)\\
            &= \LL_l F + G.
        \end{aligned}
\end{equation*}
    where we shorten 
    $$ 
    \begin{aligned}
        &F =   \sum_{x \in B_{l_\psi}} \tau_x g + 
        \sum_{i=1,2} \sum_{x \in B_{l}}x_i \big( \eta_x^{a,\Phi^a_i} + \eta_x^{p,\Phi^p_i} \big),\\
        &G = -\sum_{i=1,2} \sum_{x \in \widetilde{B}_l^i \setminus B_{l_\psi}}
        \tau_x \big( j_i^{a,\Phi^a_i} + j_i^{p,\Phi^p_i} \big). 
    \end{aligned}
    $$
    Then, we can expand  
    \begin{equation}\label{eq:8.56}
     \Big<  (-\mathcal{L}_l)^{-1} \sum_{x \in B_l} \tau_x \psi , \sum_{x \in B_l} \tau_x \psi \Big>_{l,\widehat{K}_l} 
     = \big<  F (-\mathcal{L}_l) F \big>_{l,\widehat{K}_l} 
     - 2 \big<  F G \big>_{l,\widehat{K}_l}  
     + \big<  G (-\mathcal{L}_l)^{-1} G\big>_{l,\widehat{K}_l}. 
    \end{equation}

    First, we prove the last term in~\eqref{eq:8.56} vanishes in the limit. 
    By the Lemma~\ref{lemma:8.53},  
    we obtain that 
    \begin{align*}
        \big< h, G \big>_{l,\widehat{K}}\leq \frac{1}{2} C(\Phi^a,\Phi^p)|\widetilde{B}_l^i 
        \setminus B_{l_\psi}| +  \Eu[D]_{l,\widehat{K}}(h), 
    \end{align*}
    hence  
    by variational formula for the variance, we can get 
    $$
    \big<  G (-\mathcal{L}_l)^{-1} G\big>_{l,\widehat{K}_l} = 
    \sup_h \{ \big< h, G \big>_{l,\widehat{K}} - \Eu[D]_{l,\widehat{K}} (h)\}\leq C(\Phi^a,\Phi^p) 
    |\widetilde{B}_l^i \setminus B_{l_\psi}| = O(l), 
    $$
    and therefore, after multiplying $(2l+1)^{-2}$, the corresponding contribution vanishes in the limit.
 
    Regarding the second term in~\eqref{eq:8.56}, we write  
    \[ 
        \begin{aligned}
            \big<F,G\big>_{l,\widehat{K}}=& 
            \sum_{\substack{i=1,2 \\y \in B_l }} 
            \sum_{\substack{k=1,2 \\ x \in \widetilde{B}_l^i \setminus B_{l_\psi}}}y_i
            \Big<\eta_y^{a,\Phi^a_i}+\eta_y^{p,\Phi^p_i},  
            \tau_x \Big( j_k^{a,\Phi^a_k}+j_k^{p,\Phi^p_k} \Big)
            \Big>_{l,\widehat{K}}\\
            &\quad +\sum_{y \in B_{l_{\psi}} } \sum_{\substack{k=1,2 \\ x \in \widetilde{B}_l^i \setminus B_{l_\psi}}}
            \Big<\tau_y g ,  \tau_x \Big( j_k^{a,\Phi^a_k}+j_k^{p,\Phi^p_k} \Big)\Big>_{l,\widehat{K}}.
        \end{aligned}
    \]
    Elementary computations yield
    $$
    \Big<\eta_y^{a,\Phi^a_i}+\eta_y^{p,\Phi^p_i},  
    \tau_x \Big( j_k^{a,\Phi^a_k}+j_k^{p,\Phi^p_k} \Big) \Big>_{l,\widehat{K}} = C (\mathbbm{1}_{\{y=x\}} - \mathbbm{1}_{\{y=x+e_k\}}),
    $$
    where we shortened 
    $$C = \big<\Phi_i^a \Phi_k^a (\theta_0) \eta_0^a (1 - \eta_{e_k})\big>_{l,\widehat{K}}
    +\big<\Phi_i^p \Phi_k^p (\theta_0) \eta_0^p (1 - \eta_{e_k})\big>_{l,\widehat{K}}.$$
    Again after elementary computations, we get  
    $$
    \sum_{\substack{i=1,2 \\y \in B_l }} \sum_{\substack{k=1,2 \\ x \in \widetilde{B}_l^i \setminus B_{l_\psi}}}y_i
    \Big<\eta_y^{a,\Phi^a_i}+\eta_y^{p,\Phi^p_i},  \tau_x \Big( j_k^{a,\Phi^a_k}+j_k^{p,\Phi^p_k} \Big)\Big>_{l,\widehat{K}} = O(l).
    $$
    Similarly, for any $y$ such that $\{x, x + e_k\} \cap B_{s_g}(y) = \emptyset$, we have 
    $$\Big<\tau_y g, \tau_x j_k^{\sigma,\Phi^\sigma_k}\Big>_{l,\widehat{K}}  = 0, $$ so that
    $$
    \sum_{y \in B_{l_{\psi}} } \sum_{\substack{k=1,2 \\ x \in \widetilde{B}_l^i \setminus B_{l_\psi}}}
            \Big<\tau_y g ,  \tau_x \Big( j_k^{a,\Phi^a_k}+j_k^{p,\Phi^p_k} \Big)\Big>_{l,\widehat{K}} = O(l).
    $$Combing, we get  the second term in~\eqref{eq:8.56}
    $$
    \left<F,G\right>_{l,\widehat{K}} \leq C(\psi)|\widetilde{B}_l^i \setminus B_{l_\psi}|  = O(l), 
    $$
    which after multiplying $(2l+1)^{-2}$ vanishes  in the limit as well.
    
    Finally we only need to compute the first term in~\eqref{eq:8.56}
\begin{equation}\label{eq:8.54}
            \begin{aligned}
                \big<  F (-\mathcal{L}_l) F \big>_{l,\widehat{K}_l}=& 
                \Big<\sum_{\substack{i=1,2 \\x \in B_l }} x_i (\eta_x^{a,\Phi^a_i}+\eta_x^{p,\Phi^p_i}),  
                (-\mathcal{L}_l) \sum_{\substack{i=1,2 \\x \in B_l }} x_i (\eta_x^{a,\Phi^a_i}+\eta_x^{p,\Phi^p_i})\Big>_{l,\widehat{K}}\\
                & +
                \Big<\sum_{x \in B_{l_{\psi}} } \tau_x g ,  (-\mathcal{L}_l) \sum_{x \in B_{l_{\psi}} } \tau_x g \Big>_{l,\widehat{K}}\\
                & + 
                2\Big<\sum_{x \in B_{l_{\psi}} } \tau_x g ,  
                (-\mathcal{L}_l) \sum_{\substack{i=1,2 \\x \in B_l }} x_i (\eta_x^{a,\Phi^a_i}+\eta_x^{p,\Phi^p_i}) \Big>_{l,\widehat{K}}.
            \end{aligned}
\end{equation}
    We deal these terms on the right hand side of~\eqref{eq:8.54} separately. For the second term, we have
    $$
    \Big<\sum_{x \in B_{l_{\psi}} } \tau_x g ,  (-\mathcal{L}_l) \sum_{x \in B_{l_{\psi}} } \tau_x g \Big>_{l,\widehat{K}}
     = \frac{1}{2} \sum_{y,y+z \in B_l} \Big< \Big[ \nabla_{y,y+z} \sum_{x \in B_{l_\psi}} \tau_x g \Big]^2  \Big>_{l,\widehat{K}}   .
    $$
    Suppose $s_g\geq 1$, by definition of $\psi$ and 
    \[ 
    \LL g = \sum_{(x,x+z )\in B_{s_g+1}}a(x,x+z)\widetilde{\nabla}_{x,x+z}g,  
    \]
    we have $s_{\psi}=s_{\LL g}=s_g+1$. 
    Since for any $ y \in B_l\cap B_{l_{\psi}-s_g-1} $,  
    $$\{y,y+z\}\cap B_{s_g}(x)=\emptyset, \quad  \forall x \in \mathbb{Z}^2 \setminus B_{l_\psi},$$ 
    i.e. $\tau_xg$ cannot feel the change of configuration at $y,y+z$. 
    Thus we can write
    $$
    \nabla_{y,y+z} \sum_{x \in B_{l_\psi}} \tau_x g = \nabla_{y,y+z} \Sigma_g,\quad \forall y \in B_l\cap B_{l_{\psi}-s_g-1}.
    $$ 
    Furthermore, for any $y \in B_l \setminus B_{l_{\psi}-s_g-1}$,
    $$
    \Bigg[\nabla_{y,y+z} \sum_{x \in B_{l_\psi}} \tau_x g\Bigg]^2 = 
    \Bigg[\nabla_{y,y+z} \sum_{x\in B_{2s_g+1}(y)} \tau_x g\Bigg]^2 \leq C(s_g) \|g\|^2_\infty.
    $$
    Since all the $\nabla_{y,y+z} \Sigma_g$ have the same distribution under $\mu_{l,\widehat{K}}$ for $y \in B_l\cap B_{l_{\psi}-s_g-1}$, 
    and $l_{\psi}-s_g-1 = l- 2s_g -3$, 
    we can therefore defive that 
    \begin{align*}
    &\frac{1}{(2l+1)^2} \Big<\sum_{x \in B_{l_{\psi}} } \tau_x g ,  
    (-\mathcal{L}_l) \sum_{x \in B_{l_{\psi}} } \tau_x g \Big>_{l,\widehat{K}} \\
    =& \frac{|B_{l_{\psi}-s_g-1}|}{(2l+1)^2} \sum_{i=1,2} \mathbb{E}_{l,\widehat{K}}
     \big[(\nabla_{0,e_i} \Sigma_g)^2\big] + C(\psi )O
     \Big(\frac{|B_l \setminus B_{l_{\psi}-s_g-1}|}{(2l+1)^2}\Big) \\
    =& \sum_{i=1}^2 \big<\big[\nabla_{0,e_i} \Sigma_g\big]^2\big>_{l,\widehat{K}} + C(\psi )O(1/l)\\
    =& \sum_{i=1}^2 \Big<\big[\eta^a_0(1-\eta_{e_i})\widetilde{\nabla}_{0,e_i} \Sigma_g\big]^2+
    \big[\eta^p_0(1-\eta_{e_i})\widetilde{\nabla}_{0,e_i} \Sigma_g\big]^2\Big>_{l,\widehat{K}}  +
     C(\psi )O(1/l).
    \end{align*}
    
    We now consider the first term in~\eqref{eq:8.54}
    We temporarily let $f =\sum_{k} \sum_{x} x_k (\eta_x^{a,\Phi^a_k}+\eta_x^{p,\Phi^p_k})$, and  
    we can compute that 
    \[ 
    \widetilde{\nabla}_{y,y+e_i}f = \tau_y (-\del_i \eta^{a,\Phi^a_i}-\del_i \eta^{p,\Phi^p_i}). 
    \]
    Then, by the fact that under $\mu_{l,\widehat{K}_l}$ all the terms have the same distribution, 
    \begin{align*}
        \frac{1}{(2l+1)^2}\big< f , (-\mathcal{L}_l) f \big>_{l,\widehat{K}} 
        =& \frac{1}{2(2l+1)^2} \sum_{i=1,2} \sum_{(y,y+e_i ) \subset B_l} 
        \Big< a(y,y+e_i) \Big( \widetilde{\nabla}_{y,y+e_i}f  \Big)^2  \Big>_{l,\widehat{K}}\\
        =& \frac{1}{2(2l+1)^2} \sum_{i=1,2} \sum_{(y,y+e_i ) \subset B_l} 
        \Big< \tau_y \Big( j_i^{a,\Phi^a_i}+j_i^{p,\Phi^p_i} \Big)^2  \Big>_{l,\widehat{K}}\\
        =& \frac{1}{2} \sum_{i=1,2}  \Big<  \Big( j_i^{a,\Phi^a_i}+j_i^{p,\Phi^p_i} \Big)^2  \Big>_{l,\widehat{K}} + C(\psi)O(1/l)\\
        =& \sum_{\sigma=a,p}\sum_{i=1,2} 
         \big<  \big[ \Phi^\sigma_i(\theta_0)\eta^\sigma_0(1-\eta_{e_i}) \big]^2  \big>_{l,\widehat{K}} + C(\psi)O(1/l)
    \end{align*}
    For the last term in~\eqref{eq:8.54}, by the fact that
    \begin{align*}
        \big<g,  (-\mathcal{L}_l) f \big>_{l,\widehat{K}}=
        \frac{1}{2} \sum_{i=1,2} \sum_{(y,y+e_i ) \subset B_l} 
        \big< a(y,y+e_i) \big( \widetilde{\nabla}_{y,y+e_i}f  \big) \big( \widetilde{\nabla}_{y,y+e_i}g  \big)   
        \big>_{l,\widehat{K}}
    \end{align*}
    and by straight forward adapting the previous operations to the cross term, we obtain that 
    \begin{align*}
        &\frac{1}{(2l+1)^2}
        \Big<\sum_{x \in B_{l_{\psi}} } \tau_x g ,  
        (-\mathcal{L}_l) \sum_{\substack{i=1,2 \\x \in B_l }} x_i (\eta_x^{a,\Phi^a_i}+\eta_x^{p,\Phi^p_i}) \Big>_{l,\widehat{K}} \\
        =& \frac{1}{2} \sum_{i=1,2}  \Big<  \Big( j_i^{a,\Phi^a_i}+j_i^{p,\Phi^p_i} \Big), 
        \widetilde{\nabla}_{0,e_i}\Sigma_g     \Big>_{l,\widehat{K}} + C(\psi)O(1/l)\\
        =& \sum_{\sigma=a,p}\sum_{i=1,2}  \Big<   \Phi^\sigma_i(\theta_0)\eta^\sigma_0(1-\eta_{e_i})\widetilde{\nabla}_{0,e_i}\Sigma_g  
          \Big>_{l,\widehat{K}} + C(\psi)O(1/l). 
    \end{align*}
    Combing, the first term in~\eqref{eq:8.56} is   
    \[ 
        \begin{aligned}
             \frac{1}{(2l+1)^2} \big<  F (-\mathcal{L}_l) F \big>_{l,\widehat{K}_l}
            =&\sum_{\sigma=a,p}\sum_{i=1,2}  
            \big<\big[\eta^\sigma_0(1-\eta_{e_i})\widetilde{\nabla}^{0,e_i} \Sigma_g\big]^2\big>_{l,\widehat{K}}  \\
            &+\sum_{\sigma=a,p}\sum_{i=1,2}  \big<  \big[\eta^\sigma_0(1-\eta_{e_i}) \Phi^\sigma_i(\theta_0) \big]^2    \big>_{l,\widehat{K}} \\
            &+\sum_{\sigma=a,p}\sum_{i=1,2} 2 \big<   \eta^\sigma_0(1-\eta_{e_i})\Phi^\sigma_i(\theta_0)\widetilde{\nabla}^{0,e_i}\Sigma_g    \big>_{l,\widehat{K}}\\
            &+ C(\psi )O(1/l). 
        \end{aligned}
    \]
    In the limit, the equivalence of ensembles finally yields 
    \begin{equation}\label{eq:91}
    \lim_{l \to \infty} \frac{1}{(2l+1)^2} \big<  F (-\mathcal{L}_l) F \big>_{l,\widehat{K}_l} =
    \sum_{\sigma=a,p}\sum_{i=1}^2 \mathbb{E}_{\bah} 
    \big[ \eta^\sigma_0(1-\eta_{e_i})\big(\Phi^\sigma_i(\theta_0) + \widetilde{\nabla}^{0,e_i}\Sigma_g  \big)^2\big]
    \end{equation}
    for any sequence $\widehat{K}_l$ 
    such that $\bah_{\widehat{K}_l} \to \bah$. 
    \end{proof}
    \begin{lemma}\label{lemma:8.53}
        For any $\psi \in \mathcal{C}_0 + J^* + \mathcal{L}\mathcal{C}$, 
        there exists a constant $C(\psi)$ such that for any $l$, 
        $\widehat{K} \in \widetilde{\mathbb{K}}_l$, 
        $h \in \mathcal{C}$ with $\supp(h) \subset B_l$, $\gamma > 0$, and $A \subset B_{l,\psi}$
        $$
        \Big<h, \sum_{x \in A} \tau_x \psi\Big>_{l,\widehat{K}}  \leq 
        \gamma C(\psi)|A| + \frac{1}{2\gamma} \Eu[D]_{l,\widehat{K}}^{A_\psi}(h),
        $$
        where we shortened $A_{\psi} = \{x \in B_l, d(x,A) \leq s_{\psi}\}$, 
        $\Eu[D]_{l,\widehat{K}}^A(h) = \big<h, (-\mathcal{L}_A)h\big>_{l,\widehat{K}}$ 
        and $\mathcal{L}_A$ is the SSEP generator restricted to jumps with both ends in set $A$.
    \end{lemma}
    The proof of Lemma \ref{lemma:8.53} is similar to the proof of Lemma 8.3.2 in~\cite{Erignoux21}, hence we omit it. 

    \begin{proof}[Proof of Lemma~\ref{lemma:8.51}]
    The proof of Lemma~\ref{lemma:8.51} is similar to the proof of Lemma~\ref{lemma:8.52}. 
    Using the same notations as for the proof of Lemma~\ref{lemma:8.52}, we write 
    \begin{equation*}
    \sum_{x \in B_l} \tau_x \psi = \mathcal{L}_l F + G,
    \end{equation*}
    where 
    $$ 
    \begin{aligned}
        &F =   \sum_{x \in B_{l_\psi}} \tau_x g + 
        \sum_{i=1,2} \sum_{x \in B_{l}}x_i \big( \eta_x^{a,\Phi^a_i} + \eta_x^{p,\Phi^p_i} \big),\\
        &G = -\sum_{i=1,2} \sum_{x \in \widetilde{B}_l^i \setminus B_{l_\psi}}
        \tau_x \big( j_i^{a,\Phi^a_i} + j_i^{p,\Phi^p_i} \big). 
    \end{aligned}
    $$

    Given $\varphi \in \mathcal{T}_0$, we rewrite the left-hand side in~\eqref{eq:8.51} as 
    \begin{equation}
    \lim_{l \to \infty} \frac{1}{(2l+1)^2} \Big<  -F + (-\mathcal{L}_l^{-1})G ,
    \sum_{x \in B_{l_\varphi}} \tau_x \varphi \Big>_{l,\widehat{K}_l}.
    \end{equation}
    Since $\varphi$ is centered, we can easily get 
    \[ 
        \begin{aligned}
            &\frac{1}{(2l+1)^2} \Big<   \sum_{y \in B_{l_\psi}} \tau_y g ,\sum_{x \in B_{l_\varphi}} \tau_x \varphi \Big>_{l,\widehat{K}_l}\\
            =& 
            \frac{1}{(2l+1)^2}\sum_{\substack{x\in B_{l_\varphi} ,y\in B_{l_\psi} \\  B_{s_g}(y)\cap  B_{s_\varphi}(x) \neq \emptyset} } 
            \big<  \tau_y g , \tau_x \varphi \big>_{l,\widehat{K}_l}\\
            =& \sum_{y \in B_{s_g+s_\varphi} }\big<   \tau_y g , \varphi \big>_{l,\widehat{K}_l} + C(\psi)O(1/l) \\
            =& \big<  \Sigma_g , \varphi \big>_{l,\widehat{K}_l} + C(\psi)O(1/l),  
        \end{aligned}
    \]
    as well as for each $i=1,2$, 
    \[ 
        \begin{aligned}
            &\frac{1}{(2l+1)^2} \Big<   \sum_{y \in B_l } y_i (\eta_y^{a,\Phi^a_i}+\eta_y^{p,\Phi^p_i}) ,
            \sum_{x \in B_{l_\varphi}} \tau_x \varphi \Big>_{l,\widehat{K}_l}\\
            =&\frac{1}{(2l+1)^2} \sum_{x\in B_{l_\varphi} }\sum_{ y\in B_{s_\varphi}(x) } 
            \big<   y_i(\eta_y^{a,\Phi^a_i}+\eta_y^{p,\Phi^p_i}), \tau_x \varphi \big>_{l,\widehat{K}_l}\\
            =& \sum_{y \in B_{s_\varphi} }\big<   y_i(\eta_y^{a,\Phi^a_i}+\eta_y^{p,\Phi^p_i}) , \varphi \big>_{l,\widehat{K}_l}+ C(\psi)O(1/l)\\
            =& \sum_{y \in \mathbb{Z}^2 } \big<  y_i(\eta_y^{a,\Phi^a_i}+\eta_y^{p,\Phi^p_i}) , 
            \varphi \big>_{l,\widehat{K}_l}+ C(\psi)O(1/l).
        \end{aligned}
    \]
    Using once again the equivalence of ensembles, we obtain 
    \begin{equation}\label{eq:90}
    \lim_{l \to \infty} \frac{1}{(2l+1)^2}  
    \Big<  -F  ,\sum_{x \in B_l} \tau_x \varphi \Big>_{l,\widehat{K}_l} 
    = -\Big<\Sigma_g + \sum_{i=1,2} \sum_{x \in \mathbb{Z}^2 } x_i(\eta_x^{a,\Phi^a_i}+\eta_x^{p,\Phi^p_i}) , 
    \varphi\Big>_{\bah},
    \end{equation}
    
    Therefore, it remains to prove that the contribution of $G$ vanishes. 
    The contribution of $G$ can be rewritten as
    \begin{align*}
        &\Big<   (-\mathcal{L}_l^{-1})G ,\sum_{x \in B_{l_\varphi}} \tau_x \varphi \Big>_{l,\widehat{K}_l} \\
        =& \Big<   (-\mathcal{L}_l^{-1})G ,
        (-\mathcal{L}_l)(-\mathcal{L}_l^{-1})\sum_{x \in B_{l_\varphi}} \tau_x \varphi \Big>_{l,\widehat{K}_l} \\
        =& \frac{1}{2}\sum_{(x,x+e_i)\subset B_l} 
        \Big< a(x,x+e_i) \widetilde{\nabla}^{x,x+e_i} (-\mathcal{L}_l^{-1})G, 
         \widetilde{\nabla}^{x,x+e_i} (-\mathcal{L}_l^{-1})\sum_{x \in B_{l_\varphi}} \tau_x \varphi \Big>_{l,\widehat{K}_l} \\
        =& \frac{1}{2}\sum_{x,x+e_i\in B_l} 
        \Big<  \nabla_{x,x+e_i} (-\mathcal{L}_l^{-1})G,  
         \nabla_{x,x+e_i} (-\mathcal{L}_l^{-1})\sum_{x \in B_{l_\varphi}} \tau_x \varphi
          \Big>_{l,\widehat{K}_l}.
    \end{align*}
    Using Holder's inequality, 
    we obtain that for any positive $\gamma$, 
    \begin{align*}
        &\Bigg|\frac{1}{(2l+1)^2} \Big<   (-\mathcal{L}_l^{-1})G ,
        \sum_{x \in B_{l_\varphi}} \tau_x \varphi \Big>_{l,\widehat{K}_l} \Bigg| \\
        \leq & \frac{\gamma}{4(2l+1)^2} \Big<  G, (-\mathcal{L}_l^{-1})G   \Big>_{l,\widehat{K}_l}
        +\frac{1}{4\gamma(2l+1)^2} \Big<\sum_{x \in B_{l_\varphi}} \tau_x \varphi ,   
        (-\mathcal{L}_l^{-1})\sum_{x \in B_{l_\varphi}} \tau_x \varphi \Big>_{l,\widehat{K}_l}. 
    \end{align*}
    In the proof of Lemma~\ref{lemma:8.52}, we have already shown 
    that the first term in the right-hand-side is $O(\gamma l^{-1})$, 
    whereas in the limit $l \rightarrow \infty$ the second is bounded by 
    $\iip[f]_{\bah}/\gamma$ according to Lemma~\ref{lemma:8.50}.  
    We can therefore choose $\gamma = \sqrt{l}$ 
    to obtain that both terms vanish as $l \rightarrow \infty$, thus concluding the proof of Lemma~\ref{lemma:8.51}.
    \end{proof}

We now consider the case $\varphi \in \mathcal{T}_0^\omega$, and rigorize~\eqref{eq:8.55},   
using the decomposition of the germs of closed forms obtained in Proposition~\ref{thm:1} 
and the generalized integration by parts formula in Lemma~\ref{lemma:integration_by_parts}. 
The proof of Lemma~\ref{lemma:integration_by_parts} is similar to the proof of Lemma 8.3.1 in~\cite{Erignoux21}, 
hence we omit it here. 
\begin{lemma}[Generalized integration by parts formula]\label{lemma:integration_by_parts}
    Let $\psi \in \mathcal{C}_0$ be a cylinder function, and let $a \subset B_{s_\psi}$ be an oriented edge in its domain. Then, $\psi$ is in the range of the generator $\LL_{s_\psi}$, 
    and we can define the ``primitive'' $I_a(\psi)$ of $\psi$ with respect to the gradient along the oriented edge $a$ as
\[
I_a(\psi) = \frac{1}{2} \nabla_a (-\LL_{s_\psi})^{-1} \psi.
\]
Furthermore, for any $B \subset \TN$ containing $B_{s_\psi}$, any canonical state $\widehat{K}=(K^a, \Theta^a ,K^p, \Theta^p )$  on $B$
such that $K=K^a+K^p \leq |B|$ and any $h \in \mathcal{C}$ with $\supp(h)\subset B$, we have
\begin{equation*}
\label{eq:lemma_integration_by_parts}
\big< \psi, h  \big>_{B, \hat{K}} = \sum_{a \subset B_{s_\psi}} 
\big< I_a(\psi),\,\nabla_a h  \big>_{B, \hat{K}}.
\end{equation*}
This result is also true if canonical measure $\mu_{B, \widehat{K}}$ 
is replaced by a grand-canonical measure $\mu_{\bah}$. 
Note that if $K = |B| - 1$ or $K = |B|$ the result is trivial because $\psi$ vanishes.
\end{lemma}

\begin{remark}\label{rk:8.55}
    By definition~\ref{def:iip}, for any $\varphi\in \To$,  
    denoting $j^{u,v,m,n}=u^\top j^a+v^\top j^p + m^\top j^{a,\omega }+ n^\top j^{p,\omega }$,  we can write 
    \begin{align*}
         &\iip[\varphi]_{\bah} = \sup_{\substack{g\in \mathcal{T}_0^\omega \\ j \in J^\omega } }
        \Big\{ 2 \iip[\varphi,\, j+\LL g]_{\bah} - \iip[ j+\LL g]_{\bah} \Big\}\\
        =&\sup_{\substack{g\in \To\\u,v,m,n \in \mathbb{R}^2}} 
        \Bigg\{ 
            2\E_{\bah }\Big[
                \varphi \big( \Sigma_g + 
                \sum_{i=1,2} \sum_{y \in \mathbb{Z}^2 } y_i
                (u_i\eta_y^{a}+v_i\eta_y^{p}+m_i\eta_y^{a,\omega}+n_i\eta_y^{p,\omega})
                 \big)\Big] \\
            &\qquad\qquad\qquad\qquad\qquad\qquad\qquad\qquad\qquad\qquad\qquad \quad - \iip[ \mathcal{L}g + j^{u,v,m,n}]_{\widehat{\alpha}}
            \Bigg\}.
    \end{align*}
\end{remark}
We split the proof of Lemma~\ref{lemma:8.50} in two lemmas, 
namely an upper and a lower bound. 
The lower bound is relatively easy to prove.
\begin{lemma}
    Under the assumption of Lemma~\ref{lemma:8.50},
    \begin{equation*}
    \limsup_{l \to \infty} \frac{1}{(2l + 1)^2}  
    \Big< (-\mathcal{L}_l^{-1}) \sum_{x \in B_{l_\varphi}} \tau_x \varphi \cdot \sum_{x \in B_{l_\varphi}} \tau_x \varphi \Big>_{l,\widehat{K}_l} 
    \geq \iip[\psi]_{\bah}.
    \end{equation*}  
\end{lemma}
\begin{proof}
    This proof is analogous to the proof of Proposition 8.5.3 in~\cite{Erignoux21}. 

    By denoting 
    \[ 
    F_l^{g,u,v,m,n} \coloneqq \sum_{x \in B_{l_g}} \tau_x g + \sum_{i=1,2} \sum_{x \in B_l } x_i
    (u_i\eta_x^{a}+v_i\eta_x^{p}+m_i\eta_x^{a,\omega}+n_i\eta_x^{p,\omega}). 
\]
the variational formula for variance gives 
\[ 
\begin{aligned}
        &\Big< (-\mathcal{L}_l^{-1}) \sum_{x \in B_{l_\varphi}} \tau_x \varphi , 
        \sum_{x \in B_{l_\varphi}} \tau_x \varphi \Big>_{l,\widehat{K}_l}\\ 
        =& \sup_{h\in L^2(\mu_{l,\widehat{K}_l}) } 
        \Big\{ 2  \Big< h, \sum_{x \in B_{l_\varphi}} \tau_x \varphi \Big>_{l,\widehat{K}_l} 
        -\Eu[D]_{l,\widehat{K}_l}(h) \Big\},\\
        \geq& \sup_{\substack{g\in \To\\u,v,m,n \in \mathbb{R}^2}} 
        \Big\{ 2  \Big< F_l^{g,u,v,m,n}, \sum_{x \in B_{l_\varphi}} \tau_x \varphi \Big>_{l,\widehat{K}_l} 
        - \Eu[D]_{l,\widehat{K}_l}(F_l^{g,u,v,m,n}) \Big\},
\end{aligned}
\]
By similar argument used to derive~\eqref{eq:90}, we get  
\begin{align*}
    &\lim_{l \to \infty} \frac{1}{(2l+1)^2} \Big< F_l^{g,u,v,m,n}, \sum_{x \in B_{l_\varphi}} \tau_x \varphi \Big>_{l,\widehat{K}_l} \\
=&\E_{\bah }\Big[
    \varphi \big( \Sigma_g + 
    \sum_{i=1,2} \sum_{y \in \mathbb{Z}^2 } y_i
    (u_i\eta_y^{a}+v_i\eta_y^{p}+m_i\eta_y^{a,\omega}+n_i\eta_y^{p,\omega})
     \big)\Big].
\end{align*}
and by similar argument used to derive~\eqref{eq:91} we get 
\begin{equation*}
\lim_{l \to \infty} \frac{1}{(2l+1)^2} \Eu[D]_{l,\widehat{K}_l} (F_l^{g,u,v,m,n}) = 
\iip[ \mathcal{L}g + j^{u,v,m,n}]_{\bah}.
\end{equation*}
\end{proof}
We now state and prove the upper bound. 
\begin{lemma}
    Under the assumptions of Lemma~\ref{lemma:8.50}, for any $\varphi \in \To$,
    \begin{equation*}
    \limsup_{l \to \infty} \frac{1}{(2l + 1)^2}  
    \Big< (-\mathcal{L}_l^{-1}) \sum_{x \in B_{l_\varphi}} \tau_x \varphi ,\, 
    \sum_{x \in B_{l_\varphi}} \tau_x \varphi \Big>_{l,\widehat{K}_l} \leq \iip[\psi]_{\bah}.
    \end{equation*}
\end{lemma}
\begin{proof}This proof is analogous to the proof of Proposition 8.5.3 in~\cite{Erignoux21}.
    We start by replacing the canonical measure $\mu_{\widehat{K}_l,l}$ 
    by the grand-canonical measure $\mu_{\bah}$ thanks to the equivalence of ensembles stated in 
    Proposition~\ref{prop:equivalence}. 
    The main obstacle in doing so is that the support of $\sum_{x \in B_{l_\varphi}} \tau_x \varphi$ 
    grows with $l$.

    By the variational formula for the variance, we can write for any canonical state $\widehat{K}$
    \begin{equation}\label{eq:92}
        \Big< (-\mathcal{L}_l^{-1}) \sum_{x \in B_{l_\varphi}} \tau_x \varphi \cdot \sum_{x \in B_{l_\varphi}} \tau_x \varphi \Big>_{l,\widehat{K}_l} 
        = \sup_{h } 
        \Bigg\{ 2 \Big< \sum_{x \in B_{l_\varphi}} \tau_x \varphi, h \Big>_{l,\widehat{K}} - 
        \Eu[D]_{l,\widehat{K}}(h) \Bigg\},
    \end{equation}
    In this formula the supremum is taken over all $h\in T_l^\omega = \mathcal{C}_l \cap T^\omega$. 
    As in the proof of the one-block-estimate see Section 4.3 in~\cite{Erignoux21}, let $k$ be an integer that will go to infinity after $l$, 
    and let us partition $B_l$ into disjoint boxes $\widetilde{\Lambda}_0, \ldots, \widetilde{\Lambda}_p$, 
    where $p = \lfloor \frac{2l_\varphi + 1}{2k + 1} \rfloor^2$,  $\widetilde{\Lambda}_j = B_{k}(x_j)$
    for any $1 \leq j \leq p$ and for some sites $ x_1,\dots,x_p$. Specially, 
    $\widetilde{\Lambda}_0 = B_{l_\varphi} \setminus \cup_{j=1}^p \widetilde{\Lambda}_j$. we now define
    $$
    \Lambda_j = \big\{ x \in \widetilde{\Lambda}_j, d(x, \widetilde{\Lambda}_j^c) > s_\varphi \big\} \quad \text{and} \quad \Lambda_0 = B_{l_\varphi} \setminus \bigcup_{j=1}^p \Lambda_j.
    $$
    One easily obtains that for some universal constant $C$, $|\Lambda_0| \leq Cs_\varphi(l^2/k + lk)$.

    Let $h$ be a function in $T_l^\omega$, we can split
    \begin{equation*}
    \sum_{x \in B_{l_\varphi}}\big<  \tau_x \varphi, h \big>_{l,\widehat{K}} = 
    \sum_{j=1}^p \sum_{x \in \Lambda_j} \big<  \tau_x \varphi, h \big>_{l,\widehat{K}} + 
    \sum_{x \in \Lambda_0} \big<  \tau_x \varphi, h \big>_{l,\widehat{K}}.
    \end{equation*}

    Letting $\gamma = \sqrt{k}/2$ in Lemma~\ref{lemma:8.53}, for any $l \geq k^2$, 
    \[ 
         \sum_{x \in \Lambda_0} \big<  \tau_x \varphi, h \big>_{l,\widehat{K}}
        \leq \frac{\sqrt{k}}{2}C(\varphi)|\Lambda_0|+\frac{1}{\sqrt{k}}\Eu[D]_{l,\widehat{K}}^{\Lambda_{0,\varphi}}(h)
        \leq \frac{1}{\sqrt{k}}\Big( C(\varphi)l^2 + \Eu[D]_{l,\widehat{K}}(h) \Big),
    \]
    which implies 
    \[ 
        \begin{aligned}
            &2 \Big< \sum_{x \in B_{l_\varphi}} \tau_x \varphi, h \Big>_{l,\widehat{K}} - \Eu[D]_{l,\widehat{K}}(h)\\
            \leq & 2\sum_{j=1}^p \sum_{x \in \Lambda_j} \big<  \tau_x \varphi, h \big>_{l,\widehat{K}} 
            - \Big( 1 - \frac{2}{\sqrt{k}} \Big)\Eu[D]_{l,\widehat{K}}(h) +  \frac{C(\varphi)l^2}{\sqrt{k}}. 
        \end{aligned}
    \]
    Thus, by letting $c_k = 1 - 2k^{-1/2}$, the left-hand side of~\eqref{eq:92} is therefore less than
    $$
    c_k \sup_{h \in T_l^\omega} \Bigg\{ \frac{2}{c_k}\sum_{j=1}^p \sum_{x \in \Lambda_j} <\tau_x \psi,h>_{l,\widehat{K}} - \Eu[D]_{l,\widehat{K}}(h) \Bigg\} + \frac{C(\psi)l^2}{\sqrt{k}}.
    $$
    For any $h \in T_l^\omega$, $1 \leq j \leq p$ define 
    $h_j = \mathbb{E}_{l,\widehat{K}}[h | \etah_{y},y\in \widetilde{\Lambda}_j]$, by convexity of the Dirichlet form, we have
    \begin{equation}
        \sum_{j=1}^p \Eu[D]_{l,\widehat{K}}^{\widetilde{\Lambda}_j}(h_j)\leq \sum_{j=1}^p \Eu[D]_{l,\widehat{K}}^{\widetilde{\Lambda}_j}(h)\leq \Eu[D]_{l,\widehat{K}}(h) ,
    \end{equation}
     Denoting $T_{k,j}^\omega$ the set of functions in $T^\omega$ whose support contained in 
     $\widetilde{\Lambda}_j$, we can therefore finally bound from above the left-hand side of~\eqref{eq:92} by
    $$
    c_k \sum_{j=1}^p \sup_{h \in T_{k,j}^\omega} \Bigg\{ \frac{2}{c_k} \sum_{x \in \Lambda_j} <\tau_x \psi,h>_{l,\widehat{K}} - \Eu[D]_{l,\widehat{K}}^{\widetilde{\Lambda}_j}(h) \Bigg\} 
    + \frac{C(\psi)l^2}{\sqrt{k}}.
    $$
    All the terms in the sum over $j$ are identically distributed, we use the quantity above to derive 
    \begin{align*}
        &\frac{1}{(2l+1)^2}\Big< (-\mathcal{L}_l^{-1}) \sum_{x \in B_{l_\varphi}} \tau_x \varphi,\, \sum_{x \in B_{l_\varphi}} \tau_x \varphi \Big>_{l,\widehat{K}_l} \\
        \leq &\frac{1}{(2k+1)^2}\sup_{h \in T_{k,j}^\omega} 
        \Bigg\{ \frac{2}{c_k} \sum_{x \in B_{k_\varphi}} <\tau_x \psi,h>_{l,\widehat{K}} - 
        \Eu[D]_{l,\widehat{K}}^{B_{k}}(h) \Bigg\} 
    + C(\psi)O(k^{-1/2})\\
    =& \frac{1}{(2k+1)^2} \Big< (-\mathcal{L}_k^{-1}) \sum_{x \in B_{k_\varphi}} \tau_x \varphi , \sum_{x \in B_{k_\varphi}} \tau_x \varphi \Big>_{l,\widehat{K}_l} +  C(\psi)O(k^{-1/2}).
    \end{align*}
    Now the support of $\sum_{x \in B_{k_\varphi}} \tau_x \varphi$ no longer grows with $l$, 
    by equivalence of ensembles in Proposition~\ref{prop:equivalence}, we get  
    \begin{equation*}
        \begin{aligned}
            \limsup_{l \to \infty}& \frac{1}{(2l+1)^2}
            \Big< (-\mathcal{L}_l^{-1}) \sum_{x \in B_{l_\varphi}} \tau_x \varphi
            ,\sum_{x \in B_{l_\varphi}} \tau_x \varphi \Big>_{l,\widehat{K}_l}\\
            \leq &\limsup_{k \to \infty} \frac{1}{(2k+1)^2} 
            \Big< (-\mathcal{L}_k^{-1}) \sum_{x \in B_{k_\varphi}} \tau_x \varphi ,\,
             \sum_{x \in B_{k_\varphi}} \tau_x \varphi \Big>_{\bah}. 
        \end{aligned}
        \end{equation*}
    By the variational formula for the variance again,  
    it suffices to show that 
    \begin{equation}\label{eq:93}
        \limsup_{k \to \infty} \frac{1}{(2k+1)^2}\sup_{h \in T_{k}^\omega} 
        \Bigg\{ 
            2  \Big<\sum_{x \in B_{k_\varphi}}\tau_x \varphi,h \Big>_{\bah} - 
            \Eu[D]_{\bah}^{B_{k}}(h) 
        \Bigg\} 
         \leq \iip[\varphi]_{\bah},
        \end{equation}
    According to Lemma~\ref{lemma:8.53}, there exists a constant $C(\varphi)$ such that 
    $$ 2  \Big<\sum_{x \in B_{k_\varphi}}\tau_x \varphi,h \Big>_{\bah} \leq C(\varphi)(2k+1)^2 + \frac{\Eu[D]_{\bah}^{B_{k}}(h)}{2}.$$ 
    For any $h$ such that $\Eu[D]_{\bah}^{B_{k}}(h) \geq 2C(\varphi)(2k+1)^2$, the right-hand side above is therefore negative, 
    and since it vanishes for $h = 0$, we can therefore safely assume that the supremum is taken w.r.t. functions $h \in T_k^\omega$ 
    satisfying $\Eu[D]_{\bah}^{B_k}(h) \leq 2C(\varphi)(2k+1)^2$. 
    The generalized integration by parts formula in Lemma~\ref{lemma:integration_by_parts} yields
    \begin{equation*}
        \big<\tau_x \varphi ,h \big>_{\bah}= \sum_{a \subset B_{s_\varphi}(x)} \big<I_a(\tau_x \varphi), \nabla_a h\big>_{\bah},
    \end{equation*}
    where $ I_a(\varphi) = (1/2) \nabla_a (-\mathcal{L}_{s_\varphi})^{-1} \varphi $. 
    For any oriented edge $ a $, we define 
    $$ B^\varphi(a)\coloneqq  \big\{ x \in \mathbb{Z}^2 : a \subset B_{s_\varphi}(x) \big\}, $$ 
    and define $ \widetilde{B}^\varphi(a) = B^\varphi(a) \cap B_{k_\varphi} $. 
    Moreover, for any oriented edge $ a \in B_{k_\varphi - s_\varphi} $, 
    $\widetilde{B}^\varphi(a)=B^\varphi(a)$. 
    Now summing over $x$, we have
    \begin{align*}
        &\sum_{x \in B_{k_\varphi}}\big<\tau_x \varphi ,h \big>_{\bah}\\
        =& \sum_{a \subset B_k} \sum_{x \in \widetilde{B}^\varphi(a)} \big<I_a(\tau_x \varphi), \nabla_a h\big>_{\bah}\\
        =& \sum_{a \subset B_k} \sum_{x \in B^\varphi(a)} \big<I_a(\tau_x \varphi), \nabla_a h\big>_{\bah} - 
        \sum_{\substack{a \subset B_k \setminus B_{k_\varphi - s_\varphi}\\ x \in B^\varphi(a) \setminus 
        \widetilde{B}^\varphi(a)}} \big<I_a(\tau_x \varphi), \nabla_a h\big>_{\bah}. 
    \end{align*}
    Holder's inequality gives for any $\gamma>0$, 
    \[ 
        \big<I_a(\tau_x \varphi), \nabla_a h\big>_{\bah}\leq 
        \frac{1}{2\gamma}\big<I_a(\tau_x \varphi)^2\big>_{\bah} 
        +\frac{\gamma}{2} \big< (\nabla_a h)^2 \big>_{\bah}. 
    \]
    Thanks to $|B_k \setminus B_{k_\varphi - s_\varphi}| \leq C(\varphi)k$, 
    and  the bound $\Eu[D]_{\bah}^{B_k}(h) \leq 2C(\varphi)(2k+1)^2$, 
    letting $\gamma = 1/\sqrt{k}$, it is then straightforward to obtain
\begin{equation*}
    \sum_{\substack{a \subset B_k \setminus B_{k_\varphi - s_\varphi}\\ x \in B^\varphi(a) \setminus \widetilde{B}^\varphi(a)}} \big<I_a(\tau_x \varphi), \nabla_a h\big>_{\bah} 
    \leq C(\varphi)k^{3/2}.
\end{equation*}
Now letting $\overline{I}_a(\varphi) = \sum_{x \in B^\varphi(a)} I_a(\tau_x \varphi)$, 
the left-hand-side of~\eqref{eq:93} is therefore less than
\begin{equation}\label{eq:94}
    \begin{aligned}
        \limsup_{k \to \infty} \frac{1}{(2k+1)^2}&\sup_{h \in T_{k}^\omega} 
            \Big\{ 
                2 \sum_{a \subset B_k} \big<\overline{I}_a(\varphi)  , \nabla_a h \big>_{\bah} - \Eu[D]_{\bah}^{B_{k}}(h) 
            \Big\} \\
            =& \limsup_{k \to \infty} \frac{1}{(2k+1)^2}
            \Big\{ 
                2 \sum_{a \subset B_k} \big<\overline{I}_a(\varphi)  , \nabla_a h_k \big>_{\bah} - \Eu[D]_{\bah}^{B_{k}}(h_k) 
            \Big\}, 
    \end{aligned}
\end{equation}
for some sequence $\{h_k\} \subset T_{k}^\omega$. 

Thanks to the translation invariance of $\mu_{\bah}$, 
and since $\tau_y \overline{I}_a(\varphi) = \overline{I}_{\tau_y a}(\varphi)$, 
writing the oriented bond  $a = (a_1, a_2)$, we have
$$\big<\overline{I}_a(\varphi)  , \nabla_a h_k \big>_{\bah}  
= \big<\overline{I}_{(0, a_2-a_1)}(\varphi), \nabla_{(0, a_2-a_1)} \tau_{-a_1}h_k \big>_{\bah}.$$
As seen before, a simple change of variable yields that 
$\big< \nabla_a f  , \nabla_a g \Big>_{\bah}  = \big< \nabla_{-a} f  , \nabla_{-a} g \big>_{\bah}$, 
from which we deduce that 
$$2 \sum_{a \subset B_k} \big<\overline{I}_a(\varphi)  , \nabla_a h_k \big>_{\bah} = 
4 \sum_{i=1,2} \Bigg<\overline{I}_{(0,e_i)}(\varphi)  , \nabla_{(0,e_i)} 
\Bigg(\sum_{\substack{x \in \mathbb{Z}^2 \\(x,x+e_i) \subset B_k}} \tau_{-x} h_k\Bigg) \Bigg>_{\bah}.
$$
Define
$$
\mf[u]_i^k = \frac{1}{(2k+1)^2} \nabla_{(0,e_i)} \sum_{\substack{x \in \mathbb{Z}^2 \\(x,x+e_i) \subset B_k}} \tau_{-x} h_k \in T^\omega.
$$
The elementary bound $(\sum_{i=1}^n a_i)^2 \leq n \sum_{i=1}^n a_i^2$ yields
$$
\sum_{i=1,2} \big< (\mf[u]_i^k)^2\big>_{\bah}  \leq \frac{2k(2k+1)}{(2k+1)^4} \sum_{\substack{x \in \mathbb{Z}^2 \\(x,x+e_i) \subset B_k}}
\big< (\nabla_{(x,x+e_i)} h_k)^2\big>_{\bah}
\leq \frac{1}{(2k+1)^2} \Eu[D]_{\bah}^{B_k}(h_k).
$$
with which  the right-hand-side of~\eqref{eq:94} is therefore less than
\[ 
\lim_{k\to \infty} \Big\{ 4\sum_{i=1,2} \big<\overline{I}_{(0,e_i)}(\varphi)  , \mf[u]_i^k \big>_{\bah} - 
\sum_{i=1,2} \big< (\mf[u]_i^k)^2\big>_{\bah}  \Big\}, 
\]
and since we already assumed that for some constant $C(\varphi)$, 
$\Eu[D]_{\bah}^{B_k}(h_k) \leq C(\varphi)(2k+1)^2$, 
the sequence of differential forms $(\mf[U]^k)_{k \in \mathbb{N}}$ is bounded in $L^2(\mu_{\bah})$. 
It is straightforward to check that any of its limit point $\mf[U] = (\mf[u]_1, \mf[u]_2)$ 
is the germ of a closed form in $\mathfrak{T}^\omega$.

Thus the right-hand-side of~\eqref{eq:94} is therefore less than
\[ 
    \sup_{\mf[U]\in \mathfrak{T}^\omega }\Bigg\{ 4\sum_{i=1,2} \Big<\overline{I}_{(0,e_i)}(\varphi)  , 
    \mf[u]_i \Big>_{\bah} - \sum_{i=1,2} \Big< (\mf[u]_i)^2\Big>_{\bah} \Bigg\}. 
\]
According to  Proposition~\ref{thm:1}, Lemma~\ref{lemma:integration_by_parts} and Remark~\ref{rk:8.55}, 
and using similar argument in the proof of Lemma~\ref{lemma:8.52}
the above supremum is equivalent to 
\[ 
   \begin{aligned}
     &\sup_{\substack{g\in \To\\u,v,m,n \in \mathbb{R}^2 } }
     \Bigg\{ 4\sum_{i=1,2} \Big<\overline{I}_{(0,e_i)}(\varphi)  , 
     (\mf[j]^{u,v,m,n}+\nabla_{0,e_i}\Sigma_g) \Big>_{\bah} - 
     \sum_{i=1,2} \Big< \big(\mf[j]^{u,v,m,n}+\nabla_{0,e_i}\Sigma_g\big)^2\Big>_{\bah} \Bigg\} \\
    =&\sup_{\substack{g\in \To\\u,v,m,n \in \mathbb{R}^2 } }
    \Bigg\{ 2 \Big<\varphi  , (\Sigma_g+W^{u,v,m,n}) \Big>_{\bah} - 
     \iip[\LL g + j^{u,v,m,n} ]_{\bah} \Bigg\}
    =\iip[\varphi]_{\hat{\alpha}},
   \end{aligned}
\]
where we denote  
\[ 
\begin{aligned}
    &W^{u,v,m,n} = \sum_{x\in \mathbb{Z}^2} u^\top x \, \eta_x^a + v^\top x \,\eta_x^p + m^\top x \,\eta_x^{a,\omega} + n^\top x \,\eta_x^{p,\omega},\\
    &j^{u,v,m,n}=u^\top j^a+v^\top j^p + m^\top j^{a,\omega }+ n^\top j^{p,\omega }.
\end{aligned}
\]

\end{proof}

\section{Inner products}\label{asec:inner}
In this section, we use Definition~\ref{def:iip} to calculate 
the related inner products that appear in Lemma~\ref{lemma:ortho} and~\ref{lemma:gradJ}. 
\begin{lemma}
    \begin{align*}
        &\iip[\grad^a, j^a ]_{\hat{\alpha}}=-\alpha^a(1-\alpha^a)I_2, 
           & &\iip[\grad^p, j^a ]_{\hat{\alpha}}=\alpha^a\alpha^pI_2,\\
        &\iip[\grad^a, j^p ]_{\hat{\alpha}}=\alpha^a\alpha^pI_2, 
           & &\iip[\grad^p, j^p ]_{\hat{\alpha}}=-\alpha^p(1-\alpha^p)I_2,\\
        &\iip[\grad^a, \jaoh ]_{\hat{\alpha}}=0, 
           & &\iip[\grad^p, \jaoh ]_{\hat{\alpha}}=0,\\
        &\iip[\grad^a, \jpoh ]_{\hat{\alpha}}=0. 
           & &\iip[\grad^p, \jpoh ]_{\hat{\alpha}}=0.\\
           &\quad\\
        &\iip[\grad^{a,\oha}, j^a ]_{\hat{\alpha}}=0,
        &&\iip[\grad^{p,\ohp}, j^a ]_{\hat{\alpha}}=0,\\
        &\iip[\grad^{a,\oha}, j^p ]_{\hat{\alpha}}=0,
        &&\iip[\grad^{p,\ohp}, j^p ]_{\hat{\alpha}}=0,\\
        &\iip[\grad^{a,\oha}, \jaoh ]_{\hat{\alpha}}=- \alpha^a V^a(\omega)I_2,
        &&\iip[\grad^{p,\ohp}, \jaoh ]_{\hat{\alpha}}=0,\\
        &\iip[\grad^{a,\oha}, \jpoh ]_{\hat{\alpha}}=0.
        &&\iip[\grad^{p,\ohp}, \jpoh ]_{\hat{\alpha}}=- \alpha^p V^p(\omega)I_2.\\
        &\quad\\
        &\iip[j^a, j^a ]_{\hat{\alpha}}=\alpha^a(1-\alpha)I_2,&&\iip[j^p, j^a ]_{\hat{\alpha}}=0,\\
        &\iip[j^a, j^p ]_{\hat{\alpha}}=0,&&\iip[j^p, j^p ]_{\hat{\alpha}}=\alpha^p(1-\alpha)I_2,\\
        &\iip[j^a, \jaoh ]_{\hat{\alpha}}=0,&&\iip[j^p, \jaoh ]_{\hat{\alpha}}=0,\\
        &\iip[j^a, \jpoh ]_{\hat{\alpha}}=0.&&\iip[j^p, \jpoh ]_{\hat{\alpha}}=0.\\
        &\quad\\
        &\iip[\jaoh, j^a ]_{\hat{\alpha}}=0,&&\iip[\jpoh, j^a ]_{\hat{\alpha}}=0,\\
        &\iip[\jaoh, j^p ]_{\hat{\alpha}}=0,&&\iip[\jpoh, j^p ]_{\hat{\alpha}}=0,\\
        &\iip[\jaoh, \jaoh ]_{\hat{\alpha}}=\alpha^a(1-\alpha) V^a(\omega)I_2,&&\iip[\jpoh, \jaoh ]_{\hat{\alpha}}=0,\\
        &\iip[\jaoh, \jpoh ]_{\hat{\alpha}}=0.&&\iip[\jpoh, \jpoh ]_{\hat{\alpha}}=\alpha^p(1-\alpha) V^p(\omega)I_2.
    \end{align*}
\end{lemma}
\begin{proof}Note that $(e_i)_k=\delta_{i,k}$
    \begin{align*}
        &\iip[\delta_i\eta^a, j^a_k ]_{\bah}=-\sum x_k \big< \eta^a_{e_i}-\eta^a_0 , \eta^a_x \big>_{\bah}=-\delta_{i,k} \big< \eta^a_{e_i}-\eta^a_0 , \eta^a_{e_i} \big>_{\bah}=-\delta_{i,k}\alpha^a(1-\alpha^a),\\
        &\iip[\delta_i\eta^a, j^p_k ]_{\bah}=-\sum x_k \big< \eta^a_{e_i}-\eta^a_0 , \eta^p_x \big>_{\bah}=-\delta_{i,k} \big< \eta^a_{e_i}-\eta^a_0 , \eta^p_{e_i} \big>_{\bah}=\delta_{i,k}\alpha^a\alpha^p,\\
        &\iip[\delta_i\eta^a, \jaoh_k ]_{\bah}=-\sum x_k \big< \eta^a_{e_i}-\eta^a_0 , \etaaob_{x} \big>_{\bah}=-\delta_{i,k} \big< \eta^a_{e_i}-\eta^a_0 , \etaaob_{e_i} \big>_{\bah}=0,\\
        &\iip[\delta_i\eta^a, \jpoh_k ]_{\bah}=-\sum x_k \big< \eta^a_{e_i}-\eta^a_0 , \etapob_{x} \big>_{\bah}=-\delta_{i,k} \big< \eta^a_{e_i}-\eta^a_0 , \etapob_{e_i} \big>_{\bah}=0.
    \end{align*}

    \begin{align*}
        \iip[\delta_i\etaaob, j^a_k ]_{\bah}&=-\sum x_k \big< \etaaob_{e_i}-\etaaob_0 , \eta^a_x \big>_{\bah}=-\delta_{i,k} \big< \etaaob_{e_i}-\etaaob_0 , \eta^a_{e_i} \big>_{\bah}\\
         &=0,\\
        \iip[\delta_i\etaaob, j^p_k ]_{\bah}&=-\sum x_k \big< \etaaob_{e_i}-\etaaob_0 , \eta^p_x \big>_{\bah}=-\delta_{i,k} \big< \etaaob_{e_i}-\etaaob_0 , \eta^p_{e_i} \big>_{\bah}\\
         &=0,\\
        \iip[\delta_i\etaaob, \jaoh_k ]_{\bah}&=-\sum x_k \big< \etaaob_{e_i}-\etaaob_0 , \etaaob_{x} \big>_{\bah}=-\delta_{i,k} \big< \etaaob_{e_i}-\etaaob_0 , \etaaob_{e_i} \big>_{\bah}\\
         &=-\delta_{i,k} \alpha^a V^a_{\bah}(\omega), \\
        \iip[\delta_i\etaaob, \jpoh_k ]_{\bah}&=-\sum x_k \big< \etaaob_{e_i}-\etaaob_0 , \etapob_{x} \big>_{\bah}=-\delta_{i,k} \big< \etaaob_{e_i}-\etaaob_0 , \etapob_{e_i} \big>_{\bah}\\
         &=0. 
    \end{align*}

    \begin{align*}
        \iip[j^a_i, j^a_k ]_{\bah}&=-\sum x_k \big< \eta^a_0(1-\eta_{e_i}) - \eta^a_{e_i}(1-\eta_0) , \eta^a_x \big>_{\bah}\\
        &= \delta_{i,k} \big< \eta^a_{e_i}(1-\eta_0), \eta^a_{e_i} \big>_{\bah}=\delta_{i,k}\alpha^a(1-\alpha)\\
        \iip[j^a_i, j^p_k ]_{\bah}&=-\sum x_k \big< \eta^a_0(1-\eta_{e_i}) - \eta^a_{e_i}(1-\eta_0) , \eta^p_x \big>_{\bah}\\
        &=\delta_{i,k} \big<  \eta^a_{e_i}(1-\eta_0) , \eta^p_{e_i} \big>_{\bah}=0\\
        \iip[j^a_i, \jaoh_k ]_{\bah}&=-\sum x_k \big<\eta^a_0(1-\eta_{e_i}) - \eta^a_{e_i}(1-\eta_0) , \etaaob_{x} \big>_{\bah}\\
        &=\delta_{i,k} \big<  \eta^a_{e_i}(1-\eta_0), \etaaob_{e_i} \big>_{\bah}=0\\
        \iip[j^a_i, \jpoh_k ]_{\bah}&=-\sum x_k \big< \eta^a_0(1-\eta_{e_i}) - \eta^a_{e_i}(1-\eta_0) , \etapob_{x} \big>_{\bah}\\
        &=\delta_{i,k} \big<  \eta^a_{e_i}(1-\eta_0), \etapob_{e_i} \big>_{\bah}=0
    \end{align*}
    and 
    \begin{align*}
        \iip[\jaoh_i, j^a_k ]_{\bah}&=-\sum x_k \Big< \etaaob_0(1-\eta_{e_i}) - \etaaob_{e_i}(1-\eta_0) , \eta^a_x \Big>_{\bah}\\
        & = \delta_{i,k} \Big< \etaaob_{e_i}(1-\eta_0), \eta^a_{e_i} \Big>_{\bah}=0\\
        \iip[\jaoh_i, j^p_k ]_{\bah}&=-\sum x_k \Big< \etaaob_0(1-\eta_{e_i}) - \etaaob_{e_i}(1-\eta_0) , \eta^p_x \Big>_{\bah}\\
        & =\delta_{i,k} \Big<  \etaaob_{e_i}(1-\eta_0) , \eta^p_{e_i} \Big>_{\bah}=0\\
        \iip[\jaoh_i, \jaoh_k ]_{\bah}&=-\sum x_k \Big<\etaaob_0(1-\eta_{e_i}) - \etaaob_{e_i}(1-\eta_0) , \etaaob_{x} \Big>_{\bah}\\
        &=\delta_{i,k} \Big<  \etaaob_{e_i}(1-\eta_0), \etaaob_{e_i} \Big>_{\bah}=\delta_{i,k} \alpha^a(1-\alpha)V^a_{\bah}(\omega)\\
        \iip[\jaoh_i, \jpoh_k ]_{\bah}&=-\sum x_k \Big< \etaaob_0(1-\eta_{e_i}) - \etaaob_{e_i}(1-\eta_0) , \etapob_{x} \Big>_{\bah}\\
        &=\delta_{i,k} \Big<  \etaaob_{e_i}(1-\eta_0), \etapob_{e_i} \Big>_{\bah}=0
    \end{align*}
    The others can be proved similarly.
\end{proof}

\end{document}